\documentclass[oneside,12pt,a4paper]{article}
\usepackage{amsmath}
\usepackage{amsthm,verbatim,bbm,amsfonts, yfonts, amssymb,newclude,nicefrac,enumerate,bm,mathabx}
\usepackage{geometry}
  \usepackage{paralist}
  \usepackage{graphics}
\usepackage[colorlinks=true]{hyperref}

\hypersetup{urlcolor=blue, citecolor=red}

\usepackage[shortlabels,inline]{enumitem}

\usepackage[capitalise]{cleveref}

\newtheorem{theorem}{Theorem}[section]
\newtheorem{lemma}[theorem]{Lemma}
\newtheorem{prop}[theorem]{Proposition}
\newtheorem{corollary}[theorem]{Corollary}

\theoremstyle{definition}
\newtheorem{definition}[theorem]{Definition}

\usepackage{mleftright}

\usepackage{mathtools}

\DeclarePairedDelimiter{\abs}{\lvert}{\rvert}
\DeclarePairedDelimiter{\abss}{\big\lvert}{\big\rvert}
\DeclarePairedDelimiter{\norm}{\lVert}{\rVert}

\newcommand{\E}{\mathbb{E}}

\newcommand{\R}{\mathbb{R}}
\newcommand{\N}{\mathbb{N}}

\newcommand{\V}{\mathbb{V}}

\renewcommand{\d}{{\mathrm d}}

\newcommand{\normm}[1]{ \big\lVert #1 \big\rVert }
\newcommand{\normmm}[1]{ \Big\lVert #1 \Big\rVert }

\newcommand{\Vnorm}[1]{{\left\vert\kern-0.25ex\left\vert\kern-0.25ex\left\vert #1 \right\vert\kern-0.25ex\right\vert\kern-0.25ex\right\vert}}

\newcommand{\indicator}[1]{\mathbbm{1}_{\smash{#1}}}

\title{
  \vspace{-1cm}
  On nonlinear Feynman-Kac formulas\\
  for viscosity solutions of semilinear\\
  parabolic partial differential equations\\
  with gradient-dependent
  nonlinearities}

\author{
Martin Hutzenthaler$^1$
and 
Katharina Pohl$^2$
\bigskip
\\
\small{$^1$ Faculty of Mathematics, University of Duisburg-Essen,}
\vspace{-0.1cm}\\
\small{Essen, Germany, e-mail: \texttt{martin.hutzenthaler}\textcircled{\texttt{a}}\texttt{uni-due.de}}
\smallskip
\\
\small{$^2$ Faculty of Mathematics, University of Duisburg-Essen,}
\vspace{-0.1cm}\\
\small{Essen, Germany, e-mail: \texttt{katharina.pohl}\textcircled{\texttt{a}}\texttt{uni-due.de}}
}

\begin{document}
\date{\today}
\maketitle

\begin{abstract}
\hskip -.2in
\noindent
The classical Feynman-Kac identity represents solutions of linear partial differential equations
in terms of stochastic differential euqations.
This representation has been generalized to nonlinear partial differential equations
on the one hand via backward stochastic differential equations
and on the other hand via stochastic fixed-point equations.
In this article we generalize the representation via stochastic fixed-point equations
to allow the nonlinearity in the semilinear partial differential equation
to depend also on the gradient of the solution.
\end{abstract}


\section{Introduction}
The classical Feynman-Kac identity
(see, e.g., \cite{GihmanSkorohod1972,HairerHutzenthalerJentzen2015,
	KaratzasShreve1991,
	PardouxRascanu2014}) 
is a representation of linear partial differential
equations (PDEs) in terms of stochastic
differential equations (SDEs).
This identity has various applications, e.g.,
the classical Monte Carlo method exploits this representation
and allows to approximate solutions of linear PDEs without suffering from the curse of
dimensionality.

There are different approaches for Feynman-Kac type formulas in the case of nonlinear PDEs.
One approach represents viscosity solutions (see, e.g., 
	\cite{CrandallEvansLions1984,CrandallIshiiLions1992,CrandallLions1983,HairerHutzenthalerJentzen2015,ImbertSilvestre2013})
  of nonlinear PDEs via solutions of backward stochastic differential
equations (BSDEs);
see, e.g., \cite{Bismut1976,PardouxPeng1990} for references on BSDEs and see, e.g., \cite{Antonelli1993,
	BarlesBuckdahnPardoux1997,
	HuPeng1995,
	HuYong2000,
	MaProtterYong1994,
	MaZhang2002,
	PardouxPeng1992,
	PardouxRascanu2014,
	PardouxTang1999,
	Peng1991,
  Zhang2017} for references on the connection between PDEs and BSDEs.
Another approach is via stochastic fixed-point equations (SFPEs) which arise
when the linear Feynman-Kac identity is applied to a semilinear PDE whose
nonlinear part is viewed as inhomogeneity; see, e.g., \cite{beck2021existence}.
A central motivation for this latter approach is that
full-history recursive multilevel Picard (MLP) approximation algorithms
(see \cite{beck2019overcoming,EHutzenthalerJentzenKruse2019,hutzenthaler2021multilevel,hutzenthaler2020overcoming}
for references on MLP approximation algorithms)
exploit the fact that viscosity solutions of semilinear PDEs are solutions of SFPEs.
MLP approximation algorithms are -- up to now -- the only methods
which have been mathematically proven to overcome the curse of dimensionality in the numerical approximation of solutions of semilinear Kolmogorov PDEs. 

In this article we generalize the results of \cite{beck2021existence}
from the case to gradient-independent nonlinearities to the gradient-dependent case
(e.g., the function $f$ in \eqref{eq:PDE.intro} depends on $\nabla_x u$).
To illustrate the findings of this article, we now present in Theorem \ref{thm:main_theorem} below
a special case of Theorem \ref{thm:vs} which is the main result of this article.

  \begin{theorem}\label{thm:main_theorem}
      Let $d \in \N$, 
      $\alpha, c, L, T \in (0, \infty)$,
      let $\langle\cdot,\cdot\rangle\colon
      \R^d\times\R^d\to\R$ be the standard
      Euclidean scalar product on $\R^d$,
      let $\norm{\cdot}\colon\R^d\to[0,\infty)$
      be the standard Euclidean norm on $\R^d$,
      let $\Vnorm{\cdot}\colon\R^{d+1}\to[0,\infty)$
      be the standard Euclidean norm on $\R^{d+1}$,
      let $\norm{\cdot}_F\colon \R^{d\times d}
      \to [0,\infty)$ be the Frobenius norm
      on $\R^{d\times d}$,
      let $(\Omega, \mathcal{F}, \mathbb{P}, (\mathbb{F}_s)_{s \in [0,T]})$ 
      be a filtered probability space
      satisfying the usual conditions,
      let $W \colon [0,T] \times \Omega \to \R^d$ 
      be a standard $(\mathbb{F}_s)_{s \in [0,T]}$-Brownian motion,
      let $\mu \in  C^{1}(\R^d, \R^d)$, 
      $\sigma \in C^{1} (\R^d, \R^{d \times d})$
      satisfy
      for all 
      $x, y \in \R^d$, $v\in\R^d$
      that
      \begin{equation}
      \label{eq:mu_sigma_loclip3}
      \max\Big\{\langle x-y,\mu(x)-\mu(y)\rangle, 
      \tfrac{1}{2}\norm{\sigma(x)-\sigma(y)}_F^2\Big\}
      \leq \tfrac{c}{2}\norm{x-y}^2,
      \end{equation}
      $\max \{
      \langle x, \mu(x)\rangle,
      \norm{\sigma(x)}_F^2
      \}
      \leq c (1+\norm{x}^2)$,
      and
      $v^* \sigma(x) (\sigma(x))^* v \geq \alpha \norm{v}^2$,
      assume for all 
      $j\in\{1,2,\ldots, d\}$ that
      $\frac{\partial \mu}{\partial x}$
      and $\frac{\partial \sigma}{\partial x_j}$
      are locally Lipschitz continuous,
      for every  $t\in [0,T]$,
      $x \in \R^d$ let
      $X^x_t = (X^{x}_{t,s})_{s \in [t,T]} \colon [t,T]
      \times \Omega \to \R^d$   
      be an 
      $(\mathbb{F}_s)_{s \in [t,T]}$-adapted 
      stochastic process with continuous
      sample paths satisfying that for all 
      $s \in [t,T]$ it holds a.s.\! that
      \begin{equation}
      X^x_{t,s} = x + \int_t^s \mu( X^x_{t,r}) \,\d r 
      + \int_t^s \sigma(X^x_{t,r}) \,\d W_r,
      \end{equation}
      assume for all
      $t\in [0,T]$,
      $\omega \in \Omega$
      that
      $\left([t,T] \times \R^d \ni (s,x) 
      \mapsto X^x_{t,s}(\omega) \in \R^d \right) \in C^{0,1}([t,T]$ 
      $\times \R^d, \R^d)$,
      for every $t\in[0,T]$, 
      $x \in \R^d$ let
      $Z^x_t = (Z^x_{t,s})_{s \in (t,T]} 
      \colon (t,T] \times \Omega \to \R^{d+1}$ 
      be an $(\mathbb{F}_s)_{s \in (t,T]}$-adapted
      stochastic process 
      with continuous sample paths
      satisfying that 
      for all $s \in (t,T]$ 
      it holds a.s.\! that
      \begin{equation}
      Z^x_{t,s} = 
      \begin{pmatrix}
      1\\
      \frac{1}{s-t} \int_t^s
      (\sigma(   X^x_{t,r}))^{-1} \; \Big(\frac{\partial}{\partial x} X^x_{t,r}\Big) \,\d W_r
      \end{pmatrix},     
      \end{equation} 
      let $f \in C([0,T] \times \R^d \times \R\times \R^{d}, \R)
      \cap L^2([0,T] \times \R^d \times \R\times \R^{d}, \R)$,
      $g \in C(\R^d, \R) \cap L^2(\R^d, \R)$
      be at most polynomially growing,
      and assume for all
      $t \in [0,T]$, $x_1,x_2 \in \R^d$, 
      $a_1,a_2\in\R$,
      $w_1,w_2 \in \R^{d}$  
      that
      $\lvert f(t,x_1,a_1,w_1) 
      - f(t,x_2,a_2,w_2) \rvert 
      \leq L\Vnorm{
          (a_1, w_1)
          -(a_2, w_2)}$.
      Then 
      \begin{enumerate}[label=(\roman*)]
          \item\label{it:vs1_cor2} 
          there exists a unique 
          $v\in C([0,T]\times \R^d,\R)
          \cap C^{0,1}([0,T)\times \R^d,\R)$
          which satisfies that
          $((v, \nabla_x v)(t,x)$
          $\cdot\sqrt{T-t})_{t\in [0,T), x\in\R^d}$
          grows at most polynomially 
          and for all 
          $t\in[0,T)$, $x\in \R^d$
          it holds that
          $\E[\abs{g(X^x_{t,T})} 
          \Vnorm{ Z^x_{t,T}}  
          + \int_t^T  \abs{f(r, X^x_{t,r}, 
              v(r, X^x_{t,r}),
              (\nabla_x v)(r,X^x_{t,r}))}
          \Vnorm{ Z^x_{t,r}}  \,\d r ]
          <\infty$
          and
          \begin{equation}\label{eq:v_cor2}
          \begin{split}
          &(v, \nabla_x v)(t,x) 
          =\E \left[  g(X^x_{t,T}) Z^x_{t,T}  
          + \int_t^T  f(r, X^x_{t,r}, v(r, X^x_{t,r}),
          (\nabla_x v)(r,X^x_{t,r})) 
          Z^x_{t,r}  \,\d r \right], 
          \end{split}
          \end{equation}
          \item\label{it:vs2_cor2}
          there exists a unique viscosity solution 
          $u\in\{\mathbf{u}\in C([0,T]\times \R^d,\R)
          \cap C^{0,1}([0,T)\times \R^d,\R)
          \colon$ $((\mathbf{u},\nabla_x\mathbf{u})(t,x)\sqrt{T-t})_{t\in [0,T), x\in\R^d}
          \text{ grows at most polynomially}\} $
          of
          \begin{multline}\label{eq:PDE.intro}
          (\tfrac{\partial u}{\partial t})(t,x)
          +\langle \mu(x), (\nabla_x u)(t,x)\rangle
          +\tfrac{1}{2}\operatorname{Tr}(\sigma(x)[\sigma(x)]^*(\operatorname{Hess}_x u)(t,x))
          \\
          +f(t,x,u(t,x), (\nabla_x u)(t,x))
          =0
          \end{multline}
          with $u(T,x)=g(x)$ for
          $(t,x)\in (0,T)\times \R^d$,
          and
          \item\label{it:vs3_cor2} 
          for all $t\in[0,T]$, $x\in \R^d$
          it holds that
          $u(t,x)=v(t,x)$.
      \end{enumerate}   
  \end{theorem}

\section{Existence and uniqueness results for viscosity solutions (VS) of Kolmogorov PDEs}
\label{sec:vs}

\subsection{Definitions}
\label{subsec:vs_def}

In this section we
recall the definitions 
of elliptic functions,
viscosity solutions,
and parabolic superjets.
The following definitions are from 
\cite[Definitions 2.4-2.7]{beck2021nonlinear}
and \cite[Definitions 2.11 and 2.13]{beck2021nonlinear}.

\begin{definition}[Degenerate elliptic functions]
    \label{def:deg_ell}
    Let $d\in\N$, $T\in(0,\infty)$,
    let $O\subseteq\R^d$ be a
    non-empty open set,
    and let $\langle\cdot,\cdot\rangle\colon
    \R^d\times\R^d\to\R$ be the standard
    Euklidean scalar product on $\R^d$.
    Then $G$ is degenerate elliptic
    on $(0,T)\times O\times\R\times\R^d
    \times\mathbb{S}_d$
    if
    \begin{enumerate}[label=(\roman*)]
        \item it holds that 
        $G\colon (0,T)\times O\times \R
        \times\R^d\times\mathbb{S}_d\to\R$
        is a function from 
        $(0,T)\times O\times\R\times\R^d
        \times\mathbb{S}_d$ to $\R$
        and
        \item it holds for all $t\in (0,T)$,
        $x\in O$, $r\in\R$, $p\in\R^d$,
        $A,B\in \mathbb{S}_d$ with
        $\forall y\in\R^d\colon
        \langle A y,y\rangle
        \leq \langle B y, y \rangle$
        that
        $G(t,x,r,p,A)\leq G(t,x,r,p,B)$.
    \end{enumerate}
\end{definition}

\begin{definition}[Viscosity subsolution]
    \label{def:vs_sub}
    Let $d\in\N$, $T\in(0,\infty)$,
    let $O\subseteq \R^d$ be a non-empty
    open set, and let 
    $G\colon (0,T)\times O\times\R
    \times\R^d\times\mathbb{S}_d\to\R$
    be degenerate elliptic. 
    Then $u$ is a viscosity solution of
    $(\frac{\partial}{\partial t}u)(t,x)
    +G(t,x,u(t,x), (\nabla_x u)(t,x),
    (\operatorname{Hess}_x u)(t,x))$
    $\geq 0$
    for $(t,x)\in(0,T)\times O$
    (we say that $u$ is a viscosity subsolution
    of $(\frac{\partial}{\partial t}u)(t,x)
    +G(t,x,u(t,x), (\nabla_x u)(t,x),
    (\operatorname{Hess}_x u)(t,x))
    =0 $)
    if and only if there exsits
    a set $A\subseteq \R\times\R^d$ 
    such that
    \begin{enumerate}[label=(\roman*)]
        \item it holds that 
        $(0,T)\times O\subseteq A$,
        \item it holds that 
        $u\colon A\to\R$ is upper
        semi-continuous, 
        and
        \item for all $t\in(0,T)$, $x\in O$,
        $\phi\in C^{1,2}((0,T)\times O,\R)$
        with $\phi(t,x)=u(t,x)$ and
        $\phi \geq u$ it holds that
        \begin{equation}
        (\tfrac{\partial}{\partial t}\phi)(t,x)
        +G(t,x,\phi(t,x), (\nabla_x \phi)(t,x),
        (\operatorname{Hess}_x \phi)(t,x))
        \geq 0.
        \end{equation}
    \end{enumerate}
\end{definition}

\begin{definition}[Viscosity supersolution]
    \label{def:vs_sub2}
    Let $d\in\N$, $T\in(0,\infty)$,
    let $O\subseteq \R^d$ be a non-empty
    open set, and let 
    $G\colon (0,T)\times O\times\R
    \times\R^d\times\mathbb{S}_d\to\R$
    be degenerate elliptic. 
    Then $u$ is a viscosity solution of
    $(\frac{\partial}{\partial t}u)(t,x)
    +G(t,x,u(t,x), (\nabla_x u)(t,x),
    (\operatorname{Hess}_x u)(t,x))$
    $\leq 0$
    for $(t,x)\in(0,T)\times O$
    (we say that $u$ is a viscosity supersolution
    of $(\frac{\partial}{\partial t}u)(t,x)
    +G(t,x,u(t,x), (\nabla_x u)(t,x),
    (\operatorname{Hess}_x u)(t,x))
    =0 $)
    if and only if there exsits
    a set $A\subseteq \R\times\R^d$ 
    such that
    \begin{enumerate}[label=(\roman*)]
        \item it holds that 
        $(0,T)\times O\subseteq A$,
        \item it holds that 
        $u\colon A\to\R$ is lower
        semi-continuous, 
        and
        \item for all $t\in(0,T)$, $x\in O$,
        $\phi\in C^{1,2}((0,T)\times O,\R)$
        with $\phi(t,x)=u(t,x)$ and
        $\phi \leq u$ it holds that
        \begin{equation}
        (\tfrac{\partial}{\partial t}\phi)(t,x)
        +G(t,x,\phi(t,x), (\nabla_x \phi)(t,x),
        (\operatorname{Hess}_x \phi)(t,x))
        \leq 0.
        \end{equation}
    \end{enumerate}
\end{definition}

\begin{definition}[Viscosity solution]
    \label{def:vs}
    Let $d\in\N$, $T\in(0,\infty)$,
    let $O\subseteq \R^d$ be a non-empty
    open set, and let 
    $G\colon (0,T)\times O\times\R
    \times\R^d\times\mathbb{S}_d\to\R$
    be degenerate elliptic. 
    Then $u$ is a viscosity solution of
    $(\frac{\partial}{\partial t}u)(t,x)
    +G(t,x,u(t,x), (\nabla_x u)(t,x),
    (\operatorname{Hess}_x u)(t,x))
    = 0$
    for $(t,x)\in(0,T)\times O$
    if and only if
    \begin{enumerate}[label = (\roman*)]
        \item it holds that $u$ 
        is a viscosity subsolution of
        \begin{equation}
        (\tfrac{\partial}{\partial t}u)(t,x)
        +G(t,x,u(t,x), (\nabla_x u)(t,x),
        (\operatorname{Hess}_x u)(t,x))
        = 0
        \end{equation}
        for $(t,x)\in(0,T)\times O$
        and
        \item it holds that $u$ 
        is a viscosity supersolution of
        \begin{equation}
        (\tfrac{\partial}{\partial t}u)(t,x)
        +G(t,x,u(t,x), (\nabla_x u)(t,x),
        (\operatorname{Hess}_x u)(t,x))
        = 0
        \end{equation}
        for $(t,x)\in(0,T)\times O$.
    \end{enumerate}
\end{definition}


\begin{definition}[Parabolic superjets]
    \label{def:par_superjets}
    Let $d\in \N$, $T\in (0,\infty)$,
    let $O\subseteq\R^d$ be a non-empty
    open set,
    let $t\in (0,T)$, $x\in O$,
    let $\langle\cdot, \cdot\rangle\colon
    \R^d\times\R^d\to\R$ be the 
    standard Euclidean scalar product
    on $\R^d$,
    let $\norm{\cdot}\colon\R^d\to[0,\infty)$ 
    be the standard Euclidean norm
    on $\R^d$,
    and let $u\colon(0,T)\times O\to\R$
    be a function. Then
    \begin{enumerate}[label=(\roman*)]
        \item we denote by $(\mathcal{P}^+ u)(t,x)$
        the set satisfying
        \begin{equation}
        \begin{split}
        &(\mathcal{P}^+ u)(t,x)
        =\Big\{(b,p,A)\in\R\times\R^d\times\mathbb{S}_d\colon\\
        &\limsup_{[(0,T)\times O]\setminus
            \{(t,x)\}\ni (s,y)\mapsto (t,x)}
        \Big[\tfrac{u(s,y)-u(t,x)-b(s-t)
            -\langle p,y-x\rangle
            -\frac{1}{2}\langle A(y-x),y-x  \rangle}
        {\abs{t-s}+\norm{x-y}^2} \Big]
        \leq 0
        \Big\},
        \end{split}
        \end{equation}
        and
        \item we denote by $(\mathfrak{P}^+ u)(t,x)$
        the set satisfying
        \begin{equation}
        \begin{split}
        &(\mathfrak{P}^+ u)(t,x)
        =\Big\{(b,p,A)\in\R\times\R^d\times\mathbb{S}_d\colon\\
        &\quad\Big(\exists (t_n,x_n,b_n,p_n,A_n)_{n\in\N}
        \subseteq (0,T)\times O\times\R
        \times\R^d\times\mathbb{S}_d\colon\\
        &\quad\big(\forall n\in\N\colon (b_n,p_n,A_n)
        \in (\mathcal{P}^+u)(t_n,x_n)\big)
        \text{ and }\\
        &\quad\lim\nolimits_{n\to\infty}
        (t_n,x_n,u(t_n,x_n),b_n,p_n,A_n)
        =(t,x,u(t,x),b,p,A)\Big)\Big\}.
        \end{split}
        \end{equation} 
    \end{enumerate}
\end{definition}

\subsection{Existence result for viscosity solutions of linear inhomogeneous Kolmogorov PDEs}
\label{subsec:vs_existence}

The following proposition
is a variation of \cite[Proposition 2.23]{beck2021nonlinear} where we replace $[0,T]$
by $[0,T]\setminus K_r$.

\begin{prop}\label{prop:vs23}
    Let $d,m\in\N$, $T\in (0,\infty)$,
    let $O\subseteq \R^d$ be a non-empty
    open set,
    let $\langle\cdot,\cdot\rangle\colon
    \R^d\times\R^d\to\R$ be the
    standard Euclidean scalar product
    on $\R^d$,
    let $\norm{\cdot }\colon\R^d\to[0,\infty)$
    be the standard Euclidean norm on $\R^d$,
    let $\norm{\cdot}_F\colon \R^{d\times m}
    \to [0,\infty)$ be the Frobenius norm
    on $\R^{d\times m}$,
    for every $r\in(0,\infty)$ let
    $K_r\subseteq[0,T)$, $O_r\subseteq O$
    satisfy $K_r=[0,\max\{T-\frac{1}{r},0\}]$
    and
    $O_r = \{  x \in O\colon \norm{x} \leq r \text{ and } \{ y \in \R^d\colon \norm{y-x} 
    < \frac{1}{r} \} \subseteq O \}$,
    let $g\in C(O,\R)$, 
    $h\in C([0,T]\times O,\R)$,
    $\mu\in C([0,T]\times O,\R^d)$,
    $\sigma\in C([0,T]\times O, \R^{d\times m})$,
    $V\in C^{1,2}([0,T]\times O, (0,\infty))$
    satisfy 
    for all $r\in(0,\infty)$ that
    \begin{equation}
    \sup\bigg(\bigg\{ \tfrac{\norm{\mu(t,x)-\mu(t,y)}+\norm{\sigma(t,x)-\sigma(t,y)}_F}{\norm{x-y}}\colon t\in [0,T], x,y\in O_r, x\neq y \bigg\}\cup \{0\} \bigg)
    <\infty,
    \end{equation}
    assume for all $t\in[0,T]$, $x\in O$ that
    \begin{equation}
    \label{eq:V_ass23}
    (\tfrac{\partial V}{\partial t})(t,x)
    +\langle \mu(t,x), (\nabla_x V)(t,x)  \rangle
    +\tfrac{1}{2} \operatorname{Tr}(\sigma(t,x)[\sigma(t,x)]^* (\operatorname{Hess}_x V)(t,x))    
    \leq 0,
    \end{equation}
    assume that $\sup_{r\in (0,\infty)}[\inf_{t\in[0,T)\setminus K_r}\inf_{x\in O\setminus O_r} V(t,x)]=\infty$ and
    $\inf_{r\in (0,\infty)}[\sup_{t\in [0,T)\setminus K_r}$
    $\sup_{x\in O\setminus O_r}$
    $(\frac{\abs{g(x)}}{V(T,x)}$ $+\frac{\abs{h(t,x)}}{V(t,x)}\sqrt{T-t})]
    =0$,
    let $(\Omega, \mathcal{F}, \mathbb{P},
    (\mathbb{F}_t)_{t\in[0,T]})$
    be a filtered probability space,
    let $W\colon [0,T]\times \Omega\to \R^m$
    be a standard $(\mathbb{F}_t)_{t\in [0,T]}$-Brownian motion,
    for every $t\in[0,T]$, $x\in O$ let
    $X^x_t=(X^x_{t,s})_{s\in[t,T]}\colon [t,T]\times \Omega \to O$ be an 
    $(\mathbb{F}_s)_{s\in[t,T]}$-adapted
    stochastic process with continuous
    sample paths satisfying that for all
    $s\in[t,T]$ it holds 
    a.s.\!
    that
    \begin{equation}
    X^x_{t,s}= x+\int_t^s \mu(r,X^x_{t,r}) 
    \,\d r 
    +\int_t^s \sigma(r,X^x_{t,r})\, \d W_r,
    \end{equation} 
    and let $u\colon[0,T]\times\R^d\to\R$
    satisfy for all $t\in [0,T]$, $x\in\R^d$
    that
    \begin{equation}
    \label{eq:u}
    u(t,x) = \E\Big[g(X^x_{t,T})
    +\int_t^T h(s,X^x_{t,s})\,\d s\Big].
    \end{equation}
    Then it holds that $u$ is a 
    viscosity solution of
    \begin{equation}
    \label{eq:u_vs1}
    \begin{split}
    &(\tfrac{\partial u}{\partial t})(t,x)
    +\langle \mu(t,x), (\nabla_x u)(t,x) \rangle
    \\
    &\qquad  
    +\tfrac{1}{2} \operatorname{Tr}(\sigma(t,x)[\sigma(t,x)]^* (\operatorname{Hess}_x u)(t,x))  
    +h(t,x)
    = 0
    \end{split}
    \end{equation}
    with $u(T,x)=g(x)$
    for $(t,x)\in(0,T)\times O$.
\end{prop}

\begin{proof}[Proof of Proposition~\ref{prop:vs23}]
    Throughout this proof let
    $\mathfrak{g}_n\in C(\R^d,\R)$, $n\in\N$,
    and $\mathfrak{h}_n\in C([0,T]\times \R^d,\R)$,
    $n\in\N$, be compactly supported
    functions which satisfy that
    $[\bigcup_{n\in\N}\operatorname{supp}(\mathfrak{g}_n)] \subseteq O$,
    $[\bigcup_{n\in\N}\operatorname{supp}(\mathfrak{h}_n)]\subseteq [0,T]\times O$,
    and
    \begin{equation}
    \label{eq:g_h_app_prop}
    \limsup_{n\to\infty} \bigg[
    \sup_{t\in[0,T)}\sup_{x\in O}\bigg(\frac{\abs{\mathfrak{g}_n(x)-g(x)}}{V(T,x)}+\frac{\abs{\mathfrak{h}_n(t,x)-h(t,x)}}{V(t,x)}\sqrt{T-t}\bigg)\bigg]
    =0
    \end{equation}
    (cf. \cite[Corollary 2.6]{HP2023})
    let $\mathfrak{m}_n\in C([0,T]\times\R^d,\R^d)$,
    $n\in\N$, and 
    $\mathfrak{s}_n\in C([0,T]\times\R^d,\R^{d\times m})$, 
    $n\in\N$,
    satisfy that
    \begin{enumerate}[label=(\Roman*)]
        \item
        \label{it:mu_sigma_app_1} 
        for all $n\in\N$ 
        it holds that
        \begin{equation}
        \sup_{t\in[0,T]}\sup_{\substack{x,y\in\R^d\\x\neq y}}\bigg[\frac{\norm{\mathfrak{m}_n(t,x)-\mathfrak{m}_n(t,y)}
            +\norm{\mathfrak{s}_n(t,x)-\mathfrak{s}_n(t,y)}_F}{\norm{x-y}}\bigg]
        <\infty,
        \end{equation}
        \item
        \label{it:mu_sigma_app_2}  
        for all $n\in\N$, 
        $t\in [0,T]$, $x\in O$
        it holds that
        \begin{equation}
        \indicator{\{V\leq n\}}(t,x)
        [\norm{\mathfrak{m}_n(t,x)-\mu(t,x)}
        +\norm{\mathfrak{s}_n(t,x)-\sigma(t,x)}_F]
        =0,
        \end{equation}
        and
        \item 
        \label{it:mu_sigma_app_3}
        for all $n\in\N$, $t\in [0,T]$,
        $x\in\R^d\setminus \{V\leq n+1\}$
        it holds that
        $\norm{\mathfrak{m}_n(t,x)}+\norm{\mathfrak{s}_n(t,x)}_F
        =0$
    \end{enumerate}
    (cf., e.g., the proof of 
    \cite[Lemma 3.7]{beck2021existence}),
    for every $n\in\N$, $t\in [0,T]$,
    $x\in\R^d$ let
    $\mathfrak{X}^{x,n}_t=(\mathfrak{X}^{x,n}_{t,s})_{s\in[t,T]}
    \colon [t,T]\times \Omega \to\R^d$
    be an $(\mathbb{F}_s)_{s\in [t,T]}$-adapted stochastic process with 
    continuous sample paths satisfying 
    that for all $s\in[t,T]$
    it holds 
    a.s.\! that
    \begin{equation}
    \label{eq:frak_X_n}
    \mathfrak{X}^{x,n}_{t,s} = x
    +\int_t^s \mathfrak{m}_n(r,\mathfrak{X}^{x,n}_{t,r}) \,\d r
    +\int_t^s \mathfrak{s}_n(r,\mathfrak{X}^{x,n}_{t,r})\,\d W_r
    \end{equation}
    (cf., e.g.,
    \cite[Theorem 5.2.9]{KaratzasShreve1991}),
    let $\mathfrak{u}^{n,k}\colon[0,T]\times\R^d\to\R$, 
    $n\in\N_0$, $k\in\N$,
    satisfy for all $n,k\in\N$,
    $t\in [0,T]$, $x\in\R^d$ that
    \begin{equation}
    \mathfrak{u}^{n,k}(t,x) = \E\bigg[\mathfrak{g}_k(\mathfrak{X}^{x,n}_{t,T})
    +\int_t^T \mathfrak{h}_k(s,\mathfrak{X}^{x,n}_{t,s})\,\d s\bigg]
    \end{equation}
    and
    \begin{equation}
    \mathfrak{u}^{0,k}(t,x) 
    = \E\bigg[\mathfrak{g}_k(X^{x}_{t,T})
    +\int_t^T \mathfrak{h}_k(s,X^{x}_{t,s})\,\d s\bigg],
    \end{equation}
    and for every $n\in\N$, $t\in [0,T]$,
    $x\in O$ let 
    $\tau^{x,n}_t\colon\Omega\to [t,T]$
    satisfy 
    \begin{equation}
    \tau^{x,n}_t
    =\inf(\{s\in[t,T]\colon \max\{V(s,\mathfrak{X}^{x,n}_{t,s}), V(s,X^x_{t,s})\}\geq n\}\cup\{T\}).
    \end{equation}
    Next note that \cite[Lemma 2.22]{beck2021nonlinear}
    (applied for every $n,k\in\N$ with 
    $\mu\curvearrowleft \mathfrak{m}_n$,
    $\sigma\curvearrowleft \mathfrak{s}_n$,
    $g\curvearrowleft \mathfrak{g}_k$,
    $h\curvearrowleft \mathfrak{h}_k$
    in the notation of 
    \cite[Lemma 2.22]{beck2021nonlinear}),
    item~\ref{it:mu_sigma_app_1},
    and the fact that for all $n\in\N$
    it holds that
    $\mathfrak{m}_n$ and $\mathfrak{s}_n$
    have compact support
    demonstrate that 
    for all $n,k\in\N$ it holds that
    $\mathfrak{u}^{n,k}$ is a viscosity solution of
    \begin{equation}
    \label{eq:u_n_vs1}
    \begin{split}
    &(\tfrac{\partial }{\partial t}\mathfrak{u}^{n,k})(t,x)
    +\langle \mathfrak{m}_n(t,x), 
    (\nabla_x \mathfrak{u}^{n,k})(t,x)\rangle
    \\
    &\qquad
    +\tfrac{1}{2}\operatorname{Tr}(\mathfrak{s}_n(t,x)[\mathfrak{s}_n(t,x)]^*(\operatorname{Hess}_x \mathfrak{u}^{n,k})(t,x))
    +\mathfrak{h}_k(t,x)=0
    \end{split}
    \end{equation}
    for $(t,x)\in(0,T)\times \R^d$.
    Furthermore, note that items 
    \ref{it:mu_sigma_app_1}-\ref{it:mu_sigma_app_3}
    and \eqref{eq:frak_X_n}
    ensure that for all $n\in\N$,
    $t\in [0,T]$, $x\in O$
    it holds that
    \begin{equation}
    \mathbb{P}(\forall s\in [t,T]\colon
    \indicator{\{s\leq \tau_t^{x,n}\}}
    \mathfrak{X}^{x,n}_{t,s}
    =\indicator{\{s\leq \tau_t^{x,n}\}}
    X^x_{t,s})=1.
    \end{equation}
    Hence, we obtain for all $n,k\in\N$,
    $t\in[0,T]$, $x\in O$ that
    \begin{equation}\label{eq:g_k_X^n_conv}
    \begin{split}
    &\E\Big[\abs{\mathfrak{g}_k(\mathfrak{X}^{x,n}_{t,T})
        -\mathfrak{g}_k(X^x_{t,T})}\Big]
    =\E\Big[\indicator{\{\tau^{x,n}_t<T\}}\abs{\mathfrak{g}_k(\mathfrak{X}^{x,n}_{t,T})
        -\mathfrak{g}_k(X^x_{t,T})}\Big]\\
    &\leq 2\Big[\sup\nolimits_{y\in O}
    \abs{\mathfrak{g}_k(y)}\Big] \mathbb{P}(\tau^{x,n}_t<T)
    \end{split}
    \end{equation}     
    and
    \begin{equation}
    \begin{split}\label{eq:h_k_X^n_conv}
    &\int_t^T \E\Big[\abs{\mathfrak{h}_k(s,\mathfrak{X}^{x,n}_{t,s})
        -\mathfrak{h}_k(s,X^x_{t,s})}\Big]\,\d s\\
    &= \int_t^T \E\Big[\indicator{\{\tau^{x,n}_t<T\}}
    \abs{\mathfrak{h}_k(s,\mathfrak{X}^{x,n}_{t,s})
        -\mathfrak{h}_k(s,X^x_{t,s})}\Big]\,\d s\\
    &\leq 2T \Big[\sup\nolimits_{s\in [0,T]}
    \sup\nolimits_{y\in O} \abs{\mathfrak{h}_k(s,y)}\Big]
    \mathbb{P}(\tau^{x,n}_t<T).
    \end{split}
    \end{equation}
    In addition, observe that
    \cite[Lemma 3.1]{beck2021existence} and 
    \eqref{eq:V_ass23}
    ensure that for all $n\in\N$,
    $t\in [0,T]$, $x\in O$
    it holds that
    \begin{equation}
    \E[V(\tau^{x,n}_t, X^{x}_{t,\tau^{x,n}_t})]
    \leq V(t,x).
    \end{equation}    
    Markov's  inequality,
    \eqref{eq:g_k_X^n_conv},
    and \eqref{eq:h_k_X^n_conv}
    therefore imply that for 
    all $n,k\in\N$, $t\in [0,T]$,
    $x\in O$ it holds that
    \begin{equation}
    \begin{split}
    &\abs{\mathfrak{u}^{n,k}(t,x)-\mathfrak{u}^{0,k}(t,x)}
    \leq 2\bigg[\sup_{y\in O} \abs{\mathfrak{g}_k(y)}
    +T \sup_{s\in [0,T]}\sup_{y\in O}
    \abs{\mathfrak{h}_k(s,y)} \bigg] 
    \mathbb{P}(\tau^{x,n}_t<T)\\
    &\leq 2\bigg[\sup_{y\in O} \abs{\mathfrak{g}_k(y)}
    +T \sup_{s\in [0,T]}\sup_{y\in O}
    \abs{\mathfrak{h}_k(s,y)} \bigg] 
    \mathbb{P}(V(\tau^{x,n}_t,X^x_{t,\tau^{x,n}_t})\geq n)\\
    &\leq \frac{2}{n}\bigg[\sup_{y\in O} \abs{\mathfrak{g}_k(y)}
    +T \sup_{s\in [0,T]}\sup_{y\in O}
    \abs{\mathfrak{h}_k(s,y)} \bigg] 
    \E[V(\tau^{x,n}_t,X^x_{t,\tau^{x,n}_t})]\\
    &\leq \frac{2}{n}\bigg[\sup_{y\in O} \abs{\mathfrak{g}_k(y)}
    +T \sup_{s\in [0,T]}\sup_{y\in O}
    \abs{\mathfrak{h}_k(s,y)} \bigg] 
    V(t,x).
    \end{split}
    \end{equation}
    This 
    shows that for all $k\in\N$
    and all compact
    $\mathcal{K}\subseteq (0,T)\times O$
    it holds that
    \begin{equation}
    \label{eq:u_n_conv1}
    \limsup_{n\to\infty} \bigg[
    \sup_{(t,x)\,\in\mathcal{K}}
    \abs{\mathfrak{u}^{n,k}(t,x)
        -\mathfrak{u}^{0,k}(t,x)} \bigg]
    =0. 
    \end{equation}    
    Moreover, observe that 
    item~\ref{it:mu_sigma_app_2}
    and the assumption
    that $\sup_{r\in(0,\infty)}[\inf_{t\in[0,T)\setminus K_r}$ 
    $\inf_{x\in\R^d\setminus O_r}$
    $ V(t,x)]=\infty$ 
    imply that for all compact
    $\mathcal{K}\subseteq [0,T]\times O$
    it holds that
    \begin{equation}
    \limsup_{n\to\infty} \bigg[ \sup_{(t,x)\in\mathcal{K}} \Big(\norm{\mathfrak{m}_n(t,x)-\mu(t,x)}
    +\norm{\mathfrak{s}_n(t,x)-\sigma(t,x)}\Big) \bigg]
    =0.
    \end{equation}
    Combining
    \cite[Corollary 2.20]{beck2021nonlinear},
    \eqref{eq:u_n_vs1}, and
    \eqref{eq:u_n_conv1}
    hence
    demonstrates that
    for all $k\in\N$ it holds that
    $\mathfrak{u}^{0,k}$ is a viscosity solution of
    \begin{equation}
    \label{eq:u_k_vs}
    \begin{split}
    &(\tfrac{\partial }{\partial t}\mathfrak{u}^{0,k})(t,x)
    +\langle  \mu(t,x),
    (\nabla_x \mathfrak{u}^{0,k})(t,x) \rangle
    \\
    &\qquad
    +\tfrac{1}{2}\operatorname{Tr}(\sigma(t,x)[\sigma(t,x)]^*(\operatorname{Hess}_x \mathfrak{u}^{0,k})(t,x))
    +\mathfrak{h}_k(t,x)=0
    \end{split}
    \end{equation}
    for $(t,x)\in (0,T)\times O$.
    Next observe that
    the fact that for all 
    $t\in [0,T]$, $s\in[t,T]$, 
    $x\in O$
    it holds that
    $\E[V(s, X^{x}_{t,s})]\leq V(t,x)$
    demonstrates that 
    for all $k\in\N$,
    $t\in (0,T)$, $x\in O$
    it holds that
    \begin{equation}
    \begin{split}
    &\abs{\mathfrak{u}^{0,k}(t,x)-u(t,x)}\\
    &= \Big\lvert\E[\mathfrak{g}_k(X^x_{t,T})-g(X^x_{t,T})]+\int_t^T\E[ \mathfrak{h}_k(s, X^x_{t,s})-h(s,X^x_{t,s})]\,\d s\Big\rvert\\
    &\leq \E\bigg[\frac{\abs{\mathfrak{g}_k(X^x_{t,T})-g(X^x_{t,T})}V(T,X^x_{t,T})}{V(T,X^x_{t,T})}\bigg]\\
    &\qquad +\int_t^T\E\bigg[ \frac{\abs{\mathfrak{h}_k(s, X^x_{t,s})-h(s,X^x_{t,s})}V(s,X^x_{t,s})\sqrt{T-s}}{V(s,X^x_{t,s})\sqrt{T-s}}\bigg]\,\d s\\
    &\leq \bigg[\sup_{y\in O}
    \frac{\abs{\mathfrak{g}_k(y)-g(y)}}{V(T,y)}\bigg]\E\big[V(T,X^x_{t,T})
    \big]\\
    &\qquad 
    +\bigg[\sup_{r\in [0,T)}
    \sup_{y\in O}\frac{\abs{\mathfrak{h}_k(r,y)-h(r,y)}}{V(r,y)}\sqrt{T-r}\bigg]
    \int_t^T\E\bigg[
    \frac{V(s,X^x_{t,s})}{\sqrt{T-s}}\bigg]\,\d s\\
    &\leq \bigg[\sup_{y\in O}
    \frac{\abs{\mathfrak{g}_k(y)-g(y)}}{V(T,y)}\bigg]V(T,x)\\
    &\qquad
    +\bigg[\sup_{r\in [0,T)}
    \sup_{y\in O}\frac{\abs{\mathfrak{h}_k(r,y)-h(r,y)}}{V(r,y)}\sqrt{T-r}\bigg]
    \int_t^T
    \frac{V(t,x)}{\sqrt{T-s}}\,\d s\\
    &\leq \bigg[\sup_{y\in O}
    \frac{\abs{\mathfrak{g}_k(y)-g(y)}}{V(T,y)}\bigg]V(T,x)\\
    &\qquad
    +\bigg[\sup_{r\in [0,T)}
    \sup_{y\in O}\frac{\abs{\mathfrak{h}_k(r,y)-h(r,y)}}{V(r,y)}\sqrt{T-r}\bigg]
    2\sqrt{T}\, V(t,x).
    \end{split}
    \end{equation}    
    Combining this with
    \eqref{eq:g_h_app_prop}  
    shows that for all compact
    $\mathcal{K}\subseteq (0,T)\times O$
    it holds that
    \begin{equation}
    \begin{split}
    \limsup_{k\to\infty} \bigg[
    \sup_{(t,x)\,\in\mathcal{K}}
    \abs{\mathfrak{u}^{0,k}(t,x)-u(t,x)} \bigg]
    =0.
    \end{split}
    \end{equation}
    This,
    \cite[Corollary 2.20]{beck2021nonlinear},
    \eqref{eq:g_h_app_prop},
    and
    \eqref{eq:u_k_vs}
    imply that $u$
    is a viscosity solution of
    \begin{equation}
    \label{eq:u_vs2}
    \begin{split}
    &(\tfrac{\partial u}{\partial t})(t,x)
    +\langle \mu(t,x), (\nabla_x u)(t,x) \rangle\\
    &\qquad 
    +\tfrac{1}{2} \operatorname{Tr}(\sigma(t,x)[\sigma(t,x)]^* (\operatorname{Hess}_x u)(t,x))
    +h(t,x)
    = 0
    \end{split}
    \end{equation}
    for $(t,x)\in (0,T)\times O$.
    In addition, observe that
    \eqref{eq:u} ensures that
    for all $x\in\R^d$
    it holds that $u(T,x)=g(x)$.
    This and 
    \eqref{eq:u_vs2}
    establish \eqref{eq:u_vs1}.
    The proof of Proposition~\ref{prop:vs23}
    is thus complete.
\end{proof}

\subsection{Uniqueness results for viscosity solutions of semilinear Kolmogorov PDEs}
\label{subsec:vs_uniqueness}

The following lemma
is an extension of \cite[Lemma 3.3]{beck2021nonlinear}
where we consider $(0,T)\setminus K_r$ instead of $(0,T)$ in \eqref{eq:u_i_ass}.

\begin{lemma}\label{lem:vs3_3}
    Let $d,k\in\N$, $T\in(0,\infty)$,
    let $\langle\cdot,\cdot\rangle\colon
    \R^d\times\R^d\to\R$ be the standard
    Euclidean scalar product on $\R^d$,
    let $\norm{\cdot}\colon(\bigcup_{m\in\N}\R^m)\to[0,\infty)$
    satisfy for all $m\in\N$, 
    $x=(x_1,x_2,\ldots, x_m)\in\R^m$
    that $\norm{x}=(\sum_{i=1}^m \abs{x_i}^2)^{1/2}$,
    let $O\subseteq \R^d$ be a non-empty 
    open set, 
    for every $r\in(0,\infty)$ let 
    $K_r\subseteq[0,T)$, $O_r\subseteq O$
    satisfy $K_r=[0,\max\{T-\frac{1}{r},0\}]$
    and
    $O_r = \{  x \in O\colon \norm{x} \leq r \text{ and } \{ y \in \R^d\colon \norm{y-x} 
    < \frac{1}{r} \} \subseteq O \}$,
    let $G_i\in C((0,T)\times O\times\R\times\R^d\times\mathbb{S}_d,\R)$, $i\in\{1,2,\ldots,k\}$,
    satisfy for all $i\in\{1,2,\ldots,k\}$
    that $G_i$ is degenerate elliptic
    and upper semi-continuous,
    let $u_i\colon[0,T]\times O\to\R$,
    $i\in\{1,2,\ldots,k\}$, satisfy for all
    $i\in\{1,2,\ldots,k\}$ that 
    $u_i$ is a viscosity solution of
    \begin{equation}
    \label{eq:G_i_assumption}
    (\tfrac{\partial u_i}{\partial t})(t,x)
    +G_i(t,x,u_i(t,x),(\nabla_x u_i)(t,x),
    (\operatorname{Hess}_x u_i)(t,x))
    \geq 0
    \end{equation}   
    for $(t,x)\in(0,T)\times O$,
    assume that 
    \begin{equation}\label{eq:u_i_ass}
    \begin{split}
    &\sup_{x\in O} \bigg[\sum_{i=1}^k
    u_i(T,x) \bigg] \leq 0
    \\
    &\text{and}\qquad
    \lim_{n\to\infty} \Bigg[
    \sup_{t\in(0,T)\setminus K_n}
    \sup_{x\in O\setminus O_n}\bigg[
    \bigg( \sum_{i=1}^k u_i(t,x) \bigg)
    \sqrt{T-t}\bigg]\Bigg]
    \leq 0,
    \end{split}
    \end{equation}
    and assume for all 
    $t^{(n)}\in(0,T)$,
    $n\in\N_0$, 
    and all $(x^{(n)}_i,r^{(n)}_i,A^{(n)}_i)
    \in O\times\R\times\mathbb{S}_d$,
    $n\in\N_0$, $i\in \{1,2,\ldots,k\}$,
    with
    $\limsup_{n\to\infty}[\abs{t^{(n)}-t^{(0)}}
    +\norm{x^{(n)}_1-x^{(0)}_1}]
    +\sqrt{n}\sum_{i=2}^k\norm{x^{(n)}_i-x^{(n)}_{i-1}}
    =0<\liminf_{n\to\infty}[\sum_{i=1}^k r^{(n)}_i]
    =\limsup_{n\to\infty}[\sum_{i=1}^k r^{(n)}_i]
    \leq \sup_{n\in\N}[\sum_{i=1}^k \abs{r^{(n)}_i}]
    <\infty$
    and $\forall(n\in\N,z_1,z_2,\ldots, z_k\in\R^d)\colon
    -5\sum_{i=1}^k \norm{z_i}^2
    \leq \sum_{i=1}^k \langle z_i, A^{(n)}_i z_i\rangle
    \leq 5\sum_{i=2}^k \norm{z_i-z_{i-1}}^2$ 
    that
    \begin{equation}\label{eq:G_convergence}
    \begin{split}
    \limsup_{n\to\infty}\bigg[
    &\sum_{i=1}^k G_i\big(t^{(n)},
    x^{(n)}_i,r^{(n)}_i,
    n(\indicator{[2,k]}(i)[x^{(n)}_i-x^{(n)}_{i-1}]\\
    &\qquad
    +\indicator{[1,k-1]}(i)[x^{(n)}_i-x^{(n)}_{i+1}]),
    nA^{(n)}_i\big)\bigg]
    \leq 0.
    \end{split}
    \end{equation}
    Then it holds for all $t\in (0,T]$, $x\in O$
    that $\sum_{i=1}^k u_i(t,x)\leq 0$.
\end{lemma}

\begin{proof}[Proof of Lemma~\ref{lem:vs3_3}]
    The goal of this proof is to show
    that for all $t\in (0,T]$, $x\in O$
    it holds that $\sum_{i=1}^k u_i(t,x)\leq 0$
    by demonstrating that for all $\delta \in (0,\infty)$,
    $t\in (0,T]$, $x\in O$
    it holds that
    $\sum_{i=1}^k u_i(t,x)\leq \frac{k\delta}{t}$.
    Throughout this proof let 
    $\delta\in (0,\infty)$, 
    let $v_i\colon [0,T]\times O\to [-\infty,\infty)$,
    $i\in\{1,2,\ldots,k\}$, satisfy
    for all $i\in\{1,2,\ldots,k\}$,
    $t\in[0,T]$, $x\in O$ that
    \begin{equation}
    v_i(t,x)=
    \begin{cases}
    u_i(t,x)-\frac{\delta}{t} &\colon t>0 \\
    -\infty &\colon t=0,
    \end{cases}
    \end{equation}
    let $H_i\colon (0,T)\times O\times\R
    \times\R^d\times\mathbb{S}_d\to\R$,
    $i\in\{1,2,\ldots,k\}$,
    satisfy for all $i\in\{1,2,\ldots,k\}$,
    $t\in(0,T)$, $x\in O$, $r\in\R$,
    $p\in\R^d$, $A\in\mathbb{S}_d$
    that
    \begin{equation}\label{eq:H_i}
    H_i(t,x,r,p,A)=G_i(t,x,r+\tfrac{\delta}{t},
    p,A)-\tfrac{\delta}{t^2}, 
    \end{equation}
    let $\Phi\colon[0,T]\times (\R^d)^k\to [0,\infty)$
    and $\eta\colon [0,T]\times (\R^d)^k\to[-\infty,\infty)$
    satisfy for all $t\in [0,T]$,
    $x=(x_1,x_2,\ldots,x_k)\in (\R^d)^k$ that
    $\eta(t,x)=\sum_{i=1}^k v_i(t,x_i)$ and
    \begin{equation}\label{eq:Phi}
    \Phi(t,x)= \frac{1}{2}\bigg[
    \sum_{i=2}^k \norm{x_i-x_{i-1}}^2\bigg],
    \end{equation}
    let $S\in (-\infty,\infty]$ satisfy
    $S=\sup_{t\in [0,T]}\sup_{x\in O}
    [\sum_{i=1}^k v_i(t,x)]$,
    let $S_{\alpha,r}\in (-\infty,\infty]$, 
    $\alpha,r \in [0,\infty)$,
    satisfy for all $\alpha, r\in [0,\infty)$
    that
    \begin{equation}
    S_{\alpha,r} = \sup_{t\in K_r}
    \sup_{x\in (O_r)^k}
    [\eta(t,x)-\alpha\Phi(t,x)],
    \end{equation}
    and let $\Vnorm{\cdot}
    \colon\R^{(kd)\times (kd)}\to[0,\infty)$
    satisfy for all $A\in \R^{(kd)\times (kd)}$
    that
    \begin{equation}
    \begin{split}
    \Vnorm{A}
    = \sup\Bigg\{
    &\bigg[
    \sum_{i=1}^{kd}\abs{y_i}^2\bigg]^{\frac{1}{2}}
    \bigg[\sum_{i=1}^{kd}\abs{x_i}^2
    \bigg]^{-\frac{1}{2}}
    \colon
    \begin{pmatrix}
    x=(x_1,x_2,\ldots,x_{kd})\in\R^{kd}\setminus
    \{0\},\\
    y=(y_1,y_2,\ldots,y_{kd})\in\R^{kd},\\
    y=Ax
    \end{pmatrix}     
    \Bigg\}.
    \end{split}
    \end{equation}    
    First observe that 
    \eqref{eq:u_i_ass}, the fact that 
    $\sup_{x\in O}[\sum_{i=1}^kv_i(0,x)]
    =-\infty$, and the fact that 
    for all $i\in\{1,2,\ldots,k\}$
    it holds that $v_i\leq u_i$
    show that
    \begin{equation}\label{eq:v_i_ass}
    \begin{split}
    &\sup_{x\in O}\bigg[ \sum_{i=1}^k
    v_i(T,x) \bigg] 
    \leq 0 \qquad
    \text{and}\\ 
    &
    \limsup_{n\to\infty} \Bigg[
    \sup_{t\in [0,T)\setminus K_n}
    \sup_{x\in O\setminus O_n}
    \bigg[
    \sum_{i=1}^k v_i(t,x) \sqrt{T-t}\bigg]\Bigg]
    \leq 0.
    \end{split}
    \end{equation}
    In addition, note that the 
    assumption that for all $i\in\{1,2,\ldots,k\}$
    it holds that $u_i$ is upper 
    semi-continuous ensures that for all
    $i\in\{1,2,\ldots,k\}$ it holds that
    $v_i$ is upper semi-continuous.
    Moreover, observe that 
    \eqref{eq:G_i_assumption} and
    \eqref{eq:H_i} imply that 
    for all $i\in\{1,2,\ldots,k\}$
    it holds that $v_i$ is a
    viscosity solution of
    \begin{equation}\label{eq:v_i_vs}
    (\tfrac{\partial v_i}{\partial t})(t,x)
    +H_i(t,x,v_i(t,x), (\nabla_x v_i)(t,x),
    (\operatorname{Hess}_x v_i)(t,x))\geq 0
    \end{equation}
    for $(t,x)\in(0,T)\times O$.
    Next we claim that for all
    $t\in [0,T]$, $x\in O$
    it holds that
    \begin{equation}\label{eq:S_smaller_zero}
    S= \sup_{t\in [0,T]}\sup_{x\in O}
    \bigg[\sum_{i=1}^k v_i(t,x)\bigg]
    \leq 0.
    \end{equation}
    We prove \eqref{eq:S_smaller_zero}
    by contradiction. 
    For this assume that $S\in (0,\infty]$.
    Note that the hypothesis that 
    $S\in (0,\infty]$ and 
    \eqref{eq:v_i_ass}
    imply that there exists $N\in\N$
    which satisfies that
    \begin{enumerate}[label=(\Roman*)]
        \item\label{it:O_N_nonempty}
        it holds that $K_N\neq\emptyset$
        and $O_N\neq \emptyset$,
        \item\label{it:O_N_compact} 
        it holds that $K_N$ and 
        $O_N$ are compact, and
        \item\label{it:O_N_S} 
        it holds that
        $\sup_{t\in K_N}\sup_{x\in O_N}
        [\sum_{i=1}^k v_i(t,x)]=S$.
    \end{enumerate} 
    The fact that for all 
    $i\in\{1,2,\ldots,k\}$ 
    it holds that $v_i$ is upper semi-continuous
    therefore shows that $S\in (0,\infty)$.
    Moreover, observe that 
    item~\ref{it:O_N_S} and
    the fact that for all
    $i\in\{1,2,\ldots,k\}$ 
    it holds that 
    $\sup_{x\in O} v_i(0,x)=-\infty$
    ensure that 
    $S= \sup_{t\in K_N\cap(0,T]}\sup_{x\in O_N}$
    $[\sum_{i=1}^k v_i(t,x)]$.
    Next note that
    the fact that 
    $\Phi\in C([0,T]\times (\R^d)^k, \R)$
    and the fact that for all
    $i\in\{1,2,\ldots,k\}$ 
    it holds that $v_i$ is
    upper semi-continuous
    demonstrate that 
    for all $\alpha\in(0,\infty)$
    it holds that
    $K_N \times (O_N)^{k} \ni (t,x)
    \mapsto \eta(t,x)-\alpha\Phi(t,x)
    \in[-\infty,\infty)$
    is upper semi-continuous.
    Item~\ref{it:O_N_compact}
    hence proves that there exists
    $t^{(\alpha)}\in K_N$,
    $\alpha\in (0,\infty)$, and
    $x^{(\alpha)}=(x^{(\alpha)}_1,x^{(\alpha)}_2,
    \ldots,x^{(\alpha)}_k)\in (O_N)^k$,
    $\alpha\in (0,\infty)$,
    which satisfy for all 
    $\alpha \in (0,\infty)$ that
    \begin{equation}
    \label{eq:eta_alpha_phi_globmax}
    \eta(t^{(\alpha)},x^{(\alpha)})
    -\alpha\Phi(t^{(\alpha)},x^{(\alpha)})
    =\sup_{t\in K_N}\sup_{x\in (O_N)^k}
    [\eta(t,x)-\alpha\Phi(t,x)]
    =S_{\alpha,N}.
    \end{equation}
    Furthermore, observe that
    the fact that for all $t\in [0,T]$,
    $y\in O$ it holds that
    $\eta(t,y,y,\ldots,y)$
    $= \sum_{i=1}^k v_i(t,y)$
    and the fact that for all $t\in [0,T]$,
    $y\in O$ it holds that
    $\Phi(t,y,y,\ldots,y)=0$
    show that
    for all $\alpha\in (0,\infty)$
    it holds that
    \begin{equation}\label{eq:S_alpha_N}
    \begin{split}
    &S_{\alpha,N}
    \geq \sup_{t\in [0,T]}\sup_{y\in O}
    [\eta(t,y,y,\ldots,y)-\alpha
    \Phi(t,y,y,\ldots,y)]\\
    &=\sup_{t\in [0,T]}\sup_{y\in O}
    \bigg[\sum_{i=1}^k v_i(t,y)\bigg]
    =S > 0.
    \end{split}
    \end{equation}
    This and
    the fact that for all 
    $\alpha,\beta\in (0,\infty)$ 
    with $\alpha\geq \beta$
    it holds that
    $S_{\alpha,N}\leq S_{\beta, N}$
    ensure that
    $\liminf_{\alpha\to \infty} S_{\alpha,N}
    =\limsup_{\alpha\to\infty} S_{\alpha,N}
    \in [S,\infty)\subseteq \R$.
    Next observe that
    \eqref{eq:S_alpha_N}
    and the fact that for all 
    $\alpha\in (0,\infty)$
    it holds that
    $\sup_{x\in O^k}[\eta(0,x)-\alpha\Phi(0,x)]
    =-\infty$ 
    imply that
    for all $\alpha\in(0,\infty)$
    it holds that
    \begin{equation}
    S_{\alpha,N}
    =\sup_{t\in K_N\cap(0,T]}\sup_{x\in (O_N)^k}
    [\eta(t,x)-\alpha\Phi(t,x)].
    \end{equation}
    Combining this and 
    \cite[Lemma 3.1 (i)]{beck2021nonlinear}
    (applied with $\mathcal{O}\curvearrowleft
    (K_N\cap(0,T])\times (O_N)^k$, 
    $\eta\curvearrowleft \eta|_{(K_N\cap(0,T])
        \times(O_N)^k}$,
    $\phi\curvearrowleft\Phi|_{(K_N\cap (0,T])
        \times (O_N)^k}$,
    $x\curvearrowleft ((0,\infty)\ni\alpha
    \mapsto (t^{(\alpha)},x^{(\alpha)})\in 
    (K_N\cap (0,T])\times (O_N)^k)$
    in the notation of \cite[Lemma 3.1]{beck2021nonlinear})
    demonstrates that
    \begin{equation}\label{eq:alpha_Phi_limit}
    0=
    \limsup_{\alpha\to \infty}
    [\alpha\Phi(t^{(\alpha)},x^{(\alpha)})]
    =\limsup_{\alpha\to\infty}\bigg[
    \frac{\alpha}{2}\sum_{i=2}^k 
    \norm{x^{(\alpha)}_i-x^{(\alpha)}_{i-1}}^2
    \bigg].
    \end{equation}
    In the next step note that
    item~\ref{it:O_N_compact}
    ensures
    that there exist $\mathfrak{t}\in K_N$,
    $\mathfrak{x}=(\mathfrak{x}_1, \mathfrak{x}_2,
    \ldots,$ 
    $\mathfrak{x}_k)\in (O_N)^k$,
    $(\alpha_n)_{n\in\N}\subseteq \N$
    which satisfy 
    $\liminf_{n\to\infty}\alpha_n=\infty$
    and $\limsup_{n\to\infty}
    [\abs{t^{(\alpha_n)}- \mathfrak{t}}
    +\norm{x^{(\alpha_n)}-\mathfrak{x}}]
    =0$.
    Moreover, observe that
    the fact that $\eta$ is upper
    semi-continuous 
    and the fact that $\Phi$ is continuous
    show
    that
    \begin{equation}
    \eta(\mathfrak{t},\mathfrak{x})
    \geq \limsup_{n\to\infty}
    [\eta(t^{(\alpha_n)},x^{(\alpha_n)})
    -\alpha_n\Phi(t^{(\alpha_n)},x^{(\alpha_n)})]
    \geq S
    > 0.
    \end{equation}
    The fact that
    for all $x\in O^k$ it holds that
    $\eta(0,x)=-\infty$ 
    hence demonstrates
    that $\mathfrak{t}\in K_N\cap(0,T]
    \subseteq (0,T)$.
    In addition, note that
    \eqref{eq:S_alpha_N}
    and the fact
    that for all $\alpha\in (0,\infty)$
    it holds that 
    $\eta(0,x^{(\alpha)})
    -\alpha\Phi(0,x^{(\alpha)})
    = -\infty$
    imply that for all $n\in\N$
    it holds that
    $t^{(\alpha_n)}\in K_N\cap (0,T]
    \subseteq (0,T)$.
    This and
    \cite[Lemma 3.1 (ii)]{beck2021nonlinear}
    (applied with $\mathcal{O}\curvearrowleft
    (K_N\cap(0,T])\times (O_N)^k$, 
    $\eta\curvearrowleft \eta|_{(K_N\cap(0,T])
        \times(O_N)^k}$,
    $\phi\curvearrowleft\Phi|_{(K_N\cap (0,T])
        \times (O_N)^k}$,
    $x\curvearrowleft ((0,\infty)\ni\alpha
    \mapsto (t^{(\alpha)},x^{(\alpha)})\in 
    (K_N\cap (0,T])\times (O_N)^k)$
    in the notation of \cite[Lemma 3.1]{beck2021nonlinear})
    prove that
    $0= \Phi(\mathfrak{t},\mathfrak{x})
    =\frac{1}{2}\sum_{i=2}^k\norm{\mathfrak{x}_i
        -\mathfrak{x}_{i-1}}^2$
    and 
    $\eta(\mathfrak{t},\mathfrak{x})
    =\sup_{(t,x)\in[\Phi^{-1}(0)]\cap 
        [(K_N\cap(0,T])\times (O_N)^k]}$ 
    $\eta(t,x)$.
    Therefore, we obtain
    for all $i\in \{1,2,\ldots, k\}$
    that
    $\mathfrak{x}_i = \mathfrak{x}_1$
    and
    \begin{equation}
    \label{eq:conv_to_S}
    S \leq 
    \lim_{\alpha\to\infty} S_{\alpha,N}
    \leq \sum_{i=1}^k v_i(\mathfrak{t},\mathfrak{x}_i)
    =\eta (\mathfrak{t},\mathfrak{x})
    =\sup_{t\in K_N}\sup_{y\in O_N}
    \bigg[\sum_{i=1}^k v_i(t,y) \bigg]  
    \leq S.
    \end{equation}
    %
    %
    Next observe that
    \cite[Lemma 3.2]{beck2021nonlinear}
    (applied for every $n\in\N$ with 
    $\mathcal{O}\curvearrowleft O$,
    $\varepsilon\curvearrowleft\frac{1}{\alpha_n}$,
    $\Phi\curvearrowleft \alpha_n\Phi|_{(0,T)\times O^k}$,
    $(u_i)_{i\in\{1,2,\ldots,k\}}\curvearrowleft
    (v_i|_{(0,T)\times O})_{i\in\{1,2,\ldots,k\}}$,
    $(G_i)_{i\in\{1,2,\ldots,k\}}
    \curvearrowleft (H_i)_{i\in\{1,2,\ldots,k\}}$,
    $\mathfrak{t}\curvearrowleft t^{(\alpha_n)}$,
    $\mathfrak{x}\curvearrowleft x^{(\alpha_n)}$
    in the notation of \cite[Lemma 3.2]{beck2021nonlinear})
    and \eqref{eq:eta_alpha_phi_globmax}
    demonstrates
    that there exist
    $b^{(\alpha_n)}_1, b^{(\alpha_n)}_2,\ldots, b^{(\alpha_n)}_k\in\R$, $n\in\N$,
    and
    $A^{(\alpha_n)}_1,A^{(\alpha_n)}_2,\ldots, A^{(\alpha_n)}_k\in\mathbb{S}_d$,
    $n\in\N$,
    which satisfy that
    \begin{enumerate}[label=(\Alph*)]
        \item \label{it:b_n_A_n_1}
        for all $n\in\N$,
        $i\in\{1,2,\ldots, k\}$ it holds that
        \begin{equation}
        (b^{(\alpha_n)}_i,\alpha_n(\nabla_{x_i}\Phi)(t^{(\alpha_n)},
        x^{(\alpha_n)}),
        \alpha_n A^{(\alpha_n)}_i)
        \in (\mathfrak{P}^+ v_i)(t^{(\alpha_n)},x^{(\alpha_n)}_i), 
        \end{equation}
        \item\label{it:b_n_A_n_2} 
        for all $n\in\N$
        it holds that
        $\sum_{i=1}^k b^{(\alpha_n)}_i
        =\alpha_n (\frac{\partial}{\partial t}\Phi)(t^{(\alpha_n)},x^{(\alpha_n)})
        =0$,
        and 
        \item\label{it:b_n_A_n_3}
        for all $n\in\N$
        it holds that
        \begin{equation}\label{eq:Hess_Phi1}
        \begin{split}
        &\Big(-\alpha_n+\alpha_n \Vnorm{
            (\operatorname{Hess}_x\Phi)
            \big(t^{(\alpha_n)},x^{(\alpha_n)}\big)} \Big)
        I_{kd} \\
        &\leq \alpha_n \begin{pmatrix}
        A^{(\alpha_n)}_1 &  \ldots & 0\\
        \vdots & \ddots & \vdots\\
        0 & \ldots & A^{(\alpha_n)}_k
        \end{pmatrix}\\
        &\leq \alpha_n\big(\operatorname{Hess}_x\Phi)(t^{(\alpha_n)},x^{(\alpha_n)}\big)
        +\tfrac{1}{\alpha_n}
        \Big[\alpha_n (\operatorname{Hess}_x\Phi)
        \big(t^{(\alpha_n)},x^{(\alpha_n)}\big)\Big]^2.
        \end{split}
        \end{equation}
    \end{enumerate}
    Note that the fact that for all $t\in (0,T)$,
    $x\in O^k$ it holds that
    $(\operatorname{Hess}_x\Phi)(t,x)
    =(\operatorname{Hess}_x\Phi)(0,0)$
    and item~\ref{it:b_n_A_n_3}
    prove that
    for all $n\in\N$
    it holds that
    \begin{equation}\label{eq:Hess_Phi2}
    \begin{split}
    &-\big(1+\Vnorm{
        (\operatorname{Hess}_x\Phi)(0,0)}\big)
    I_{kd}\\
    &\leq \begin{pmatrix}
    A^{(\alpha_n)}_1 & \ldots & 0\\
    \vdots & \ddots & \vdots\\
    0 & \ldots & A^{(\alpha_n)}_k
    \end{pmatrix}
    \leq (\operatorname{Hess}_x\Phi)(0,0)
    +\big[(\operatorname{Hess}_x\Phi)(0,0)\big]^2.
    \end{split}
    \end{equation}      
    Furthermore, observe that
    \cite[Lemma 2.16]{beck2021nonlinear}
    (applied for all 
    $i\in \{ 1,\ldots,k\}$ with 
    $u\curvearrowleft v_i$,
    $G\curvearrowleft H_i$
    in the notation of 
    \cite[Lemma 2.16]{beck2021nonlinear}), 
    item~\ref{it:b_n_A_n_1},
    and \eqref{eq:v_i_vs}
    ensure that
    for all $n\in\N$,
    $i\in\{1,2,\ldots,k\}$
    it holds that
    \begin{equation}
    b^{(\alpha_n)}_i 
    + H_i\Big(t^{(\alpha_n)},x^{(\alpha_n)}_i,
    v_i(t^{(\alpha_n)},x^{(\alpha_n)}_i),
    \alpha_n(\nabla_{x_i}\Phi)
    (t^{(\alpha_n)},x^{(\alpha_n)}), 
    \alpha_n A^{(\alpha_n)}_i\Big)
    \geq 0.
    \end{equation}
    Item~\ref{it:b_n_A_n_2},
    \eqref{eq:H_i}
    and the fact that 
    $\frac{\partial}{\partial t}
    \Phi = 0$ hence
    show that 
    for all $n\in\N$
    it holds that
    \begin{equation}\label{eq:G_i_upper}
    \begin{split}
    &\sum_{i=1}^k G_i\Big(
    t^{(\alpha_n)},x^{(\alpha_n)}_i,
    v_i(t^{(\alpha_n)},x^{(\alpha_n)}_i)+\frac{\delta}{t^{(\alpha_n)}},
    \alpha_n(\nabla_{x_i}\Phi)(t^{(\alpha_n)},x^{(\alpha_n)}),
    \alpha_n A^{(\alpha_n)}_i \Big)\\
    &\geq \frac{k\delta}{[t^{(\alpha_n)}]^2}.
    \end{split}
    \end{equation}
    Throughout the rest of the proof let
    $r^{(n)}_i$, $n\in\N$, $i\in\{1,2,\ldots,k\}$
    satisfy for all $n\in\N$,
    $i\in\{1,2,\ldots, k\}$ that
    \begin{equation}\label{eq:r_i^n}
    r^{(n)}_i= v_i(t^{(n)}, x_i^{(n)})
    +\frac{\delta}{t^{(n)}}. 
    \end{equation}
    This 
    and the fact that 
    $\limsup_{n\to \infty}
    \abs{\eta(t^{(\alpha_n)},
        x^{(\alpha_n)})-S}=0$
    ensure that
    \begin{equation}\label{eq:sum_r_n}
    \liminf_{n\to\infty}\bigg[
    \sum_{i=1}^k r^{(\alpha_n)}_i\bigg]
    =\limsup_{n\to\infty}\bigg[
    \sum_{i=1}^k r^{(\alpha_n)}_i\bigg]
    = S +\frac{k\delta}{\mathfrak{t}}
    >0.
    \end{equation}
    In  addition, observe that
    the fact that 
    $\{(t^{(\alpha_n)}, x^{(\alpha_n)})
    \in (0,T)\times O^k\colon n\in\N\}
    \cup \{(\mathfrak{t},\mathfrak{x})\}$
    is compact 
    and the fact that for all
    $i\in\{1,2,\ldots, k\}$
    it holds that $v_i$ is
    upper semi-continuous 
    demonstrate that
    \begin{equation}\label{eq:sup_r_n}
    \sup\big\{r^{(\alpha_n)}_i\colon
    n\in\N, i\in \{1,2,\ldots,k\}\big\}
    <\infty.
    \end{equation}  
    This and 
    \eqref{eq:sum_r_n}
    show that
    \begin{equation}
    \sup_{n\in\N}\bigg[
    \sum_{i=1}^k 
    r^{(\alpha_n)}_i\bigg]
    <\infty.
    \end{equation} 
    %
    Moreover, observe that
    \eqref{eq:Phi}
    ensures that
    for all $t\in (0,T)$, 
    $x=(x_1,x_2,\ldots,x_k)\in O^k$
    it holds that
    \begin{equation}
    \label{eq:nabla_Phi}
    \begin{split}
    &\Big(\frac{\partial}{\partial x_i}
    \Phi\Big)(t,x)
    =\begin{cases}
    x_1-x_2 &\colon 1=i<k\\
    2x_i-x_{i-1}-x_{i+1} &\colon 1<i<k\\
    x_k-x_{k-1} &\colon 1< i=k\\
    0 &\colon 1=i=k
    \end{cases}\\
    &= \indicator{[2,k]}(i)
    [x_i-x_{i-1}]
    +\indicator{[1,k-1]}(i)
    [x_i-x_{i+1}].
    \end{split}
    \end{equation}
    This and
    the Taylor expansion
    $\forall z \in (\R^d)^k\colon
    \Phi(0,z)=\Phi(0,0)
    +\langle (\nabla_x \Phi)(0,0),z\rangle
    +\frac{1}{2}\langle z,
    (\operatorname{Hess}_x \Phi)(0,0)z\rangle$
    $= \frac{1}{2}\langle z,
    (\operatorname{Hess}_x \Phi)(0,0)z\rangle$
    demonstrate that
    for all $z\in (\R^d)^k$
    it holds that
    \begin{equation}
        (\nabla_x\Phi)(0,z)
        =(\operatorname{Hess}_x\Phi)(0,0)z.
    \end{equation}
    Combining this with
    \eqref{eq:nabla_Phi}
    and the fat that for all 
    $a,b\in\R$ 
    it holds that
    $(a+b)^2\leq 2(a^2+b^2)$
    proves that
    for all $z\in (\R^d)^k$
    it holds that
    \begin{equation}\label{eq:Hess_Phi^2_1}
    \begin{split}
    &\langle z, ((\operatorname{Hess}_x\Phi)
    (0,0))^2 z\rangle
    = \langle (\operatorname{Hess}_x\Phi)(0,0)z,
    (\operatorname{Hess}_x \Phi)(0,0)z\rangle\\
    &=\norm{(\operatorname{Hess}_x\Phi)(0,0)z}^2
    =\norm{(\nabla_x \Phi)(0,z)}^2\\
    &=\norm{z_1-z_2}^2
    +\bigg[
    \sum_{i=2}^{k-1}\norm{2z_i
        -z_{i-1}-z_{i+1}}^2\bigg]
    +\norm{z_k-z_{k-1}}^2\\
    &\leq 2\norm{z_1-z_2}^2+\bigg[
    \sum_{i=2}^{k-1}2(\norm{z_i
        -z_{i-1}}^2+\norm{z_i-z_{i+1}}^2)\bigg]
    +2\norm{z_k-z_{k-1}}^2\\
    &= 4\bigg[\sum_{i=2}^k 
    \norm{z_i-z_{i-1}}^2 \bigg].   
    \end{split}
    \end{equation}
    Hence, we obtain for all
    $z\in (\R^d)^k$ 
    that
    \begin{equation}\label{eq:Hess_Phi^2_2}
    \norm{(\operatorname{Hess}_x \Phi)(0,0)
        z}^2
    \leq 8 \bigg[\sum_{i=2}^k \norm{z_i}^2\bigg]
    + 8 \bigg[\sum_{i=2}^k \norm{z_{i-1}}^2\bigg]
    \leq 16 \norm{z}^2.
    \end{equation}
    This ensures that
    $\Vnorm{
        (\operatorname{Hess}_x\Phi)(0,0)}
    \leq 4$.
    Combining 
    this with \eqref{eq:Hess_Phi2}
    and \eqref{eq:Hess_Phi^2_1}
    demonstrates that
    for all $n\in\N$,
    $z_1,z_2,\ldots, z_k \in\R^d$
    it holds that
    \begin{equation}
    \begin{split}
    &-5\bigg[\sum_{i=1}^k \norm{z_i}^2 \bigg]
    \leq \sum_{i=1}^k \langle z_i,
    A^{(\alpha_n)}_i z_i \rangle\\
    &\leq 2\Phi(0,z) + \langle z,
    (\operatorname{Hess}_x\Phi)(0,0))^2z
    \rangle 
    \leq 5 \bigg[ \sum_{i=2}^k
    \norm{z_i-z_{i-1}}^2\bigg].
    \end{split}
    \end{equation}
    Combining this with
    \eqref{eq:G_convergence}
    and \eqref{eq:G_i_upper}
    -\eqref{eq:nabla_Phi}
    proves that
    \begin{equation}
    \begin{split}
    &0 < \frac{k\delta}{\mathfrak{t}^2}
    =\limsup_{n\to\infty} \frac{k\delta}{[t^{(\alpha_n)}]^2}\\
    &\leq \limsup_{n\to\infty} \bigg[
    \sum_{i=1}^k G_i\Big(t^{(\alpha_n)},
    x^{(\alpha_n)}_i, 
    r^{(\alpha_n)}_i,
    \alpha_n\big(\indicator{[2,k]}(i)[x^{(\alpha_n)}_i-x^{(\alpha_n)}_{i-1}]\\
    &\qquad
    +\indicator{[1,k-1]}(i)[x^{(\alpha_n)}_i-x^{(\alpha_n)}_{i+1}]\big),
    \alpha_n A^{(\alpha_n)}_i\Big)\bigg]
    \leq 0.
    \end{split}
    \end{equation}
    This contradiction implies that
    $S\leq 0$.
    Therefore, we obtain that
    for all $t\in (0,T]$, $y\in O$
    it holds that
    $\sum_{i=1}^k u_i(t,y)
    \leq \frac{k\delta}{t}$.
    The proof of Lemma~\ref{lem:vs3_3}
    is thus complete.
\end{proof}

The following corollary
extends \cite[Corollary 3.4]{beck2021nonlinear}
which assumes \eqref{eq:u_ass} 
to hold without $\sqrt{T-t}$
and $(0,T)\setminus K_r$ replaced
by $(0,T)$.

\begin{corollary}\label{cor:vs3_4}
    Let $d\in\N$, $T\in(0,\infty)$,
    let $\langle\cdot,\cdot\rangle\colon
    \R^d\times\R^d\to\R$ be the standard
    Euclidean scalar product on $\R^d$,
    let $\norm{\cdot}\colon\R^d\to[0,\infty)$
    be the standard Euclidean norm on $\R^d$,
    let $O\subseteq \R^d$ be a non-empty 
    open set, 
    for every $r\in(0,\infty)$ let 
    $K_r\subseteq[0,T)$, $O_r\subseteq O$
    satisfy $K_r=[0,\max\{T-\frac{1}{r},0\}]$
    and
    $O_r = \{  x \in O\colon \norm{x} \leq r \text{ and } \{ y \in \R^d\colon \norm{y-x}
    < \frac{1}{r} \} \subseteq O \}$,
    let $G\in C((0,T)\times O\times\R\times\R^d\times\mathbb{S}_d,\R)$,
    $u,v\in C([0,T]\times O, \R)$,
    assume that
    \begin{equation}
    \label{eq:u_ass}
    \begin{split}
    &\sup\nolimits_{x\in O}(u(T,x)-v(T,x))\leq 0
    \quad \text{and}\\
    &\inf\nolimits_{r\in(0,\infty)}\Big[
    \sup\nolimits_{t\in(0,T)\setminus K_r}
    \sup\nolimits_{x\in O\setminus O_r}
    \big[(u(t,x)-v(t,x))\sqrt{T-t}\big]\Big]
    \leq 0,
    \end{split}
    \end{equation}
    assume that $G$ is degenerate elliptic,
    assume that $u$ is a viscosity solution of
    \begin{equation}
    (\tfrac{\partial u}{\partial t})(t,x)
    +G(t,x,u(t,x),(\nabla_x u)(t,x),
    (\operatorname{Hess}_x u)(t,x))
    \geq 0
    \end{equation}   
    for $(t,x)\in(0,T)\times O$,
    assume that $v$ is a viscosity solution of
    \begin{equation}
    \label{eq:v_vs4}
    (\tfrac{\partial v}{\partial t})(t,x)
    +G(t,x,v(t,x),(\nabla_x v)(t,x),
    (\operatorname{Hess}_x v)(t,x))
    \leq 0
    \end{equation}
    for $(t,x)\in(0,T)\times O$,
    and assume for all $t_n\in(0,T)$,
    $n\in\N_0$, 
    all $(x_n,r_n,A_n)\in O\times\R\times\mathbb{S}_d$,
    $n\in\N_0$,
    and all $(\mathfrak{x}_n,\mathfrak{r}_n,
    \mathfrak{A}_n)\in O\times\R\times\mathbb{S}_d$,
    $n\in\N_0$, with
    $\limsup_{n\to\infty}[\abs{t_n-t_0}
    +\norm{x_n-x_0}]+\sqrt{n}\norm{x_n-\mathfrak{x}_n}
    =0<\liminf_{n\to\infty}(r_n-\mathfrak{r}_n)
    =\limsup_{n\to\infty}(r_n-\mathfrak{r}_n)
    \leq \sup_{n\in\N}(\abs{r_n}+\abs{\mathfrak{r}_n})
    <\infty$
    and $\forall(n\in\N,z,\mathfrak{z}\in\R^d)\colon
    -5(\norm{z}^2+\norm{\mathfrak{z}}^2)
    \leq \langle z, A_nz\rangle
    -\langle\mathfrak{z},\mathfrak{A}_n\mathfrak{z}\rangle
    \leq 5\norm{z-\mathfrak{z}}^2$ 
    that
    \begin{equation}\label{eq:G_convergence2}
    \limsup_{n\to\infty}[G(t_n,x_n,r_n,
    n(x_n-\mathfrak{x}_n),nA_n)
    -G(t_n,\mathfrak{x}_n,\mathfrak{r}_n,
    n(x_n-\mathfrak{x}_n),n\mathfrak{A}_n)]
    \leq 0.
    \end{equation}
    Then it holds for all $t\in [0,T]$,
    $x\in O$ that
    $u(t,x)\leq v(t,x)$.
\end{corollary}

\begin{proof}[Proof of Corollary~\ref{cor:vs3_4}]
    Throughout this proof let 
    $H\colon (0,T)\times O\times\R\times\R^d
    \times\mathbb{S}_d\to \R$
    satisfy for all $t\in(0,T)$, $x\in O$,
    $r\in\R$, $p\in\R^d$, $A\in\mathbb{S}_d$
    that
    \begin{equation}\label{eq:H}
    H(t,x,r,p,A)=-G(t,x,-r,-p,-A).
    \end{equation}
    Note that the fact that $G$ is
    degnerate elliptic implies that $H$
    is degenerate elliptic.
    In addition, observe that \eqref{eq:H}
    and the assumption that 
    $G\in C((0,T)\times O\times\R\times\R^d\times\mathbb{S}_d,\R)$
    ensure that
    $H\in C((0,T)\times O\times\R\times\R^d\times\mathbb{S}_d,\R)$.
    Next note that 
    \eqref{eq:v_vs4}
    and \eqref{eq:H}
    assure that $-v$ is a viscosity solution of
    \begin{equation}
    (\tfrac{\partial(-v)}{\partial t})(t,x)
    +H(t,x,(-v)(t,x),(\nabla_x(-v))(t,x),
    (\operatorname{Hess}_x(-v))(t,x))
    \geq 0
    \end{equation}
    for $(t,x)\in(0,T)\times O$.
    Furthermore, observe that \eqref{eq:G_convergence2}
    shows that for all $t_n\in(0,T)$, $n\in\N_0$,
    all $(x_n,r_n,A_n)\in O\times\R\times\mathbb{S}_d$,
    $n\in\N_0$,
    and all $(\mathfrak{x}_n,\mathfrak{r}_n,
    \mathfrak{A}_n)\in O\times\R\times\mathbb{S}_d$,
    $n\in\N_0$, with
    $\limsup_{n\to\infty}[\abs{t_n-t_0}
    +\norm{x_n-x_0}]+\sqrt{n}\norm{x_n-\mathfrak{x}_n}
    =0<\liminf_{n\to\infty}(r_n+\mathfrak{r}_n)
    =\limsup_{n\to\infty}(r_n+\mathfrak{r}_n)
    \leq \sup_{n\in\N}(\abs{r_n}+\abs{\mathfrak{r}_n})
    <\infty$
    and $\forall(n\in\N,z,\mathfrak{z}\in\R^d)\colon
    -5(\norm{z}^2+\norm{\mathfrak{z}}^2)
    \leq \langle z, A_nz\rangle
    +\langle\mathfrak{z},\mathfrak{A}_n\mathfrak{z}\rangle
    \leq 5\norm{z-\mathfrak{z}}^2$ 
    it holds that
    \begin{equation}
    \begin{split}
    &\limsup_{n\to\infty}\Big[
    G(t_n,x_n, r_n,n(x_n-\mathfrak{x}_n),nA_n)
    +H(t_n,\mathfrak{x}_n,\mathfrak{r}_n,
    n(\mathfrak{x}_n-x_n),n\mathfrak{A}_n) \Big]\\
    &=\limsup_{n\to\infty}\Big[
    G(t_n,x_n, r_n,n(x_n-\mathfrak{x}_n),nA_n)
    -G(t_n,\mathfrak{x}_n,-\mathfrak{r}_n,
    n(x_n-\mathfrak{x}_n),-n\mathfrak{A}_n)\Big]\\
    &\leq 0.
    \end{split}
    \end{equation}
    Lemma~\ref{lem:vs3_3}
    (applied with 
    $k\curvearrowleft 2$,
    $u_1\curvearrowleft u$,
    $u_2\curvearrowleft -v$,
    $G_1\curvearrowleft G$,
    $G_2 \curvearrowleft H$
    in the notation of Lemma~\ref{lem:vs3_3})
    therefore demonstrates that
    for all $t\in (0,T]$, $x\in O$
    it holds that 
    $u(t,x)-v(t,x)\leq 0$.           
    Combining this with the assumption that
    $u,v\in C([0,T]\times O,\R)$ 
    proves
    that for all $t\in [0,T]$, $x\in O$
    it holds that  $u(t,x)-v(t,x)\leq 0$.
    The proof of Corollary~\ref{cor:vs3_4}
    is thus complete.    
\end{proof}

The following proposition
extends \cite[Proposition 3.5]{beck2021nonlinear}
to the case of semi-linear PDEs with gradient-dependent nonlinearities.

\begin{prop}
    \label{prop:vs3_5}
    Let $d,m\in\N$, $L,T\in (0,\infty)$,
    let $\langle\cdot,\cdot\rangle\colon
    \R^d\times\R^d\to\R$ be the
    standard Euclidean scalar product
    on $\R^d$,
    let $\norm{\cdot}\colon\R^d\to[0,\infty)$
    be the standard Euclidean norm on $\R^d$,
    let $\Vnorm{\cdot}\colon\R^{d+1}\to[0,\infty)$
    be the standard Euclidean norm
    on $\R^{d+1}$,
    let 
    $\norm{\cdot}_F \colon (\bigcup_{a,b=1}^\infty
    \R^{a\times b})\to [0,\infty)$
    satisfy for all $a,b\in\N$,
    $A=(A_{ij})_{(i,j)\in \{1,2,\ldots, a\}
        \times\{1,2,\ldots,b\}}\in\R^{a\times b}$ that
    $\norm{A}_F=[\sum_{i=1}^a
    \sum_{j=1}^b \abs{A_{ij}}^2
    ]^{\frac{1}{2}}$,
    let $O\subseteq \R^d$ be a non-empty
    open set,
    for every $r\in(0,\infty)$ let
    $K_r\subseteq[0,T)$, $O_r\subseteq O$
    satisfy $K_r=[0,\max\{T-\frac{1}{r},0\}]$
    and
    $O_r = \{  x \in O\colon \norm{x} \leq r \text{ and } \{ y \in \R^d\colon \norm{y-x} 
    < \frac{1}{r} \} \subseteq O \}$,
    let $g\in C(O,\R)$, 
    $f\in C([0,T]\times O\times\R\times\R^{d},\R)$,
    $\mu\in C([0,T]\times O,\R^d)$,
    $\sigma\in C([0,T]\times O, \R^{d\times m})$,
    $V\in C^{1,2}([0,T]\times O, (0,\infty))$
    satisfy 
    for all $r\in(0,\infty)$ that
    \begin{equation}
    \label{eq:mu_sigma_loclip2}
    \begin{split}
    &\sup\bigg(\bigg\{ \tfrac{\norm{\mu(t,x)-\mu(t,y)}+\norm{\sigma(t,x)-\sigma(t,y)}_F}{\norm{x-y}}\colon t\in [0,T], x,y\in O_r, x\neq y \bigg\}\cup \{0\} \bigg)
    <\infty,
    \end{split}
    \end{equation}
    assume for all $t\in [0,T]$, $x\in O$, 
    $a,b\in\R$, $v,w\in\R^{d}$ that
    $\abs{f(t,x,a,v)-f(t,x,b,w)}
    \leq L \Vnorm{ (a,v)
        -(b,w) }$,
    $\limsup_{r\to\infty}[\sup_{s\in[0,T)\setminus K_r}
    \sup_{y\in O\setminus O_r}
    \frac{\abs{f(s,y,0,0)}}{V(s,y)}]=0$, and
    \begin{equation}
    \label{eq:V_vissol}
    \begin{split}
    &(\tfrac{\partial V}{\partial t})(t,x)
    +\langle \mu(t,x), (\nabla_x V)(t,x) \rangle
    +\tfrac{1}{2} \operatorname{Tr}(\sigma(t,x)[\sigma(t,x)]^* (\operatorname{Hess}_x V)(t,x))
    \\
    &\qquad
    +L\norm{(\nabla_x V)(t,x)}
    \leq 0,
    \end{split}
    \end{equation}
    and let $u_1, u_2\in 
    C([0,T]\times O, \R)$
    satisfy for all $i\in \{1,2\}$
    that
    $\limsup_{r\to\infty}
    [\sup_{t\in[0,T)\setminus K_r}$
    $\sup_{x\in O\setminus O_r}$
    $(\frac{\abs{u_i(t,x)}}{V(t,x)}\sqrt{T-t})]=0$
    and that
    $u_i$ is a viscosity solution of
    \begin{equation}
    \label{eq:u_i_vissol}
    \begin{split}
    &(\tfrac{\partial u_i}{\partial t})(t,x)
    +\langle \mu(t,x), (\nabla_x u_i)(t,x)\rangle
    +\tfrac{1}{2} \operatorname{Tr}(\sigma(t,x)[\sigma(t,x)]^* (\operatorname{Hess}_x u_i)(t,x))
    \\
    &\qquad 
    +f(t,x,u_i(t,x),(\nabla_x u_i)(t,x))
    = 0
    \end{split}
    \end{equation}
    with $u_i(T,x)=g(x)$ for 
    $(t,x)\in (0,T)\times O$.
    Then it holds for all $t\in[0,T]$, $x\in O$
    that $u_1(t,x)=u_2(t,x)$.
\end{prop}

\begin{proof}[Proof of Proposition~\ref{prop:vs3_5}]
    Throughout this proof 
    let $\V\colon [0,T]\times O\to (0,\infty)$
    satisfy for all $t\in [0,T]$, $x\in O$ that
    $\V(t,x)=e^{-Lt}V(t,x)$,
    let $v_i\colon [0,T]\times O\to \R$,
    $i\in\{1,2\}$,
    satisfy for all $i\in \{1,2\}$,
    $t\in [0,T]$, $x\in O$ that
    $v_i(t,x)= \frac{u_i(t,x)}{\V(t,x)}$,
    let $G\colon (0,T)\times O\times \R\times \R^d\times \mathbb{S}_d \to \R$ satisfy 
    for all $t\in(0,T)$, $x\in O$,
    $r\in\R$, $p\in\R^d$, $A\in\mathbb{S}_d$
    that
    \begin{equation}
    \label{eq:G}
    G(t,x,r,p,A)
    =\langle \mu(t,x), p\rangle
    +\tfrac{1}{2}\operatorname{Tr}(\sigma(t,x)[\sigma(t,x)]^*A)
    +f(t,x,r,p),
    \end{equation}    
    and let $H\colon(0,T)\times O\times\R\times\R^d\times \mathbb{S}_d\to\R$
    satisfy for all $t\in(0,T)$, $x\in O$,
    $r\in\R$, $p\in\R^d$, $A\in\mathbb{S}_d$
    that
    \begin{equation}
    \label{eq:H2}
    \begin{split}
    &H(t,x,r,p,A)
    = \tfrac{r}{\V(t,x)}(\tfrac{\partial}{\partial t}\V)(t,x)\\
    &\qquad
    +\tfrac{1}{\V(t,x)}G\Big(t,x,r\V(t,x),\V(t,x)p+r(\nabla_x \V)(t,x),\\
    &\qquad\quad 
    \V(t,x)A+p[(\nabla_x \V)(t,x)]^*+(\nabla_x \V)(t,x)p^*+r(\operatorname{Hess}_x \V)(t,x)\Big).
    \end{split}
    \end{equation}
    Observe that 
    \eqref{eq:V_vissol}
    and the assumption that
    $V\in C^{1,2}([0,T]\times O, (0,\infty))$
    ensure that for all $t\in[0,T]$,
    $x\in O$ it holds that
    $\V\in C^{1,2}([0,T]\times O, (0,\infty))$
    and
    \begin{equation}
    \label{eq:dubbelV_prop}
    \begin{split}
    &(\tfrac{\partial}{\partial t}\V)(t,x)
    +\langle \mu(t,x), (\nabla_x \V)(t,x)\rangle
    +\tfrac{1}{2}\operatorname{Tr}(\sigma(t,x)[\sigma(t,x)]^*(\operatorname{Hess}_x \V)(t,x))
    \\
    &\qquad +L\V(t,x)
    +L \norm{(\nabla_x\V)(t,x)}
    \leq 0.
    \end{split}
    \end{equation}
    Next note that 
    \eqref{eq:G}
    implies that $G\in C((0,T)\times O\times\R\times\R^d\times\mathbb{S}_d,\R)$
    is degenerate elliptic.
    Combining this with \eqref{eq:H2}
    shows that $H\in C((0,T)\times O\times \R\times \R^d\times \mathbb{S}_d,\R)$
    is degenerate elliptic.
    In the next step observe that
    the assumption that for all 
    $i\in\{1,2\}$, $x\in O$ it holds that
    $u_i(T,x)=g(x)$ implies that
    for all $x\in O$
    it holds that
    \begin{equation}
    \label{eq:v_1_v_2_ineq}
    v_1(T,x)\leq v_2(T,x) \leq v_1(T,x).
    \end{equation}
    Furthermore, note that the hypothesis
    that 
    $\limsup_{r\to\infty}
    [\sup_{t\in[0,T)\setminus K_r}
    \sup_{x\in O\setminus O_r}$\\
    $(\frac{\abs{u_1(t,x)}+\abs{u_2(t,x)}}{V(t,x)}
    \sqrt{T-t})]=0$
    shows that
    \begin{equation}
    \label{eq:v_1_v_2_conv}
    \limsup_{r\to\infty}
    \bigg[\sup_{t\in[0,T)\setminus K_r}
    \sup_{x\in O\setminus O_r}
    \Big(\abs{v_1(t,x)-v_2(t,x)}\sqrt{T-t}\Big)\bigg]
    =0.
    \end{equation}
    In addition, observe that
    \cite[Lemma 4.12]{HairerHutzenthalerJentzen2015}
    (applied for every $i\in \{1,2\}$ with
    $\tilde{G}\curvearrowleft -H$,
    $V\curvearrowleft \V$,
    $\tilde{u}\curvearrowleft v_i$
    in the notation of \cite[Lemma 4.12]{HairerHutzenthalerJentzen2015}),
    \eqref{eq:u_i_vissol},
    and \eqref{eq:H2}
    demonstrate that
    for all $i\in\{1,2\}$
    it holds that $v_i$
    is a viscosity solution of
    \begin{equation}
    (\tfrac{\partial}{\partial t}v_i)(t,x)
    +H(t,x,v_i(t,x),(\nabla_x v_i)(t,x),
    (\operatorname{Hess}_x v_i)(t,x))
    =0
    \end{equation}
    for $(t,x)\in (0,T)\times O$.
    Throughout the rest of the proof let
    $\textbf{e}_1, \textbf{e}_2, \ldots,
    \textbf{e}_m\in\R^m$ satisfy 
    $\textbf{e}_1=(1,0,\ldots,0)$,
    $\textbf{e}_2=(0,1,0,\ldots, 0)$,
    $\ldots$, $\textbf{e}_m=(0,\ldots,0,1)$,
    let $t_n\in(0,T)$, $n\in\N_0$, satisfy
    $\limsup_{n\to\infty}\abs{t_n-t_0}=0$,
    and let $(x_n,r_n,A_n)\in O\times \R\times\mathbb{S}_d$, $n\in\N_0$, and
    $(\mathfrak{x}_n, \mathfrak{r}_n,
    \mathfrak{A}_n)\in O\times\R\times
    \mathbb{S}_d$, $n\in\N_0$, satisfy
    $\limsup_{n\to\infty}[\abs{t_n-t_0}
    +\norm{x_n-x_0}+\sqrt{n}
    \norm{x_n-\mathfrak{x}_n}]=0
    < r_0 = \liminf_{n\to\infty} \abs{r_n-\mathfrak{r}_n}
    = \limsup_{n\to\infty}\abs{r_n-\mathfrak{r}_n}
    \leq \sup_{n\in\N} (\abs{r_n}
    +\abs{\mathfrak{r}_n}) < \infty$
    and for all $n\in\N$, $y,z\in\R^d$
    that
    $\langle y, A_n y\rangle
    -\langle z, \mathfrak{A}_n z \rangle
    \leq 5 \norm{y-z}^2$.
    Note that \eqref{eq:mu_sigma_loclip2}
    and the fact that
    $\limsup_{n\to\infty}[\sqrt{n}\norm{x_n-\mathfrak{x}_n}]=0$
    ensure that
    \begin{equation}
    \limsup_{n\to\infty}\Big[n\norm{\sigma(t_n,x_n)
        -\sigma(t_n,\mathfrak{x}_n)}_F^2\Big]
    =0.
    \end{equation}
    The fact that for all $B\in \mathbb{S}_d$,
    $C\in\R^{d\times m}$
    it holds that
    $\operatorname{Tr}(CC^*B)
    =\sum_{i=1}^m \langle C e_i, BC e_i\rangle$
    and
    the assumption that for all 
    $n\in\N$, $y,z\in\R^d$
    it holds that
    $\langle y, A_n y \rangle
    - \langle z, \mathfrak{A}_n z \rangle
    \leq 5 \norm{y-z}^2$
    therefore imply that
    \begin{equation}
    \label{eq:sigmasigma_conv1}
    \begin{split}
    &\limsup_{n\to\infty} \Big[\tfrac{1}{2}
    \operatorname{Tr}\Big(\tfrac{\sigma(t_n,x_n)[\sigma(t_n,x_n)]^*}{\V(t_n,x_n)}\V(t_n,x_n)nA_n
    -\tfrac{\sigma(t_n,\mathfrak{x}_n)[\sigma(t_n,\mathfrak{x}_n)]^*}{\V(t_n,\mathfrak{x}_n)}\V(t_n,\mathfrak{x}_n)n
    \mathfrak{A}_n\Big) \Big]\\
    &=\limsup_{n\to\infty} \Big[\tfrac{n}{2}
    \operatorname{Tr}\Big(\sigma(t_n,x_n)[\sigma(t_n,x_n)]^* A_n
    - \sigma(t_n,\mathfrak{x}_n)[\sigma(t_n,\mathfrak{x}_n)]^*
    \mathfrak{A}_n\Big) \Big]\\
    &=\limsup_{n\to\infty} \Big[\tfrac{n}{2}
    \sum_{i=1}^m \big( \langle \sigma(t_n,x_n)e_i,
    A_n \sigma(t_n,x_n) e_i\rangle\
    - \langle \sigma(t_n,\mathfrak{x}_n)e_i,
    \mathfrak{A}_n\sigma(t_n,\mathfrak{x}_n)e_i\rangle\big) \Big]\\
    &\leq \limsup_{n\to\infty} \bigg[ \sum_{i=1}^m
    \tfrac{5n}{2} \norm{\sigma(t_n,x_n)e_i
        -\sigma(t_n,\mathfrak{x}_n)e_i}^2\bigg]\\
    &= \tfrac{5}{2}\limsup_{n\to\infty} \big[ 
    n \norm{\sigma(t_n,x_n)
        -\sigma(t_n,\mathfrak{x}_n)}_F^2\big]
    =0.
    \end{split}
    \end{equation}
    Furthermore, note that
    \eqref{eq:mu_sigma_loclip2}
    and the fact that 
    $\V\in C^{1,2}([0,T]\times O,(0,\infty))$  
    show that
    for all compact 
    $\mathcal{K}\subseteq O$
    there exists $c\in\R$ which satisfies
    for all $s\in[0,T]$, 
    $y_1,y_2\in\mathcal{K}$ that
    \begin{equation}
    \begin{split}
    &\normmm{\tfrac{\sigma(s,y_1)[\sigma(s,y_1)]^*}{\V(s,y_1)}
        -\tfrac{\sigma(s,y_2)[\sigma(s,y_2)]^*}{\V(s,y_2)}}_{F}
    +\norm{(\nabla_x \V)(s,y_1)-(\nabla_x\V)(s,y_2)}\\
    &\leq c \norm{y_1-y_2}.
    \end{split}
    \end{equation}
    This, 
    the fact that
    $\limsup_{n\to\infty}[\abs{t_n-t_0}
    +\norm{x_n-x_0}]=0$, and 
    the assumption that
    $\limsup_{n\to\infty}[\sqrt{n}\norm{x_n-\mathfrak{x}_n}]=0$
    demonstrate that
    \begin{equation}
    \label{eq:sigmasigma_conv2a}
    \limsup_{n\to\infty} \bigg[
    n\norm{x-\mathfrak{x}_n} \normmm{\frac{\sigma(t_n,x_n)[\sigma(t_n,x_n)]^*}{\V(t_n,x_n)}
        -\frac{\sigma(t_n,\mathfrak{x}_n)[\sigma(t_n,\mathfrak{x}_n)]^*}{\V(t_n,\mathfrak{x}_n)}}_{F} \bigg]
    =0
    \end{equation}
    and
    \begin{equation}
    \label{eq:sigmasigma_conv2b}
    \limsup_{n\to\infty} [n\norm{x_n-\mathfrak{x}_n} \, \norm{(\nabla_x\V)(t_n,x_n)
        -(\nabla_x\V)(t_n, \mathfrak{x}_n)}]
    =0.
    \end{equation}
    In addition, note that for all 
    $B\in\mathbb{S}_d$, $v,w\in\R^d$
    it holds that
    \begin{equation}
    \begin{split}
    &\operatorname{Tr}(Bvw^*)
    =\operatorname{Tr}(w^*Bv)
    = w^*Bv 
    =\langle w, Bv \rangle
    =\langle Bw, v \rangle\\
    &=\langle v, Bw \rangle
    = v^* Bw
    = \operatorname{Tr}(v^*Bw)
    = \operatorname{Tr}(Bwv^*).
    \end{split}
    \end{equation}
    This, 
    Cauchy-Schwarz inequality,
    \eqref{eq:sigmasigma_conv2a},
    and \eqref{eq:sigmasigma_conv2b}
    demonstrate that
    \begin{equation}
    \label{eq:sigmasigma_conv3}
    \begin{split}
    &\limsup_{n\to\infty} \bigg[
    \tfrac{1}{2} \operatorname{Tr}\bigg(
    \tfrac{\sigma(t_n,x_n)[\sigma(t_n,x_n)]^*}{\V(t_n,x_n)}
    \Big(n(x_n-\mathfrak{x}_n)[(\nabla_x\V)(t_n,x_n)]^*\\
    &\qquad
    +(\nabla_x\V)(t_n,x_n)n(x_n-\mathfrak{x}_n)^*\Big)\\
    &\qquad - \tfrac{\sigma(t_n,\mathfrak{x}_n)[\sigma(t_n,\mathfrak{x}_n)]^*}{\V(t_n,\mathfrak{x}_n)}
    \Big(n(x_n-\mathfrak{x}_n)[(\nabla_x\V)(t_n,\mathfrak{x}_n)]^*\\
    &\qquad
    +(\nabla_x\V)(t_n,\mathfrak{x}_n)n(x_n-\mathfrak{x}_n)^*\Big)\bigg) \bigg]\\
    &=\limsup_{n\to\infty} \bigg[
    \bigg\langle \tfrac{\sigma(t_n,x_n)[\sigma(t_n,x_n)]^*}{\V(t_n,x_n)}
    n(x_n-\mathfrak{x}_n), (\nabla_x\V)(t_n,x_n)
    \bigg\rangle
    \\
    &\qquad - \bigg\langle\tfrac{\sigma(t_n,\mathfrak{x}_n)[\sigma(t_n,\mathfrak{x}_n)]^*}{\V(t_n,\mathfrak{x}_n)}
    n(x_n-\mathfrak{x}_n),
    (\nabla_x\V)(t_n,\mathfrak{x}_n)
    \bigg\rangle
    \bigg]\\
    &= \limsup_{n\to\infty} \bigg[
    \bigg\langle 
    \bigg(\tfrac{\sigma(t_n,x_n)[\sigma(t_n,x_n)]^*}{\V(t_n,x_n)}
    -\tfrac{\sigma(t_n,\mathfrak{x}_n)[\sigma(t_n,\mathfrak{x}_n)]^*}{\V(t_n,\mathfrak{x}_n)}\bigg)
    n(x_n-\mathfrak{x}_n), 
    (\nabla_x\V)(t_n,x_n)
    \bigg\rangle
    \\
    &\qquad + \bigg\langle\tfrac{\sigma(t_n,\mathfrak{x}_n)[\sigma(t_n,\mathfrak{x}_n)]^*}{\V(t_n,\mathfrak{x}_n)}
    n(x_n-\mathfrak{x}_n),
    (\nabla_x\V)(t_n,x_n)-(\nabla_x\V)(t_n,\mathfrak{x}_n)
    \bigg\rangle
    \bigg]\\
    &\leq \limsup_{n\to\infty} \bigg[
    \normmm{
        \tfrac{\sigma(t_n,x_n)[\sigma(t_n,x_n)]^*}{\V(t_n,x_n)}
        -\tfrac{\sigma(t_n,\mathfrak{x}_n)[\sigma(t_n,\mathfrak{x}_n)]^*}{\V(t_n,\mathfrak{x}_n)}}_{F}
    n\norm{x_n-\mathfrak{x}_n}
    \,\norm{(\nabla_x\V)(t_n,x_n)}
    \\
    &\qquad + \normmm{\tfrac{\sigma(t_n,\mathfrak{x}_n)[\sigma(t_n,\mathfrak{x}_n)]^*}{\V(t_n,\mathfrak{x}_n)}}_{F}\,
    n\norm{x_n-\mathfrak{x}_n}\,
    \norm{(\nabla_x\V)(t_n,x_n)-(\nabla_x\V)(t_n,\mathfrak{x}_n)}
    \bigg]\\
    &=0.
    \end{split}
    \end{equation}
    Next observe that
    the fact that 
    $(0,T)\times O \ni (s,y)\mapsto
    \frac{\sigma(s,y)[\sigma(s,y)]^*}{\V(s,y)}
    (\operatorname{Hess}_x\!\V)(s,y)$
    $\in\R^{d\times d}$
    is continuous 
    and the assumption that
    $\limsup_{n\to\infty}[\abs{t_n-t_0}
    +\norm{x_n-x_0}]=0$ and 
    $\limsup_{n\to\infty}[\sqrt{n}\norm{x_n-\mathfrak{x}_n}]=0$
    show that
    \begin{equation}
    \begin{split}
    &\limsup_{n\to\infty} \Big\lvert
    \operatorname{Tr}\Big(
    \tfrac{\sigma(t_n,\mathfrak{x}_n)[\sigma(t_n,\mathfrak{x}_n)]^*}{\V(t_n,\mathfrak{x}_n)}
    (\operatorname{Hess}_x\!\V)(t_n,\mathfrak{x}_n)\\
    &\qquad
    -\tfrac{\sigma(t_0,x_0)[\sigma(t_0,x_0)]^*}{\V(t_0,x_0)}
    (\operatorname{Hess}_x\!\V)(t_0,x_0)\Big)
    \Big\rvert\\
    &= \limsup_{n\to\infty} \Big\lvert
    \operatorname{Tr}\Big(
    \tfrac{\sigma(t_n,x_n)[\sigma(t_n,x_n)]^*}{\V(t_n,x_n)}
    (\operatorname{Hess}_x\!\V)(t_n,x_n)\\
    &\qquad
    -\tfrac{\sigma(t_0,x_0)[\sigma(t_0,x_0)]^*}{\V(t_0,x_0)}
    (\operatorname{Hess}_x\!\V)(t_0,x_0)\Big)
    \Big\rvert
    =0.
    \end{split}
    \end{equation}
    The fact that 
    $0< r_0 
    = \liminf_{n\to\infty} \abs{r_n-\mathfrak{r}_n}
    = \limsup_{n\to\infty}\abs{r_n-\mathfrak{r}_n}
    \leq \sup_{n\in\N} (\abs{r_n}
    +\abs{\mathfrak{r}_n}) < \infty$
    therefore ensures that 
    \begin{equation}
    \begin{split}
    &\limsup_{n\to\infty}\bigg[
    \tfrac{1}{2}\operatorname{Tr}\bigg(
    \tfrac{\sigma(t_n,x_n)[\sigma(t_n,x_n)]^*}{\V(t_n,x_n)}r_n (\operatorname{Hess}_x\!\V)(t_n,x_n)\\
    &\qquad
    -\tfrac{\sigma(t_n,\mathfrak{x}_n)[\sigma(t_n,\mathfrak{x}_n)]^*}{\V(t_n,\mathfrak{x}_n)}
    \mathfrak{r}_n (\operatorname{Hess}_x\!\V)(t_n,\mathfrak{x}_n)
    \bigg)\bigg]\\
    &=\tfrac{1}{2} \limsup_{n\to\infty}
    \bigg[ (r_n-\mathfrak{r}_n)
    \operatorname{Tr}\bigg(
    \tfrac{\sigma(t_n,x_n)[\sigma(t_n,x_n)]^*}{\V(t_n,x_n)} (\operatorname{Hess}_x\!\V)(t_n,x_n)\bigg)\\
    &\qquad +\mathfrak{r_n} 
    \operatorname{Tr}\bigg(
    \tfrac{\sigma(t_n,x_n)[\sigma(t_n,x_n)]^*}{\V(t_n,x_n)} (\operatorname{Hess}_x\!\V)(t_n,x_n)\\
    &\qquad
    -\tfrac{\sigma(t_n,\mathfrak{x}_n)[\sigma(t_n,\mathfrak{x}_n)]^*}{\V(t_n,\mathfrak{x}_n)}
    (\operatorname{Hess}_x\!\V)(t_n,\mathfrak{x}_n)
    \bigg) \bigg]\\
    &= \tfrac{r_0}{2\V(t_0,x_0)}
    \operatorname{Tr}\big(\sigma(t_0,x_0)
    [\sigma(t_0,x_0)]^* (\operatorname{Hess}_x\!\V)(t_0,x_0)\big).
    \end{split}
    \end{equation}
    Combining this with
    \eqref{eq:sigmasigma_conv1}
    and \eqref{eq:sigmasigma_conv3}
    ensures that
    \begin{equation}
    \label{eq:sigmasigma_conv_final}
    \begin{split}
    &\limsup_{n\to\infty}\bigg[
    \tfrac{1}{2}\operatorname{Tr}\bigg(
    \tfrac{\sigma(t_n,x_n)[\sigma(t_n,x_n)]^*}{\V(t_n,x_n)}
    \Big(\V(t_n,x_n)n A_n 
    +n(x_n-\mathfrak{x}_n)
    [(\nabla_x\V)(t_n,x_n)]^*\Big)\\
    &\qquad +(\nabla_x\V)(t_n,x_n)
    n(x_n-\mathfrak{x}_n)^*
    +r_n(\operatorname{Hess}_x\!\V)(t_n,x_n)\bigg)\\
    &\qquad -\tfrac{1}{2}\operatorname{Tr}\bigg(
    \tfrac{\sigma(t_n,\mathfrak{x}_n)[\sigma(t_n,\mathfrak{x}_n)]^*}{\V(t_n,\mathfrak{x}_n)}
    \Big(\V(t_n,\mathfrak{x}_n)
    n \mathfrak{A}_n 
    +n(x_n-\mathfrak{x}_n)
    [(\nabla_x\V)(t_n,\mathfrak{x}_n)]^*\Big)\\
    &\qquad +(\nabla_x\V)(t_n,\mathfrak{x}_n)
    n(x_n-\mathfrak{x}_n)^*
    +\mathfrak{r}_n
    (\operatorname{Hess}_x\!\V)(t_n,\mathfrak{x}_n)\bigg)
    \bigg]\\
    &\leq \tfrac{r_0}{2\V(t_0,x_0)}
    \operatorname{Tr}\big( \sigma(t_0,x_0)
    [\sigma(t_0,x_0)]^*
    (\operatorname{Hess}_x\!\V)(t_0,x_0)\big).
    \end{split}
    \end{equation}
    Next note that 
    the fact that $(0,T)\times O\ni(s,y)
    \mapsto \frac{1}{\V(s,y)}(\frac{\partial}{\partial t}\V)(s,y)\in\R$
    is continuous
    and the assumption that 
    $0<r_0
    =\liminf_{n\to\infty} (r_n-\mathfrak{r}_n)
    =\limsup_{n\to\infty} (r_n-\mathfrak{r}_n)
    \leq \sup_{n\in\N}(\abs{r_n}
    +\abs{\mathfrak{r}_n})<\infty$
    prove that
    \begin{equation}
    \label{eq:V_timeder_conv}
    \begin{split}
    &\limsup_{n\to\infty} \bigg[
    \tfrac{r_n}{\V(t_n,x_n)}
    (\tfrac{\partial}{\partial t}\V)(t_n,x_n)
    - \tfrac{\mathfrak{r}_n}{\V(t_n,\mathfrak{x}_n)}
    (\tfrac{\partial}{\partial t}\V)(t_n,\mathfrak{x}_n)\bigg]\\
    &=\limsup_{n\to\infty} \bigg[
    \tfrac{r_n-\mathfrak{r}_n}{\V(t_n,x_n)}
    (\tfrac{\partial}{\partial t}\V)(t_n,x_n)\\
    &\qquad
    +\mathfrak{r}_n \bigg( 
    \tfrac{1}{\V(t_n,x_n)}
    (\tfrac{\partial}{\partial t}\V)(t_n,x_n)
    - \tfrac{1}{\V(t_n,\mathfrak{x}_n)}
    (\tfrac{\partial}{\partial t}\V)(t_n,\mathfrak{x}_n)\bigg)
    \bigg]\\
    &=\tfrac{r_0}{\V(t_0,x_0)}(\tfrac{\partial}{\partial t}\V)(t_0,x_0).
    \end{split}
    \end{equation} 
    Furthermore, observe that
    \eqref{eq:mu_sigma_loclip2} and 
    the fact that
    $\limsup_{n\to\infty}[\abs{t_n-t_0}
    +\norm{x_n-x_0}]=0$ 
    $=\limsup_{n\to\infty}[\sqrt{n}\norm{x_n-\mathfrak{x}_n}]$
    ensure that 
    \begin{equation}
    \limsup_{n\to\infty} \Big[ 
    n\norm{\mu(t_n,x_n)-\mu(t_n,\mathfrak{x}_n)}\,\norm{x_n-\mathfrak{x}_n}\Big]
    =0.
    \end{equation}
    This,
    the Cauchy-Schwarz inequality,
    the fact that
    $(0,T)\times O\ni (s,y)\mapsto
    \langle \frac{\mu(s,y)}{\V(s,y)},
    (\nabla_x\V)(s,y) \rangle\in\R$
    is continuous, 
    and 
    the assumption that
    $0<r_0
    =\liminf_{n\to\infty} (r_n-\mathfrak{r}_n)
    =\limsup_{n\to\infty} (r_n-\mathfrak{r}_n)
    \leq \sup_{n\in\N}(\abs{r_n}
    +\abs{\mathfrak{r}_n})<\infty$
    imply that
    \begin{equation}
    \label{eq:mu_conv}
    \begin{split}
    &\limsup_{n\to\infty}\bigg[
    \frac{1}{\V(t_n,x_n)}\langle \mu(t_n,x_n),
    \V(t_n,x_n)n(x_n-\mathfrak{x}_n)
    +r_n(\nabla_x\V)(t_n,x_n) \rangle\\
    &\qquad
    -\frac{1}{\V(t_n,\mathfrak{x}_n)}\langle \mu(t_n,\mathfrak{x}_n),
    \V(t_n,\mathfrak{x}_n)n(x_n-\mathfrak{x}_n)
    +\mathfrak{r}_n(\nabla_x\V)(t_n,\mathfrak{x}_n) \rangle\bigg]\\
    &= \limsup_{n\to\infty}\bigg[
    \bigg\langle \mu(t_n,x_n)-\mu(t_n,\mathfrak{x}_n),
    n(x_n-\mathfrak{x}_n)\bigg\rangle\\
    &\qquad
    +r_n\bigg\langle \frac{\mu(t_n,x_n)}{\V(t_n,x_n)},
    (\nabla_x\V)(t_n,x_n)\bigg\rangle
    - \mathfrak{r}_n\bigg\langle \frac{\mu(t_n,\mathfrak{x}_n)}{\V(t_n,\mathfrak{x}_n)},
    (\nabla_x\V)(t_n,\mathfrak{x}_n)\bigg\rangle\bigg]\\
    &\leq \limsup_{n\to\infty}\big[
    \norm{ \mu(t_n,x_n)-\mu(t_n,\mathfrak{x}_n)}
    n\norm{x_n-\mathfrak{x}_n}\big]\\
    &\qquad  +\limsup_{n\to\infty}\bigg[
    (r_n-\mathfrak{r}_n)
    \bigg\langle \frac{\mu(t_n,x_n)}{\V(t_n,x_n)},
    (\nabla_x\V)(t_n,x_n)
    \bigg \rangle\bigg]\\
    &\qquad 
    +\limsup_{n\to\infty}\bigg[ \mathfrak{r}_n
    \bigg(\bigg\langle \frac{\mu(t_n,x_n)}{\V(t_n,x_n)},
    (\nabla_x\V)(t_n,x_n)\bigg\rangle\\
    &\qquad
    -\bigg\langle \frac{\mu(t_n,\mathfrak{x}_n)}{\V(t_n,\mathfrak{x}_n)},
    (\nabla_x\V)(t_n,\mathfrak{x}_n)\bigg\rangle
    \bigg)\bigg]\\
    &= \frac{r_0}{\V(t_0,x_0)}\big\langle \mu(t_0,x_0),
    (\nabla_x\V)(t_0,x_0)\big\rangle.
    \end{split}
    \end{equation}
    Next note that the assumption that
    $f\in C([0,T]\times O\times\R
    \times\R^{d},\R)$
    shows that for all compact
    $\mathcal{K}\subseteq[0,T]\times O
    \times\R\times\R^{d}$ it holds that
    \begin{equation}
    \label{eq:f_cont}
    \begin{split}
    \limsup\nolimits_{(0,\infty)\ni
        \,\varepsilon\to 0}
    &\Big[\sup\big(
    \{\abs{f(s_1,y_1,a_1,w_1)-
        f(s_2,y_2,a_2,w_2)}\colon\\
    &\qquad (s_1,y_1,a_1,w_1), 
    (s_2,y_2,a_2,w_2)\in \mathcal{K},\\
    &\qquad \abs{s_1-s_2}\leq \varepsilon,
    \norm{y_1-y_2}\leq \varepsilon,\\
    &\qquad
    \Vnorm{ (a_1,w_1)
        -(a_2,w_2)}\leq\varepsilon\}
    \cup \{0\}\big)\Big]=0.
    \end{split}
    \end{equation} 
    Moreover, observe that
    the assumption that for all 
    $s\in [0,T]$, $y\in O$, $a,b\in\R$, 
    $v,w\in\R^{d}$
    it holds that
    $\abs{f(s,y,a,v)-f(s,y,b,w)}\leq 
    \Vnorm{ (a,v)
        -(b,w)}$
    ensures that for all $n\in\N$
    it holds that
    \begin{equation}
    \begin{split}
    &\tfrac{f(t_n,x_n,r_n\V(t_n,x_n), \V(t_n,x_n)n(x_n-\mathfrak{x}_n)+r_n(\nabla_x\V)(t_n,x_n))}{V(t_n,x_n)}\\
    &\qquad
    -\tfrac{f(t_n,\mathfrak{x}_n,\mathfrak{r}_n\V(t_n,\mathfrak{x}_n), \V(t_n,\mathfrak{x}_n)n(x_n-\mathfrak{x}_n)+\mathfrak{r}_n(\nabla_x\V)(t_n,\mathfrak{x}_n))}{V(t_n,\mathfrak{x}_n)}\\
    &\leq \tfrac{f(t_n,x_n,r_n\V(t_n,x_n), \V(t_n,x_n)n(x_n-\mathfrak{x}_n)+r_n(\nabla_x\V)(t_n,x_n))}{\V(t_n,x_n)}\\
    &\qquad
    - \tfrac{f(t_n,\mathfrak{x}_n,r_n\V(t_n,\mathfrak{x}_n), \V(t_n,\mathfrak{x}_n)n(x_n-\mathfrak{x}_n)+r_n(\nabla_x\V)(t_n,\mathfrak{x}_n))}{\V(t_n,\mathfrak{x}_n)}\\
    &\quad +\tfrac{f(t_n,\mathfrak{x}_n,r_n\V(t_n,\mathfrak{x}_n), \V(t_n,\mathfrak{x}_n)n(x_n-\mathfrak{x}_n)
        +r_n(\nabla_x\V)(t_n,\mathfrak{x}_n))}{\V(t_n,\mathfrak{x}_n)}\\
    &\qquad
    -\tfrac{f(t_n,\mathfrak{x}_n,\mathfrak{r}_n\V(t_n,\mathfrak{x}_n), \V(t_n,\mathfrak{x}_n)n(x_n-\mathfrak{x}_n)
        +\mathfrak{r}_n(\nabla_x\V)(t_n,\mathfrak{x}_n))}{\V(t_n,\mathfrak{x}_n)}\\
    &\leq 
    \tfrac{f(t_n,x_n,r_n\V(t_n,x_n), \V(t_n,x_n)n(x_n-\mathfrak{x}_n)+r_n(\nabla_x\V)(t_n,x_n))}{\V(t_n,x_n)}\\
    &\qquad
    - \tfrac{f(t_n,\mathfrak{x}_n,r_n\V(t_n,\mathfrak{x}_n), \V(t_n,\mathfrak{x}_n)n(x_n-\mathfrak{x}_n)+r_n(\nabla_x\V)(t_n,\mathfrak{x}_n))}{\V(t_n,\mathfrak{x}_n)}\\
    &\qquad 
    +\tfrac{L}{\V(t_n,\mathfrak{x}_n)} 
    \vvvert (r_n\V(t_n, \mathfrak{x}_n),
    \V(t_n,\mathfrak{x}_n)n(x_n-\mathfrak{x}_n)
    +r_n(\nabla_x\V)(t_n,\mathfrak{x}_n)) \\
    &\qquad
    -(\mathfrak{r}_n \V(t_n,\mathfrak{x}_n),
    \V(t_n,\mathfrak{x}_n)n(x_n-\mathfrak{x}_n)
    +\mathfrak{r}_n(\nabla_x\V)(t_n,\mathfrak{x}_n))  \vvvert.
    \end{split}
    \end{equation}   
    This and \eqref{eq:f_cont}
    demonstrate that
    \begin{equation}
    \begin{split}
    &\limsup_{n\to \infty}
    \bigg[ \tfrac{f(t_n,x_n,r_n\V(t_n,x_n), \V(t_n,x_n)n(x_n-\mathfrak{x}_n)
        +r_n(\nabla_x\V)(t_n,x_n))}{\V(t_n,x_n)}\\
    &\qquad\qquad -\tfrac{f(t_n,\mathfrak{x}_n,\mathfrak{r}_n\V(t_n,\mathfrak{x}_n), \V(t_n,\mathfrak{x}_n)n(x_n-\mathfrak{x}_n)
        +\mathfrak{r}_n(\nabla_x\V)(t_n,\mathfrak{x}_n))}{\V(t_n,\mathfrak{x}_n)} \bigg]\\
    &\leq  \limsup_{n\to\infty} \bigg[ 
    L\abs{r_n-\mathfrak{r}_n}
    +\tfrac{L\abs{r_n-\mathfrak{r}_n}
        \norm{(\nabla_x \V)(t_n,\mathfrak{x}_n)}}{\V(t_n,\mathfrak{x}_n)}\bigg]
    = L r_0\Big(1 + \tfrac{\norm{(\nabla_x \V)(t_0,x_0)}}{\V(t_0,x_0)}\Big).
    \end{split}
    \end{equation}
    Combing this with \eqref{eq:G}, 
    \eqref{eq:H2},
    \eqref{eq:dubbelV_prop} 
    \eqref{eq:sigmasigma_conv_final},
    \eqref{eq:V_timeder_conv},
    \eqref{eq:mu_conv}, and
    proves that
    \begin{equation}
    \begin{split}
    &\limsup_{n\to\infty} [
    H(t_n,x_n,r_n,n(x_n-\mathfrak{x}_n),nA_n)
    -H(t_n,\mathfrak{x}_n,\mathfrak{r}_n,
    n(x_n-\mathfrak{x}_n),n\mathfrak{A}_n)]\\
    &= \limsup_{n\to\infty}\bigg[
    \tfrac{r_n}{\V(t_n,x_n)}(\tfrac{\partial}{\partial t}\V)(t_n,x_n)
    -\tfrac{\mathfrak{r}_n}{\V(t_n,\mathfrak{x}_n)}(\tfrac{\partial}{\partial t}\V)(t_n,\mathfrak{x}_n)\\
    &\qquad + \tfrac{1}{2}\operatorname{Tr}\bigg(
    \tfrac{\sigma(t_n,x_n)[\sigma(t_n,x_n)]^*}{\V(t_n,x_n)}
    \Big(\V(t_n,x_n)n A_n 
    +n(x_n-\mathfrak{x}_n)
    [(\nabla_x\V)(t_n,x_n)]^*\Big)\\
    &\qquad +(\nabla_x\V)(t_n,x_n)
    n(x_n-\mathfrak{x}_n)^*
    +r_n(\operatorname{Hess}_x\!\V)(t_n,x_n)\bigg)\\
    &\qquad -\tfrac{1}{2}\operatorname{Tr}\bigg(
    \tfrac{\sigma(t_n,\mathfrak{x}_n)[\sigma(t_n,\mathfrak{x}_n)]^*}{\V(t_n,\mathfrak{x}_n)}
    \Big(\V(t_n,\mathfrak{x}_n)
    n \mathfrak{A}_n 
    +n(x_n-\mathfrak{x}_n)
    [(\nabla_x\V)(t_n,\mathfrak{x}_n)]^*\Big)\\
    &\qquad +(\nabla_x\V)(t_n,\mathfrak{x}_n)
    n(x_n-\mathfrak{x}_n)^*
    +\mathfrak{r}_n
    (\operatorname{Hess}_x\!\V)(t_n,\mathfrak{x}_n)\bigg)\\
    &\qquad
    +\tfrac{1}{\V(t_n,x_n)}\langle \mu(t_n,x_n),
    \V(t_n,x_n)n(x_n-\mathfrak{x}_n)
    +r_n(\nabla_x\V)(t_n,x_n) \rangle\\
    &\qquad
    -\tfrac{1}{\V(t_n,\mathfrak{x}_n)}\langle \mu(t_n,\mathfrak{x}_n),
    \V(t_n,\mathfrak{x}_n)n(x_n-\mathfrak{x}_n)
    +\mathfrak{r}_n(\nabla_x\V)(t_n,\mathfrak{x}_n) \rangle\\
    &\qquad
    +\tfrac{f(t_n,x_n,r_n\V(t_n,x_n), \V(t_n,x_n)n(x_n-\mathfrak{x}_n)
        +r_n(\nabla_x\V)(t_n,x_n))}{\V(t_n,x_n)}\\
    &\qquad -\tfrac{f(t_n,\mathfrak{x}_n,\mathfrak{r}_n\V(t_n,\mathfrak{x}_n), \V(t_n,\mathfrak{x}_n)n(x_n-\mathfrak{x}_n)
        +\mathfrak{r}_n(\nabla_x\V)(t_n,\mathfrak{x}_n))}{\V(t_n,\mathfrak{x}_n)}
    \bigg]\\
    &\leq \tfrac{r_0}{\V(t_0,x_0)}\bigg[
    (\tfrac{\partial}{\partial t}\V)(t_0,x_0)
    +\tfrac{1}{2}\operatorname{Tr}(\sigma(t_0,x_0)
    [\sigma(t_0,x_0)]^*(\operatorname{Hess}_x\!\V)(t_0,x_0)\\
    &\qquad +\langle\mu(t_0,x_0),  (\nabla_x\V)(t_0,x_0)\rangle
    +L\V(t_0,x_0)
    +L\norm{(\nabla_x\V)(t_0,x_0)}\bigg]
    \leq 0. 
    \end{split}  
    \end{equation}
    Corollary~\ref{cor:vs3_4},
    \eqref{eq:v_1_v_2_ineq},
    and \eqref{eq:v_1_v_2_conv}
    therefore demonstrate that
    $v_1\leq v_2$ and 
    $v_2\leq v_1$.
    This implies $v_1=v_2$.
    Hence, we obtain that $u_1=u_2$.
    The proof of Proposition~\ref{prop:vs3_5}
    is thus complete.
\end{proof}


\section{Bismut-Elworthy-Li type formula}
\label{sec:BEL}

In this section we derive a 
Bismut-Elworthy-Li type formula that 
holds under certain assumptions.
To achieve this, we
work with results from the Malliavin
calculus to establish the derivative
representation in \eqref{eq:u_der_BEL}.
We therefore follow the notation of
\cite{Nualart1995} and
denote the Malliavin derivative
of a random variable $X\in \mathbb{D}^{1,2}$ by
$\{D_t X\colon t\in [0,T]\}$
where $ \mathbb{D}^{1,2} \subseteq L^2(\Omega, \R^d)$ 
denotes the space
of Malliavin differentiable random variables.
For the Skorohod integral of
a stochastic process $u\in L^2([0,T]\times \Omega, \R^d)$
we write $\int_0^T u_r \, \delta W_r$.                                              

To prove the Bismut-Elworthy-Li type
formula in Theorem~\ref{thm:BEL_formula2} 
we need the following 
results, Lemma~\ref{lem:derivative_rule}
and Lemma~\ref{lem:unif_conv_der}.
Lemma~\ref{lem:derivative_rule} establishes a
representation for the 
Malliavin derivative of a solution
of the SDE in \eqref{eq:dr_X}
under the global monotonicity 
assumption.
The proof of this lemma
is based on the ideas 
in \cite[§2.3.1]{Nualart1995}.

\begin{lemma}\label{lem:derivative_rule}
    Let $ d,m\in\N$, $c,T\in(0,\infty)$,
    let $\langle\cdot,\cdot\rangle\colon
    \R^d\times\R^d\to\R$ be the
    standard Euclidean scalar product
    on $\R^d$,
    let $\norm{\cdot}\colon\R^d\to[0,\infty)$
    be the standard Euclidean norm on $\R^d$,
    let $\norm{\cdot}_F\colon\R^{d\times m}\to[0,\infty)$ 
    be the Frobenius norm
    on $\R^{d\times m}$,
    let $(\Omega, \mathcal{F}, \mathbb{P}, (\mathbb{F}_s)_{s \in [0,T]})$ 
    be a filtered probability     
    space satisfying the usual   
    conditions,
    let $W=(W^1,W^2,\ldots, W^m) 
    \colon$ 
    $[0,T]\times\Omega\to\R^m$ 
    be a standard $(\mathbb{F}_s)_{s \in [0,T]}$-Brownian motion,
    let $\mu=(\mu_1,\mu_2,\ldots,\mu_d)
    \in  C^{0,1}([0,T] \times \R^d, \R^d)$ and
    $\sigma=(\sigma_{ij})_{i\in\{1,2,\ldots,d\},
        j\in\{1,2,\ldots,m\}}
    \in C^{0,1} ([0,T] \times \R^d, \R^{d \times m})$
    satisfy for all $s\in[0,T]$, $x, y \in \R^d$
    that
    \begin{equation}
    \label{eq:dr_globmon}
    \max\{\langle x-y, \mu(s,x)-\mu(s,y)\rangle,
    \tfrac{1}{2 }\norm{\sigma(s,x)-\sigma(s,y)}^2_F\} 
    \leq \tfrac{c}{2} \norm{x-y}^2, 
    \end{equation}
    let $X = ((X^{(1)}_s,X^{(2)}_s, \ldots, 
    X^{(d)}_s))_{s\in [0,T]} \colon [0,T] \times \Omega \to \R^d$ 
    be an adapted stochastic process 
    with continuous sample paths
    satisfying that 
    $\E[\norm{X_0}^2]<\infty$
    and for all $s \in [0,T]$ it holds a.s.\! that
    \begin{equation}\label{eq:dr_X}
    X_{s} = X_0 + 
    \int \limits_0^s \mu(r, X_{r}) \,\d r 
    + \int \limits_0^s \sigma(r, X_{r}) \,\d W_r,
    \end{equation}
    and let $Y=(Y^{(i,j)}_s)_{i,j\in\{1,2,\ldots, d\}, s\in [0,T]}
    \colon[0,T]\times\Omega\to\R^{d\times d}$ be 
    an adapted stochastic process
    with continuous sample paths 
    satisfying that for all
    $s\in[0,T]$ 
    it holds a.s.\! that
    \begin{equation}\label{eq:dr_Y}
    \begin{split}
    Y^{(i,j)}_s 
    &= \delta_{ij}
    +\int_0^s \sum_{k=1}^d \Big(\frac{\partial \mu_i}{\partial x_k} \Big)(r,X_r)Y^{(k,j)}_r \d r
    +\sum_{l=1}^m \int_0^s \sum_{k=1}^d \Big(\frac{\partial\sigma_{il}}{\partial x_k} \Big)(r, X_r) Y^{(k,j)}_r \d W^l_r.
    \end{split}
    \end{equation}    
    Then
    there exists an adapted stochastic process
    $Y^{-1}\colon [0,T]\times\Omega\to\R^{d\times d}$
    with continuous sample paths
    satisfying that for all 
    $t\in[0,T]$, $s\in [t,T]$
    it holds a.s.\! that
    $Y_s Y^{-1}_s = Y^{-1}_s Y_s
    =I_d$
    and
    \begin{equation}\label{eq:dr_der}
    D_t X_s = Y_s Y^{-1}_t 
    \sigma(t, X_t). 
    \end{equation}
\end{lemma}

\begin{proof}[Proof of Lemma~\ref{lem:derivative_rule}]
    Throughout this proof let 
    $Z=(Z^{(i,j)}_s)_{i,j\in\{1,2,\ldots,d\}, s\in [0,T]}
    \colon [0,T]\times\Omega \to\R^{d\times d}$
    be an adapted stochastic process
    with continuous sample paths
    satisfying that 
    for all $i,j\in\{1,2,\ldots,d\}$,
    $s\in[0,T]$ it holds a.s.\! that
    \begin{equation}\label{eq:dr_Z}
    \begin{split}
    &Z^{(i,j)}_s 
    = \delta_{ij}
    -\int_0^s \sum_{\alpha=1}^d Z^{(i,\alpha)}_r 
    \bigg[ \Big(\frac{\partial \mu_\alpha}{\partial x_j} \Big)(r,X_r)\\
    &\qquad
    - \sum_{n=1}^m \sum_{p=1}^d \Big(\frac{\partial \sigma_{\alpha n}}{\partial x_p}\Big)(r,X_r) 
    \Big(\frac{\partial \sigma_{pn}}{\partial x_j}\Big)(r,X_r) \bigg] \d r\\
    &\qquad -\sum_{l=1}^m \int_0^s \sum_{\alpha=1}^d 
    Z^{(i,\alpha)}_r\Big(\frac{\partial\sigma_{\alpha l}}{\partial x_j} \Big)(r, X_r)  \d W^l_r.
    \end{split}
    \end{equation}
    Observe that \eqref{eq:dr_Y},
    \eqref{eq:dr_Z}, and Itô's lemma 
    ensure that 
    for all $i,k\in\{1,2,\ldots, d\}$,
    $s\in[0,T]$ it holds a.s.\! that
    \begin{equation}\label{eq:dr_ZY}
    \begin{split}
    &\sum_{j=1}^d Z^{(i,j)}_s Y^{(j,k)}_s\\
    &=\sum_{j=1}^d\bigg[\delta_{ij}\delta_{jk}
    + \int_0^s Z^{(i,j)}_r \sum_{\alpha=1}^d  \Big(\frac{\partial \mu_j}{\partial x_\alpha} \Big)(r,X_r)Y^{(\alpha,k)}_r \d r\\
    &\quad +\sum_{l=1}^m \int_0^s Z^{(i,j)}_r 
    \sum_{\alpha=1}^d
    \Big(\frac{\partial\sigma_{jl}}{\partial x_\alpha} \Big)(r, X_r) Y^{(\alpha,k)}_r \d W^l_r\\
    &\quad -\int_0^s \sum_{\alpha=1}^d Z^{(i,\alpha)}_r 
    \bigg[ \Big(\frac{\partial \mu_\alpha}{\partial x_j} \Big)(r,X_r)
    - \sum_{n=1}^m \sum_{p=1}^d \Big(\frac{\partial \sigma_{\alpha n}}{\partial x_p}\Big)(r,X_r) 
    \Big(\frac{\partial \sigma_{pn}}{\partial x_j}\Big)(r,X_r) \bigg] Y^{(j,k)}_r \d r\\
    &\quad -\sum_{l=1}^m \int_0^s \sum_{\alpha=1}^d 
    Z^{(i,\alpha)}_r\Big(\frac{\partial\sigma_{\alpha l}}{\partial x_j} \Big)(r, X_r)
    Y^{(j,k)}_r  \d W^l_r\\
    &\quad -  \int_0^s \sum_{\alpha=1}^d Z^{(i,\alpha)}_r \sum_{n=1}^m \sum_{p=1}^d
    \Big(\frac{\partial \sigma_{\alpha n}}{\partial x_j}\Big)(r,X_r)
    \Big(\frac{\partial \sigma_{j n}}{\partial x_p}\Big)(r,X_r) Y^{(p,k)}_r \d r
    \bigg]\\
    &= \delta_{ik}.
    \end{split}
    \end{equation}
    Combining this with, e.g.,
    \cite[Proposition 4.1]{Carrell2017}
    and \cite[Proposition 4.4]{Carrell2017}
    (applied for every $s\in[0,T]$ 
    with $n\curvearrowleft d$,
    $A\curvearrowleft Y_s$,
    $B\curvearrowleft Z_s$
    in the notation of 
    \cite[Proposition 4.4]{Carrell2017})
    implies that 
    for all $s\in[0,T]$
    it holds a.s.\! that 
    $Y_s Z_s = I_d = Z_s Y_s$.
    Next note that
    \cite[Corollary 3.5]{IdRS2018}
    (applied with
    $p\curvearrowleft 2$,
    $\theta\curvearrowleft (\Omega\ni\omega\mapsto X_0(\omega) \in\R^d)$,
    $b\curvearrowleft ([0,T]\times\Omega\times \R^d\ni(s,\omega,x)\mapsto \mu(s,x)\in\R^d)$, 
    $\sigma\curvearrowleft ([0,T]\times\Omega\times \R^d\ni(s,\omega,x)\mapsto \sigma(s,x)\in\R^{d\times d})$
    in the notation of \cite[Corollary 3.5]{IdRS2018}),
    the assumption that
    $\mu\in C^{0,1}([0,T] \times \R^d, \R^d)$,
    $\sigma\in C^{0,1} ([0,T] \times \R^d, 
    \R^{d \times m})$,
    and \eqref{eq:dr_globmon}
    demonstrate that for all
    $i\in\{1,2,\ldots, d\}$,
    $j\in\{1,2,\ldots,m\}$,
    $t\in[0,T]$, $s\in[t,T]$
    it holds a.s.\! that 
    \begin{equation}\label{eq:dr_DX}
    \begin{split}
    D^j_t X^{(i)}_s
    = &\sigma_{ij}(t,X_t)
    +\int_t^s \sum_{k=1}^d \Big(\frac{\partial \mu_i}{\partial x_k} \Big)(r,X_r) D^j_t X^{(k)}_r \d r\\
    &\quad + \sum_{l=1}^m \int_t^s 
    \sum_{k=1}^d \Big(\frac{\partial \sigma_{il}}{\partial x_k} \Big)(r,X_r) D^j_t X^{(k)}_r \d W^l_r.
    \end{split}
    \end{equation}
    Moreover, observe that \eqref{eq:dr_Y}
    and the fact that
    for all $s\in[0,T]$
    it holds a.s\! that 
    $Y_s Z_s = I_d = Z_s Y_s$
    imply that for all 
    $i\in\{1,2,\ldots,d\}$,
    $j\in\{1,2,\ldots,m\}$,
    $t\in[0,T]$, $s\in[t,T]$
    it holds a.s.\! that
    \begin{equation}
    \begin{split}
    &\sigma_{ij}(t,X_t)
    +\int_t^s \sum_{k=1}^d \Big(\frac{\partial \mu_{i}}{\partial x_k}\Big)(r,X_r) \sum_{n=1}^d\sum_{p=1}^d Y^{(k,n)}_r Z^{(n,p)}_t\sigma_{pj}(t,X_t) \d r\\
    &\quad +\sum_{l=1}^m 
    \int_t^s\sum_{k=1}^d \Big(\frac{\partial \sigma_{il}}{\partial x_k}\Big)(r,X_r) \sum_{n=1}^d\sum_{p=1}^d Y^{(k,n)}_r Z^{(n,p)}_t \sigma_{pj}(t,X_t) \d W^l_r\\
    &= \sigma_{ij}(t,X_t)
    +\sum_{n=1}^d\sum_{p=1}^d [Y^{(i,n)}_s -
    Y^{(i,n)}_t]Z^{(n,p)}_t \sigma_{pj}(t,X_t)\\
    &= \sum_{n=1}^d\sum_{p=1}^d
    Y^{(i,n)}_t Z^{(n,p)}_t  
    \sigma_{pj}(t,X_t) 
    +\sum_{n=1}^d\sum_{p=1}^d [Y^{(i,n)}_s -
    Y^{(i,n)}_t]Z^{(n,p)}_t \sigma_{pj}(t,X_t)\\
    &=\sum_{n=1}^d \sum_{p=1}^d Y^{(i,n)}_s Z^{(n,p)}_t \sigma_{pj}(t,X_t).
    \end{split}
    \end{equation}
    This, \eqref{eq:dr_DX},
    and the fact that 
    linear SDEs are pathwise 
    unique
    establish \eqref{eq:dr_der}.
    The proof of Lemma~\ref{lem:derivative_rule}
    is thus complete.
\end{proof}

The following Lemma
is a well-known result on the connection between
uniform convergence and convergence of the derivative,
generalized to $d$ dimensions.
Lemma~\ref{lem:unif_conv_der} is a 
direct consequence of 
\cite[Theorem 7.17]{R2009}.

\begin{lemma}\label{lem:unif_conv_der}
    Let $d\in \N$,
    let $O\subseteq \R^d$ open,
    let $f_n\colon\R^d\to \R$, $n\in\N_0$,
    satisfy 
    that $(f_n)_{n\in\N}$ converges
    pointwise to $f_0$ on $O$
    and $(\nabla f_n)_{n\in\N}$
    converges uniformly on $O$.
    Then it holds that
    $(f_n)_{n\in\N}$ 
    converges uniformly on $O$
    and for all $x\in O$
    it holds that 
    $(\nabla f_0)(x)= \lim_{n\to\infty}
    (\nabla f_n)(x)$.
\end{lemma}

\begin{proof} [Proof of Lemma~\ref{lem:unif_conv_der}]
    Throughout this proof let
    $\mathbf{e}_1,\mathbf{e}_2,
    \dots,\mathbf{e}_d\in\R^d$ 
    satisfy that $\mathbf{e}_1=(1,0,\dots,0)$, 
    $\mathbf{e}_2=(0,1,0,\dots,0)$, $\dots$,
    $\mathbf{e}_d=(0,\dots,0,1)$.
    For every $j\in \{1,2,\ldots,d\}$, $x\in O$
    let $y^x_j\in O$, 
    $\lambda^x_j\in (0,1)$ satisfy
    $y_j^x-x=\lambda^x_j\mathbf{e}_j$ and 
    for all $t\in [0,1]$ that
    $x+\lambda^x_j \mathbf{e}_j \in O$
    and let $g^x_j\colon [0,1]\to \R^d$
    satisfy for all $t\in [0,1]$
    that $g^x_j(t)= (1-t)x+ty^x_j$.
    Note that for all 
    $j\in \{1,2\ldots,d\}$, $x\in O$
    it holds that
    \begin{equation}
    g^x_j(t)= (1-t)x+ty^x_j
    = x+t(y^x_j-x)
    = x+t\lambda^x_j \mathbf{e}_j.
    \end{equation}
    Furthermore, observe that
    the assumption that
    $(f_n)_{n\in\N}$ converges pointwise
    to $f_0$ on $O$,
    the hypothesis that
    $(\nabla f_n)_{n\in\N}$
    converges uniformly on $O$,
    and
    the fact that for all 
    $j\in \{1,2,\ldots,d\}$, $t\in [0,1]$,
    $x\in O$
    it holds that
    $g^x_j(t)\in O$ 
    ensure that
    for all $j\in \{1,2,\ldots,d\}$,
    $x\in O$
    it holds that
    $(f_n\circ g^x_j)_{n\in\N}$ converges
    pointwise to $f_0\circ g^x_j$ on $[0,1]$
    and $(\nabla f_n\circ g^x_j)_{n\in\N}$
    converges uniformly on $[0,1]$.   
    This and
    \cite[Theorem 7.17]{R2009}
    (applied for every $j\in \{1,2, \ldots,d\}$, 
    $x\in O$
    with 
    $a\curvearrowleft 0$,
    $b\curvearrowleft 1$,
    $(f_n)_{n\in\N}\curvearrowleft
    (f_n\circ g^x_j)_{n\in\N}$
    in the notation of 
    \cite[Theorem 7.17]{R2009})
    demonstrate that 
    for all $j\in \{1,2, \ldots,d\}$,
    $t\in [0,1]$, $x\in O$
    it holds that
    $(f_n\circ g^x_j)_{n\in\N}$
    converges uniformly to
    $f\circ g^x_j$ and
    \begin{equation}
    (f\circ g^x_j)'(t)
    = \lim_{n\to\infty} 
    (f_n\circ g^x_j)'(t).
    \end{equation}
    The chain rule
    and the fact that for all
    $j\in \{1,2,\ldots,d\}$
    $t\in [0,1]$, $x\in O$ it holds that
    $(g^x_j)'(t)= \lambda^x_j \mathbf{e}_j$ therefore 
    demonstrate that
    for all  $j\in \{1,2, \ldots,d\}$,
    $t\in [0,t]$, $x\in O$ it holds that
    \begin{equation}
    \label{eq:grad_f_limit}
    \begin{split}
    &(\nabla f)(g^x_j(t))\lambda^x_j \mathbf{e}_j
    =(\nabla f)(g^x_j(t))(g^x_j)'(t)
    =(f\circ g^x_j)'(t)\\
    &= \lim_{n\to\infty} 
    (f_n\circ g^x_j)'(t)
    = \lim_{n\to\infty}
    (\nabla f_n)(g^x_j(t))(g^x_j)'(t)
    = \lim_{n\to\infty}
    (\nabla f_n)(g^x_j(t))\lambda^x_j \mathbf{e}_j.
    \end{split}
    \end{equation}
    Hence, we obtain that
    for all $j\in\{1,2, \ldots, d\}$, $x\in O$ 
    it holds that
    \begin{equation}
    (\nabla f)(g^x_j(0))\lambda^x_j \mathbf{e}_j
    =(\nabla f)(x) \lambda^x_j \mathbf{e}_j
    = \lim_{n\to\infty} (\nabla f_n)(x)
    \lambda^x_j \mathbf{e}_j.
    \end{equation}   
    This proves that
    for all $j\in \{1,2,\ldots, d\}$, $x\in O$
    it holds that
    \begin{equation}
    \Big(\frac{\partial f}{\partial x_j}\Big)(x)
    =\lim_{n\to \infty}
    \Big(\frac{\partial f_n}{\partial x_j}\Big)(x).
    \end{equation}
    The proof of Lemma~\ref{lem:unif_conv_der}
    is thus complete.
\end{proof}

The following theorem,
Theorem~\ref{thm:BEL_formula2}, 
is a
Bismut-Elworthy-Li
type formula for continuous 
$L^2$-functions under global monotonicity
assumption on the coefficients of the
considered SDE.

\begin{theorem}\label{thm:BEL_formula2}
    Let $d \in \N$, 
    $c\in [0,\infty)$,
    $\alpha, T \in (0, \infty)$, $t\in[0,T)$,
    let $O\subseteq \R^d$ be an open set,
    let $\langle\cdot,\cdot\rangle\colon
    \R^d\times\R^d\to\R$ be the standard
    Euclidean scalar product on $\R^d$,
    let $\norm{\cdot}\colon\R^d\to[0,\infty)$
    be the standard Euclidean norm on $\R^d$,
    let $\norm{\cdot}_F\colon \R^{d\times d}
    \to [0,\infty)$ be the Frobenius norm
    on $\R^{d\times d}$,
    let $(\Omega, \mathcal{F}, \mathbb{P}, (\mathbb{F}_s)_{s \in [0,T]})$ be a filtered probability space satisfying the usual conditions,
    let $W \colon [0,T] \times \Omega \to \R^d$ 
    be a standard $(\mathbb{F}_s)_{s \in [0,T]}$-Brownian motion,
    let $\mu \in  C^{0,1}([0,T] \times O, \R^d)$, 
    $\sigma \in C^{0,1} ([0,T] \times O, \R^{d \times d})$
    satisfy for all $s\in[t,T]$, $x, y \in O$,
    $v\in \R^d$
    that
    \begin{equation}
    \label{eq:globmon_BEL}
    \max\{\langle x-y,\mu(s,x)-\mu(s,y)\rangle, 
    \tfrac{1}{2}\norm{\sigma(s,x)-\sigma(s,y)}_F^2\}
    \leq \tfrac{c}{2}\norm{x-y}^2
    \end{equation}
    and $v^* \sigma(s,x) (\sigma(s,x))^* v \geq \alpha \norm{v}^2$, 
    for every $x \in O$ let
    $X^x = (X^{x}_{s})_{s \in [t,T]} \colon [t,T] \times \Omega \to O$ be an 
    $(\mathbb{F}_s)_{s \in [t,T]}$-adapted stochastic process with continuous sample paths satisfying that for all $s \in [t,T]$ it holds a.s.\! that
    \begin{equation}\label{eq:X_BEL}
    X^x_{s} = x + \int \limits_t^s \mu(r, X^x_{r}) \,\d r + \int \limits_t^s \sigma(r, X^x_{r}) \,\d W_r,
    \end{equation}
    assume for all
    $\omega \in \Omega$
    that
    $\left([t,T] \times O \ni (s,x) \mapsto X^x_{s}(\omega) \in \R^d \right) \in C^{0,1}([t,T] \times O, O)$,
    let $f\in C(O,\R)\cap L^2(O,\R)$,
    let $u\colon O \to \R$ satisfy for all
    $x \in O$ that
    \begin{equation}\label{u_BEL}
    u(x) = \E [f(X^x_{T})],
    \end{equation}
    and
    for every $x \in O$ let
    $Z^x = (Z^x_{s})_{s \in (t,T]} \colon (t,T] \times \Omega \to \R^d$ 
    be an $(\mathbb{F}_s)_{s \in (t,T]}$-adapted stochastic process 
    with continuous sample paths
    satisfying that 
    for all $s \in (t,T]$ it holds a.s.\! that
    \begin{equation}\label{eq:Z_BEL}
    Z^x_{s} = \frac{1}{s-t} \int \limits_t^s
    (\sigma(r,   X^x_{r}))^{-1} \; \Big(\frac{\partial}{\partial x} X^x_{r}\Big) \,\d W_r.
    \end{equation} 
    Then 
    \begin{enumerate}[label=(\roman*)]
        \item\label{it:f_X_Z_bd_BEL}
        for all 
        $x \in O$
        it holds that
        \begin{equation}
        \E\left[\norm{f(X^x_{T})Z^x_{T}}\right]
        <\infty,
        \end{equation}
        \item\label{it:u_diff_BEL}
        it holds that
        $u \in C^{1}(O, \R)$,
        and
        \item\label{it:u_der_BEL}
        for all $x \in O$ 
        it holds that
        \begin{equation}\label{eq:u_der_BEL}
        \big(\nabla u \big)(x)
        = \E\! \left[f(X^x_{T}) Z^x_{T} \right].
        \end{equation}
    \end{enumerate}
\end{theorem}

\begin{proof}[Proof of Theorem~\ref{thm:BEL_formula2}]
    First note that
    \cite[Lemma 3.2 (iv)]{HP2023}
    proves that
    for all 
    $x \in O$
    it holds that
    \begin{equation}
    \label{eq:Z_L2_BEL}
    \E\Big[\norm{Z^x_{T}}^2\Big]
    \leq\frac{d}{\alpha(T-t)^2}
    \int_t^T  \exp(2(r-t)c) \,\d r.
    \end{equation}
    Next observe that
    the assumption that
    $f\in L^2(O,\R)$
    ensures that
    for all $x\in O$
    it holds that
    \begin{equation}
    \E\left[\abs{f(X^x_T)}^2\right]
    \leq \sup_{y\in\R^d} \abs{f(y)}^2
    \leq \int_{y\in\R^d} \abs{f(y)}^2 \d y
    < \infty.
    \end{equation}
    The Cauchy-Schwarz inequality
    and 
    \eqref{eq:Z_L2_BEL}
    hence
    show that for all
    $x\in O$
    it holds that
    \begin{equation}
    \begin{split}
    \E\left[\norm{f(X^x_T)Z^x_T}\right]
    \leq \E\left[\abs{f(X^x_T)}\norm{Z^x_T}\right]
    \leq \left(\E\left[\abs{f(X^x_T)}^2\right]\right)^{\frac{1}{2}}
    \left(\E\left[\norm{Z^x_T}^2\right]\right)^{\frac{1}{2}}
    <\infty.
    \end{split}
    \end{equation}
    This establishes 
    item~\ref{it:f_X_Z_bd_BEL}.
    Next we prove
    items~\ref{it:u_diff_BEL}
    and \ref{it:u_der_BEL} in
    two steps.\\
    \underline{Step 1:}
    In addition to the assumptions of Theorem~\ref{thm:BEL_formula2}
    we assume in step 1 that 
    $f\in C_c^\infty(O,\R)$.
    Observe that the assumption
    that $f\in C_c^\infty(O,\R)$ 
    ensures that there exists
    $L\in(0,\infty)$ which satisfies
    for all $x,y\in O$ that
    \begin{equation}\label{eq:f_lip_BEL}
    \abs{f(x)-f(y)}
    \leq L \norm{x-y}.
    \end{equation}
    This implies that for all
    $x\in O$
    it holds that 
    $\norm{(\nabla f)(x)}\leq L$. 
    The chain rule,
    the fundamental theorem of calculus,
    Jensen's inequality, 
    Fubini's theorem, 
    and \cite[Lemma 3.2 (ii)]{HP2023}
    hence
    demonstrate that
    for all $h\in\R\setminus \{0\}$,
    $j\in\{1,2,\ldots, d\}$, $x\in O$
    it holds that
    \begin{equation}
    \label{eq:dif_quot_bd2}
    \begin{split}
    &\E\bigg[\Big\lvert\frac{f(X^{x+h\mathbf{e}_j}_T)-f(X^x_T)}{h}\Big\rvert^2\bigg]
    =\E\bigg[\Big\lvert\int_0^1 
    (\nabla f)(X^{x+\lambda h \mathbf{e}_j}_T) 
    \Big(\frac{\partial}{\partial x_j}X^{x+\lambda h \mathbf{e}_j}_T\Big) 
    \d \lambda\Big\rvert^2\bigg]\\
    &\leq \E\bigg[\int_0^1 
    \abss{ (\nabla f)(X^{x+\lambda h \mathbf{e}_j}_T) 
        \Big(\frac{\partial}{\partial x_j}X^{x+\lambda h \mathbf{e}_j}_T \Big)}^2  
    \d \lambda\bigg]\\
    &= \int_0^1\E\bigg[ 
    \abss{ (\nabla f)(X^{x+\lambda h \mathbf{e}_j}_T) 
        \Big(\frac{\partial}{\partial x_j}X^{x+\lambda h \mathbf{e}_j}_T\Big) 
    }^2\bigg]
    \d \lambda\\
    &\leq \int_0^1\E\bigg[ 
    \normm{\big(\nabla f\big)(X^{x+\lambda h \mathbf{e}_j}_T)}^2 
    \normmm{\frac{\partial}{\partial x_j}X^{x+\lambda h \mathbf{e}_j}_T 
    }^2 \bigg]
    \d \lambda 
    \leq L^2 \int_0^1\E\bigg[  
    \normmm{\frac{\partial}{\partial x_j}X^{x+\lambda h \mathbf{e}_j}_T}^2 \bigg]
    \d \lambda\\
    &\leq L^2 \int_0^1
    \exp(2(T-t)c)
    \d \lambda
    = L^2 \exp(2(T-t)c).
    \end{split}
    \end{equation}
    In addition, observe that 
    the chain rule,
    the assumption that $f\in C^\infty(O,\R)$,
    and 
    the fact that for all $\omega\in\Omega$,
    $s\in[t,T]$
    it holds that
    $(O \ni x \mapsto X^x_s(\omega)\in\R^d)\in C^1(O,O)$
    ensure that for all
    $j\in\{1,2,\ldots,d\}$, $x\in O$,
    $\omega\in\Omega$ 
    it holds a.s.\! that
    \begin{equation}
    \label{eq:diff_quot_pt_conv1_2}
    \begin{split}
    &\lim_{\R\setminus \{0\}\ni h\to 0}
    \frac{f(X^{x+h\mathbf{e}_j}_T(\omega))-f(X^x_T(\omega))}{h}
    =\frac{\partial}{\partial x_j}
    (f(X^x_T(\omega)))\\
    &=(\nabla f)(X^x_T(\omega))
    \Big(\frac{\partial}{\partial x_j}(X^x_T(\omega))\Big).
    \end{split}
    \end{equation}
    This, \eqref{eq:dif_quot_bd2}, and
    the Vitali convergence theorem
    demonstrate that
    for all  $j\in\{1,2,\ldots,d\}$, $x\in O$
    it holds that
    \begin{equation}
    \begin{split}\label{eq:u_limit_diffquot2}
    &0=\lim_{\R\setminus \{0\}\ni h \to 0}
    \E\bigg[\Big| 
    \frac{f(X^{x+h \mathbf{e}_j}_T)
        -f(X^x_T)}{h}
    - (\nabla f)(X^x_T)
    \Big(\frac{\partial}{\partial x_j}X^x_T\Big) \Big| \bigg]\\
    &\geq \limsup_{\R\setminus \{0\}\ni h \to 0}
    \bigg| \E\bigg[ 
    \frac{f(X^{x+h \mathbf{e}_j}_T)
        -f(X^x_T)}{h}
    - (\nabla f)(X^x_T)
    \Big(\frac{\partial}{\partial x_j}X^x_T\Big)  \bigg]\bigg|\\
    &= \limsup_{\R\setminus \{0\}\ni h \to 0}
    \bigg|
    \frac{u(x+h\mathbf{e}_j)-u(x)}{h}
    -\E\Big[(\nabla f)(X^x_T)
    \Big(\frac{\partial}{\partial x_j}X^x_T\Big)\Big]\bigg|
    \geq 0.
    \end{split}
    \end{equation}
    This proves that
    for all $j\in\{1,2,\ldots,d\}$, 
    $x\in O$
    it holds that
    \begin{equation}\label{eq:u_limit_diffquot2_2}
    \lim_{\R\setminus \{0\}\ni h \to 0} \frac{u(x+h\mathbf{e}_j)-u(x)}{h}
    = \E\Big[(\nabla f)(X^x_T)
    \Big(\frac{\partial}{\partial x_j}X^x_T\Big)\Big].
    \end{equation}
    In addition, note that 
    \cite[Lemma 3.2 (ii)]{HP2023}
    and
    the fact that for all $x\in O$
    it holds that 
    $\norm{(\nabla f)(x)}\leq L$
    demonstrate that 
    for all $h\in(0,\infty)$,
    $j\in\{1,2,\ldots, d\}$, $x\in O$
    it holds that
    \begin{equation}\label{eq:der_f_X_L2_2}
    \begin{split}
    &\E\bigg[\normmm{(\nabla f)(X^{x+h\mathbf{e}_j}_T)\Big(\frac{\partial}{\partial x_j}X^{x+h\mathbf{e}_j}_T\Big)}^2 \bigg]
    \leq \E\bigg[\normmm{(\nabla f)(X^{x+h\mathbf{e}_j}_T)}^2
    \normmm{\frac{\partial}{\partial x_j}X^{x+h\mathbf{e}_j}_T}^2 \bigg]\\
    &\leq d L^2 \E\bigg[
    \normmm{\frac{\partial}{\partial x_j}X^{x+h\mathbf{e}_j}_T}^2 \bigg]
    \leq L^2 \exp(2(T-t)c).\\
    \end{split}
    \end{equation}
    Moreover, observe that the
    assumption that $f\in C^\infty(O,\R)$
    and the fact that for all $s\in[t,T]$,
    $\omega\in\Omega$
    it holds that 
    $(O \ni x \mapsto X^x_s(\omega)\in\R^d)\in C^1(O,O)$
    show that for all
    $j\in\{1,2,\ldots, d\}$, $x\in O$,
    $\omega\in\Omega$
    it holds that
    \begin{equation}
    \begin{split}\label{eq:der_f_X_pt_conv1_2}
    \lim_{\R\ni h \to 0} \left[
    (\nabla f)(X^{x+h\mathbf{e}_j}_T(\omega))\Big(\frac{\partial}{\partial x_j}X^{x+h\mathbf{e}_j}_T(\omega)\Big)\right]
    = (\nabla f)(X^x_T(\omega))  \Big(\frac{\partial}{\partial x_j}X^{x}_T(\omega)\Big).
    \end{split}
    \end{equation}
    Combining this, \eqref{eq:der_f_X_L2_2},
    and the Vitali convergence theorem 
    shows that for all $x\in O$, 
    $j\in\{1,2,\ldots,d\}$ 
    it holds that
    \begin{equation}
    \begin{split}
    &0=
    \lim_{\R\ni h\to 0}
    \E\bigg[ \Big\lvert (\nabla f)(X^{x+h\mathbf{e}_j}_T)\Big(\frac{\partial}{\partial x_j}X^{x+h\mathbf{e}_j}_T\Big)
    -(\nabla f)(X^x_T)  \Big(\frac{\partial}{\partial x_j}X^{x}_T\Big)\Big\rvert \bigg]\\
    &\geq \limsup_{\R\ni h\to 0}
    \Big\lvert\E\Big[ (\nabla f)(X^{x+h\mathbf{e}_j}_T)\Big(\frac{\partial}{\partial x_j}X^{x+h\mathbf{e}_j}_T\Big)
    -(\nabla f)(X^x_T)  \Big(\frac{\partial}{\partial x_j}X^{x}_T\Big) \Big]\Big\rvert\\
    &= \limsup_{\R\ni h\to 0}
    \Big\lvert\E\Big[ (\nabla f)(X^{x+h\mathbf{e}_j}_T)\Big(\frac{\partial}{\partial x_j}X^{x+h\mathbf{e}_j}_T\Big)\Big]
    -\E\Big[(\nabla f)(X^x_T)  \Big(\frac{\partial}{\partial x_j}X^{x}_T\Big) \Big]\Big\rvert
    \geq 0.
    \end{split}
    \end{equation}
    This proves that
    for all $j\in\{1,2,\ldots,d\}$
    it holds that
    $(O \ni x\mapsto (\nabla f)(X^x_T)(\frac{\partial}{\partial x_j}X^{x}_T)$ $\in L^1(\mathbb{P},\R)) \in C^0(O,L^1(\mathbb{P},\R))$.
    This and \eqref{eq:u_limit_diffquot2_2}
    demonstrate that for all
    $j\in\{1,2,\ldots,d\}$, $x\in O$
    it holds that $u\in C^1(O,\R)$
    and
    \begin{equation}\label{eq:u_der4_2}
    \frac{\partial u}{\partial x_j}(x)
    =\E\bigg[(\nabla f)(X^{x}_T)\Big(\frac{\partial}{\partial x_j}X^{x}_T\Big)\bigg].
    \end{equation}
    Next observe that
    \cite[Corollary 3.5]{IdRS2018}
    (applied with $T\curvearrowleft T-t$, 
    $p\curvearrowleft 2$,
    $m\curvearrowleft d$,
    $\theta\curvearrowleft (\Omega\ni\omega\mapsto x \in O)$,
    $b\curvearrowleft ([0,T-t]\times\Omega\times O \ni(s,\omega,x)\mapsto \mu(t+s,x)\in\R^d)$, 
    $\sigma\curvearrowleft ([0,T-t]\times\Omega\times  O \ni(s,\omega,x)\mapsto \sigma(t+s,x)\in\R^{d\times d})$
    in the notation of \cite[Corollary 3.5]{IdRS2018}),
    the assumption that
    $\mu\in C^{0,1}([0,T]\times O, \R^d)$, 
    $\sigma\in C^{0,1}([0,T]\times O,\R^{d\times d})$,
    and \eqref{eq:globmon_BEL}
    ensure that for all $s\in[t,T]$, $x\in O$
    it holds that 
    $X^{x}_{s} \in \mathbb{D}^{1,2}$.
    In addition, note that 
    Lemma~\ref{lem:derivative_rule}
    (applied for every $x\in O$
    with $m\curvearrowleft d$,
    $T\curvearrowleft T-t$, 
    $X_0 \curvearrowleft x$,
    $X\curvearrowleft ([0,T-t]\times\Omega
    \ni(s,\omega)\mapsto X^x_{t+s} \in O)$,
    $Y\curvearrowleft ([0,T-t]\times\Omega\ni
    (s,\omega)\mapsto (\frac{\partial }{\partial x}X^x_{t+s})(\omega)\in\R^{d\times d})$
    in the notation of Lemma~\ref{lem:derivative_rule})
    shows that
    for all $x \in O$
    there exists a stochastic process
    $(\frac{\partial}{\partial x} X^x)^{-1}
    =((\frac{\partial}{\partial x} X^x_r)^{-1})_{r\in[t,T]}\colon
    [t,T]\times \Omega \to \R^{d\times d}$
    which satisfies that for all
    $r \in [t,T]$, $s\in[t,r]$
    it holds a.s.\! that 
    $(\frac{\partial}{\partial x} X^x_r)(\frac{\partial}{\partial x} X^x_r)^{-1}
    = I_d
    =(\frac{\partial}{\partial x} X^x_r)^{-1}(\frac{\partial}{\partial x} X^x_r)$
    and 
    \begin{equation}
    D_s X^x_{r} = \Big(\frac{\partial}{\partial x}
    X^x_{r}\Big) \Big(\frac{\partial}{\partial x} X^x_{s}\Big)^{-1} \sigma(s, X^x_{s}).
    \end{equation}
    This, 
    \cite[Proposition 1.2.2]{Nualart1995}
    (applied for every $r\in[t,T]$, $x\in O$ 
    with $m\curvearrowleft d$,
    $\varphi\curvearrowleft f$, $p\curvearrowleft 2$, $F\curvearrowleft X^x_r$
    in the notation of \cite[Proposition 1.2.2]{Nualart1995}),
    and the assumption that 
    $f\in C^\infty_c(O,\R)$
    demonstrate that 
    for all $r \in [t,T]$, $x \in O$
    it holds that
    $f(X^x_{r})\in \mathbb{D}^{1,2}$
    and 
    for all $r \in [t,T]$, $s\in[t,r]$, $x \in O$
    it holds a.s.\! that
    \begin{equation}\label{eq:malliavin_der1_2}
    D_s \left( f(X^x_{r}) \right)
    = (\nabla f)(X^x_{r}) D_s X^x_{r} 
    =  (\nabla f)(X^x_{r}) \Big(\frac{\partial}{\partial x} X^x_{r}\Big)  \Big(\frac{\partial}{\partial x} X^x_{s}\Big)^{-1} \sigma(s, X^x_{s}).
    \end{equation}
    Integrating both sides of \eqref{eq:malliavin_der1_2} shows  that
    for all $r \in (t,T]$, $x \in O$
    it holds a.s.\! that
    \begin{equation}\label{eq:Z_int_f2}
    (\nabla f)(X^x_{r}) \Big(\frac{\partial}{\partial x} X^x_{r}\Big) 
    = \frac{1}{r-t} \int \limits_t^r D_s \left( f(X^x_{r}) \right) (\sigma(s, X^x_{s}))^{-1} \Big(\frac{\partial}{\partial x} X^x_{s}\Big) \,\d s.
    \end{equation}
    Next note that
    the fact that 
    $\sigma\in C^{0,0} ([0,T] \times O, \R^{d \times d})$ implies that 
    $\sigma^{-1}\in  C^{0,0} ([0,T] \times O, \R^{d \times d})$.
    The assumption that for every $x\in O$
    it holds that $X^x$ is an $(\mathbb{F}_s)_{s \in [t,T]}$-adapted process
    therefore shows that
    for all $x\in O$
    it holds that 
    $((\sigma(s,X^x_s))^{-1}(\frac{\partial}{\partial x} X^x_s))_{s\in[t,T]}$ is an $(\mathbb{F}_s)_{s \in [t,T]}$-adapted process.
    Combining this with 
    \cite[Proposition 1.3.4]{Nualart1995}
    (applied for every $j\in\{1,2,\ldots, d\}$
    with $u\curvearrowleft
    ([0,T]\times\Omega\ni (s,\omega)\mapsto
    (\sigma({t+s(T-t)},X^x_{t+s(T-t)}))^{-1} $
    $\cdot (\frac{\partial}{\partial x} X^x_{t+s(T-t)})\mathbf{e}_j \in\R^d)$
    in the notation of \cite[Proposition 1.3.4]{Nualart1995})
    and \eqref{eq:Z_L2_BEL} demonstrates
    that for all $r\in[t,T]$, $x\in O$
    it holds a.s.\! that
    \begin{equation}
    \int \limits_t^r (\sigma(s, X^x_{s}))^{-1} \Big(\frac{\partial}{\partial x} X^x_{s} \Big) \,\delta W_s
    = \int \limits_t^r (\sigma(s, X^x_{s}))^{-1} \Big(\frac{\partial}{\partial x} X^x_{s} \Big) \,\d W_s.
    \end{equation}
    The duality 
    property of the Skorohod integral
    (cf., e.g., \cite[Definition 1.3.1 (ii)]{Nualart1995}
    (applied with $T\curvearrowleft T-t$,
    $F\curvearrowleft (\Omega\ni\omega\mapsto f(X^x_{T}(\omega))\in\R^d)$,
    $u \curvearrowleft ([0,T-t]\times\Omega\ni(s,\omega)\mapsto (\sigma(t+s,X^x_{t+s}))^{-1}(\frac{\partial}{\partial x}X^x_{t+s}))$))
    therefore shows that
    for all $x\in O$
    it holds that
    \begin{equation}
    \begin{split}
    &\E \left[\int \limits_t^T D_s \left( f(X^x_{T}) \right) (\sigma(s, X^x_{s}))^{-1} \Big(\frac{\partial}{\partial x} X^x_{s} \Big) \d s \right] \\
    &= \E \bigg[ f(X^x_{T}) \int \limits_t^T (\sigma(s, X^x_{s}))^{-1} \Big(\frac{\partial}{\partial x} X^x_{s} \Big) \,\d W_s \bigg].
    \end{split}
    \end{equation}
    This, 
    \eqref{eq:Z_BEL}, and
    \eqref{eq:Z_int_f2}
    ensure that for all
    $x \in O$
    it holds that
    \begin{equation}
    \label{eq:exp_f_Z2}
    \begin{split}
    &\E \left[ (\nabla f)(X^x_{T}) \Big(\frac{\partial}{\partial x} X^x_{T} \Big) \right] \\
    &= \frac{1}{T-t} \E \left[\int \limits_t^T D_s \left( f(X^x_{T}) \right) (\sigma(s, X^x_{s}))^{-1} \Big(\frac{\partial}{\partial x} X^x_{s} \Big) \d s \right] \\
    &= \frac{1}{T-t}  \E \bigg[ f(X^x_{T}) \int \limits_t^T (\sigma(s, X^x_{s}))^{-1} \Big(\frac{\partial}{\partial x} X^x_{s} \Big) \,\d W_s \bigg]
    = \E \left[ f(X^x_{T}) Z^x_{T} \right].
    \end{split}
    \end{equation}
    Combining this with \eqref{eq:u_der4_2}   
    proves that $u\in C^1(O,\R)$
    and for all $x\in O$ 
    it holds that
    $(\nabla u)(x)
    = \E[f(X^x_T)Z^x_T]$.\\
    \underline{Step 2:}
    For the second step note that
    the fact that $C^\infty_c(O,\R)$
    is dense in $L^2(O,\R)$
    (cf., e.g., \cite[Proposition 8.17]{Folland1984})
    ensures that there exist
    $(f_n)_{n\in\N}\subseteq 
    C^\infty_c(O,\R)$
    which satisfy for all $x\in O$
    that 
    \begin{equation}\label{eq:f_n_conv}
    \limsup\nolimits_{n\to\infty}
    \abs{f_n(x)-f(x)}=0.
    \end{equation}
    For every $n\in\N$
    let $u_n\colon O\to\R$
    satisfy for all $x\in O$ that
    $u_n(x)=\E[f_n(X^x_T)]$.
    Observe that
    \eqref{eq:f_n_conv},
    the triangle inequality, 
    the dominated convergence theorem,
    and the assumption that
    $f\in L^2(O,\R)$
    show that
    for all $x\in O$
    it holds that
    \begin{equation}
    \begin{split}\label{eq:u_n_conv}
    &\limsup_{n\to\infty}
    \abs{u_n(x)-u(x)}
    = \limsup_{n\to\infty}
    \abss{\E[f_n(X^x_T)]-\E[f(X^x_T)]}\\
    &\leq \limsup_{n\to\infty}
    \E[\abs{f_n(X^x_T)-f(X^x_T)}]
    = \E\bigg[\limsup_{n\to\infty}
    \abs{f_n(X^x_T)-f(X^x_T)}\bigg]
    =0.
    \end{split}
    \end{equation}
    Next note that step 1
    demonstrates that for all
    $n\in\N$, $x\in O$ it holds that
    $u_n\in C^1(O,\R)$
    and 
    \begin{equation}\label{eq:der_u_n}
    (\nabla u_n)(x)
    = \E[f_n(X^x_T)Z^x_T].
    \end{equation}
    Combining this
    with the triangle inequality  
    and the Cauchy-Schwarz inequality
    proves that
    for all $n\in\N$, $x\in O$
    it holds that
    \begin{equation}
    \begin{split}
    &\normm{(\nabla u_n)(x)
        -\E[f(X^x_T)Z^x_T]}
    =\normm{\E[f_n(X^x_T)Z^x_T]
        -\E[f(X^x_T)Z^x_T]}\\
    &\leq \E[\abs{f_n(X^x_T)
        -f(X^x_T)}\norm{Z^x_T}]
    \leq \Big(\E[\abs{f_n(X^x_T)
        -f(X^x_T)}^2]\Big)^{\frac{1}{2}}
    \Big(\E[\norm{Z^x_T}^2]\Big)^{\frac{1}{2}}.
    \end{split}
    \end{equation}
    The assumption that $f\in C(O,\R)$,
    the assumption that
    $\sigma\in C([0,T]\times O,\R^{d\times d})$,
    and the fact that 
    $(f_n)_{n\in\N}\subseteq C^\infty(O,\R)$
    therefore ensure that for all
    compact $K\subseteq O$
    there exists $\tilde{x}\in K$
    which satisfies
    \begin{equation}
    \sup_{x\in K}
    \norm{(\nabla u_n)(x)
        -\E[f(X^x_T)Z^x_T]}
    \leq \Big(\E[\abs{f_n(X^{\tilde{x}}_T)
        -f(X^{\tilde{x}}_T)}^2]\Big)^{\frac{1}{2}}
    \Big(\E[\norm{Z^{\tilde{x}}_T}^2]\Big)^{\frac{1}{2}}.
    \end{equation}
    Hence, we obtain that
    $u_n$ converges uniformly to
    $(O\ni x\mapsto 
    \E[f(X^x_T)Z^x_T]\in\R^d)$
    on compact subsets of $O$.
    Combining this with 
    \eqref{eq:u_n_conv}  
    and Lemma~\ref{lem:unif_conv_der}
    proves that 
    for all $x\in O$
    it holds that
    \begin{equation}
    (\nabla u)(x)
    = \lim_{n\to\infty} (\nabla u_n)(x)
    = \lim_{n\to\infty} \E[f_n(X^x_T)Z^x_T]
    = \E[f(X^x_T)Z^x_T].
    \end{equation}
    This establishes items~\ref{it:u_diff_BEL}
    and \ref{it:u_der_BEL}.
    The proof of Theorem~\ref{thm:BEL_formula2} 
    is thus complete.
\end{proof}

\section{Existence and uniqueness result for viscosity solutions of semilinear PDEs with gradient-dependent nonlinearities}
\label{sec:vs_main}

In this section we use the results 
from Section~\ref{sec:vs} and
\ref{sec:BEL} to show
that 
the unique viscosity solution of semilinear PDEs
and the unique solution of their connected SFPEs coincide.
The following theorem proves exactly this
connection under 
differentiability 
and global monotonicity
assumptions on $\mu$ and $\sigma$
and a Lipschitz and continuity assumption
on $f$.
Theorem~\ref{thm:vs} 
extends \cite[Theorem 3.7]{beck2021nonlinear}
to PDEs with gradient-dependent 
nonlinearities.

\begin{theorem}\label{thm:vs}
    Let $d \in \N$, 
    $\alpha, b, c, K, L, T  \in (0, \infty)$,
    let $\langle\cdot,\cdot\rangle\colon
    \R^d\times\R^d\to\R$ be the standard
    Euclidean scalar product on $\R^d$,
    let $\norm{\cdot}\colon\R^d\to[0,\infty)$
    be the standard Euclidean norm on $\R^d$,
    let $\Vnorm{\cdot}\colon\R^{d+1}\to[0,\infty)$
    be the standard Euclidean norm on $\R^{d+1}$,
    let $\norm{\cdot}_F\colon \R^{d\times d}
    \to [0,\infty)$ be the Frobenius norm
    on $\R^{d\times d}$,
    let $O\subseteq \R^d$ be an
    open set,
    for every $r \in (0, \infty)$
    let $K_r\subseteq [0,T)$,  
    $O_r \subseteq O$
    satisfy
    $K_r=[0,\max\{T-\frac{1}{r},0\}]$
    and
    $O_r = \{  x \in O\colon \norm{x} \leq r \text{ and } \{ y \in \R^d\colon \norm{y-x} 
    < \frac{1}{r} \} \subseteq O \}$,
    let $(\Omega, \mathcal{F}, \mathbb{P}, (\mathbb{F}_s)_{s \in [0,T]})$ 
    be a filtered probability space
    satisfying the usual conditions,
    let $W \colon [0,T] \times \Omega \to \R^d$ 
    be a standard $(\mathbb{F}_s)_{s \in [0,T]}$-Brownian motion,
    let $\mu \in  C^{0,1}([0,T] \times O, \R^d)$, 
    $\sigma \in C^{0,1} ([0,T] \times O, \R^{d \times d})$
    satisfy
    for all 
    $s\in[0,T]$, $x, y \in O$, $v\in\R^d$
    that
    \begin{equation}
    \label{eq:mu_sigma_globmon3}
    \max\Big\{\langle x-y,\mu(s,x)-\mu(s,y)\rangle, 
    \tfrac{1}{2}\norm{\sigma(s,x)-\sigma(s,y)}_F^2\Big\}
    \leq \tfrac{c}{2}\norm{x-y}^2
    \end{equation}
    and
    $v^* \sigma(s,x) (\sigma(s,x))^* v \geq \alpha \norm{v}^2$,
    assume for all $r\in (0,\infty)$,
    $j\in\{1,2,\ldots, d\}$ that
    \begin{equation}     
    \label{eq:der_mu_sigma_loclip3}
    \begin{split}
    \sup \bigg(\bigg\{
    &\tfrac{\norm{\frac{\partial \mu}{\partial x}(t,x)-\frac{\partial \mu}{\partial x}(t,y)}_F
        +\norm{\frac{\partial \sigma}{\partial x_j}(t,x)
            -\frac{\partial \sigma}{\partial x_j}(t,y)}_F}{\norm{x-y}}
    \colon\\ 
    &\qquad t\in[0,T], x,y\in O_r, x\neq y \bigg\} 
    \cup \{0\} \bigg)
    <\infty,
    \end{split}
    \end{equation}
    for every  $t\in [0,T]$,
    $x \in O$ let
    $X^x_t = (X^{x}_{t,s})_{s \in [t,T]} \colon [t,T] \times \Omega \to O$   
    be an 
    $(\mathbb{F}_s)_{s \in [t,T]}$-adapted 
    stochastic process with continuous
    sample paths satisfying that for all 
    $s \in [t,T]$ it holds a.s.\! that
    \begin{equation}
    \label{eq:X_ito_process3}
    X^x_{t,s} = x + \int_t^s \mu(r, X^x_{t,r}) \,\d r 
    + \int_t^s \sigma(r, X^x_{t,r}) \,\d W_r,
    \end{equation}
    assume for all
    $t\in [0,T]$,
    $\omega \in \Omega$
    that
    $\left([t,T] \times O \ni (s,x) 
    \mapsto X^x_{t,s}(\omega) \in O \right) \in C^{0,1}([t,T] \times O, O)$,
    for every $t\in[0,T]$, 
    $x \in O$ let
    $Z^x_t = (Z^x_{t,s})_{s \in (t,T]} 
    \colon (t,T] \times \Omega \to \R^{d+1}$ 
    be an $(\mathbb{F}_s)_{s \in (t,T]}$-adapted
    stochastic process 
    with continuous sample paths
    satisfying that 
    for all $s \in (t,T]$ 
    it holds a.s.\! that
    \begin{equation}\label{eq:Z3}
    Z^x_{t,s} = 
    \begin{pmatrix}
    1\\
    \frac{1}{s-t} \int_t^s
    (\sigma(r,   X^x_{t,r}))^{-1} \; \Big(\frac{\partial}{\partial x} X^x_{t,r}\Big) \,\d W_r
    \end{pmatrix},     
    \end{equation} 
    let $V \in C^{1,2}([0,T]\times O,(0, \infty))$
    satisfy that for all 
    $t\in [0,T]$, $s\in[t,T]$, $x\in O$
    it holds a.s.\! that 
    \begin{equation}\label{eq:V_ass3_1}
    \begin{split}
    &(\tfrac{\partial V}{\partial t})(s, X^x_{t,s})
    +\langle \mu(s, X^x_{t,s}),
    (\nabla_x V)(s, X^x_{t,s})\rangle
    \\
    &\qquad
    +\tfrac{1}{2}\operatorname{Tr}(\sigma(s, X^x_{t,s})[\sigma(s, X^x_{t,s})]^*(\operatorname{Hess}_x\!V)(s, X^x_{t,s}))
    \\&\qquad\qquad
    +\tfrac{1}{2}\tfrac{\norm{[(\nabla_x V)(s, X^x_{t,s})]^*
            \sigma(s, X^x_{t,s})}^2}{V(s, X^x_{t,s})}
    \leq K V(s, X^x_{t,s})+b
    \end{split}
    \end{equation}
    and 
    \begin{equation}\label{eq:V_ass3_2}
    \begin{split}
    &(\tfrac{\partial V}{\partial t})(t,x)
    +\langle \mu(t,x), (\nabla_x V)(t,x)\rangle
    +\tfrac{1}{2} \operatorname{Tr}(\sigma(t,x)[\sigma(t,x)]^* (\operatorname{Hess}_x\!V)(t,x))
    \\
    &\qquad +L\norm{(\nabla_x V)(t,x)}
    \leq 0,
    \end{split}
    \end{equation}
    let $f \in C([0,T] \times O \times \R\times \R^{d}, \R)
    \cap L^2([0,T] \times O \times \R\times \R^{d}, \R)$,
    $g \in C(O, \R) \cap L^2(O, \R)$ satisfy for all
    $t \in [0,T]$, $x_1,x_2 \in O$, 
    $a_1,a_2\in\R$,
    $w_1,w_2 \in \R^{d}$  
    that
    $\lvert f(t,x_1,a_1,w_1) 
    - f(t,x_2,a_2,w_2) \rvert 
    \leq L\Vnorm{
        (a_1, w_1)
        -(a_2, w_2)}$,
    and assume that
    $\inf_{r \in (0, \infty)} [ \sup_{t \in [0,T)\setminus K_r} $
    $\sup_{x \in O \setminus O_r}$ 
    $(\frac{\lvert  g(x) \rvert}{V(T,x)} 
    +\frac{\lvert f(t,x,0,0) \rvert}{V(t,x)}
    \sqrt{T-t})] 
    = 0$,
    $\liminf_{r\to\infty}
    [\inf_{t\in [0,T]}
    \inf_{x\in O\setminus O_r}$
    $V(t,x)]=\infty$, 
    and 
    $\inf_{t\in[0,T]}$
    $\inf_{x\in O}
    V(t,x)>0$.
    Then 
    \begin{enumerate}[label=(\roman*)]
        \item\label{it:vs1} 
        there exists a unique 
        $v\in C([0,T]\times O,\R)
        \cap C^{0,1}([0,T)\times O,\R)$
        which satisfies for all 
        $t\in[0,T)$, $x\in O$ that
        $\limsup_{r \to \infty}[
        \sup_{s \in [0,T)\setminus K_r} 
        \sup_{y \in O \setminus O_r}( \frac{\Vnorm{ 
                (v,\nabla_x v)(s,y)}}{V(s,y)}  
        \sqrt{T-s} ) ] 
        = 0$,
        $\E[\abs{g(X^x_{t,T})}\, 
        \Vnorm{ Z^x_{t,T}}  
        + \int_t^T  \abs{f(r, X^x_{t,r}, 
            v(r, X^x_{t,r}),
            (\nabla_x v)(r,X^x_{t,r}))}
        \Vnorm{ Z^x_{t,r}}  \,\d r ]
        <\infty$,\\
        $v(T,x)=g(x)$,
        and
        \begin{equation}\label{eq:v}
        \begin{split}
        &(v, \nabla_x v)(t,x)
        =\E \left[  g(X^x_{t,T}) Z^x_{t,T}  
        + \int_t^T  f(r, X^x_{t,r}, v(r, X^x_{t,r}),
        (\nabla_x v)(r,X^x_{t,r})) 
        Z^x_{t,r}  \,\d r \right], 
        \end{split}
        \end{equation}
        \item\label{it:vs2}
        there exists a unique viscosity solution 
        $u\in\{\mathbf{u}\in C([0,T]\times O,\R)
        \cap C^{0,1}([0,T)\times O,\R)
        \colon$ 
        $\limsup_{r \to \infty}$
        $[\sup_{t \in [0,T)\setminus K_r}$ 
        $\sup _{x \in O \setminus O_r}
        ( \frac{\Vnorm{
                (\mathbf{u},\nabla_x \mathbf{u})(t,x)
        }}{V(t,x)}  
        \, \sqrt{T-t} ) ] = 0\}$ 
        of
        \begin{multline}
        (\tfrac{\partial u}{\partial t})(t,x)
        +\langle \mu(t,x), (\nabla_x u)(t,x)\rangle
        +\tfrac{1}{2}\operatorname{Tr}(\sigma(t,x)[\sigma(t,x)]^*(\operatorname{Hess}_x u)(t,x))
        \\     
        +f(t,x,u(t,x), (\nabla_x u)(t,x))
        =0
        \end{multline}
        with $u(T,x)=g(x)$ for
        $(t,x)\in (0,T)\times O$,
        and
        \item\label{it:vs3} 
        for all $t\in[0,T]$, $x\in O$
        it holds that
        $u(t,x)=v(t,x)$.
    \end{enumerate}   
\end{theorem}

\begin{proof}[Proof of Theorem~\ref{thm:vs}]
    First note that 
    \cite[Theorem 3.5]{HP2023}
    \eqref{eq:der_mu_sigma_loclip3},
    and \eqref{eq:V_ass3_1}
    prove that there exists 
    a unique 
    $w=(w_1,w_2,\ldots, w_{d+1})
    \in C([0,T)\times O,\R^{d+1})$
    which satisfies 
    \begin{enumerate}[label=(\Roman*)]
        \item\label{it:w_boundedness}
        that
        $\limsup_{r \to \infty}$
        $[\sup_{s \in [0,T)\setminus K_r} 
        \sup_{y \in O \setminus O_r}
        ( \frac{\Vnorm{
                w(s,y)
        }}{V(s,y)} \, \sqrt{T-s} ) ] 
        = 0$,
        \item\label{it:w_integrability}
        for all $t\in [0,T)$, 
        $x\in O$ that
        \begin{equation}
        \E\bigg[\abs{g(X^x_{t,T})}\, 
        \Vnorm{ 
            Z^x_{t,T}}
        + \int_t^T  \abs{f(r, X^x_{t,r}, 
            w(r,X^x_{t,r}))}\, 
        \Vnorm{ Z^x_{t,r} }
        \,\d r \bigg]
        <\infty,
        \end{equation}
        and
        \item\label{it:w_presentation}
        for all $t\in [0,T)$,
        $x\in O$ it holds that
        \begin{equation}\label{eq:v2}
        w(t,x) 
        =\E \left[  g(X^x_{t,T}) Z^x_{t,T}  
        + \int_t^T  f(r, X^x_{t,r}, 
        w(r,X^x_{t,r})) 
        Z^x_{t,r}  \,\d r \right]. 
        \end{equation}
    \end{enumerate}
    Let $v\colon [0,T]\times O
    \to \R$ 
    satisfy for all $t\in [0,T)$,
    $x\in O$ that
    $v(t,x)=w_1(t,x)$
    and $v(T,x)=g(x)$.
    Observe that the fact that 
    $w
    \in C([0,T)\times O,\R^{d+1})$
    implies that
    $v \in  C([0,T)\times O,\R)$.
    To prove that $v$ is 
    continuous in $T$ let 
    $(t_n)_{n\in\N}\subseteq [0,T)$
    satisfy 
    $\limsup_{n\to\infty}
    \abs{t_n-T}=0$.
    Note that
    \cite[Lemma 3.1]{beck2021existence}
    and \eqref{eq:V_ass3_2}
    imply that for all 
    $t\in [0,T]$, $s\in [t,T]$,
    $x\in O$
    it holds that 
    \begin{equation}
    \label{eq:V_bd_startingpoint}
    E[V(s,X^x_{t,s})]
    \leq V(t,x).
    \end{equation}
    Combining this with
    Fubini's theorem 
    and the assumption that
    for all
    $t \in [0,T]$, $x_1,x_2 \in O$, 
    $a_1,a_2\in\R$,
    $w_1,w_2 \in \R^{d}$  
    it holds that
    $\lvert f(t,x_1,a_1,w_1) 
    - f(t,x_2,a_2,w_2) \rvert 
    \leq L\Vnorm{
        (a_1, w_1)
        -(a_2, w_2) }$  
    demonstrates
    that for all $n\in\N$, $x\in O$
    it holds that
    \begin{equation}
    \begin{split}
    &\E\bigg[\bigg| \int_{t_n}^T
    f(r,X^x_{t_n,r}, 
    w(r,X^x_{t_n,r}))
    \,\d r \bigg|\bigg]
    \leq  \E\bigg[\int_{t_n}^T
    \abs{f(r,X^x_{t_n,r}, 
        w(r,X^x_{t_n,r}))}\,\d r 
    \bigg]\\
    &\leq \E\bigg[ \int_{t_n}^T
    \Big[\abs{f(r,X^x_{t_n,r}, 
        w(r,X^x_{t_n,r}))
        -f(r,X^x_{t_n,r},0,0)}
    +\abs{f(r,X^x_{t_n,r},0,0)}\Big]\,\d r 
    \bigg]\\ 
    &\leq  \E\bigg[\int_{t_n}^T
    \Big[L\Vnorm{
        w(r,X^x_{t_n,r})}
    +\abs{f(r,X^x_{t_n,r},0,0)}\Big]\,\d r 
    \bigg]\\ 
    &=\E\Bigg[\int_{t_n}^T
    \bigg[
    \tfrac{L\Vnorm{
            w(r,X^x_{t_n,r})}
        +\abs{f(r,X^x_{t_n,r},0,0)}}
    {V(r,X^x_{t_n,r})}\sqrt{T-r}
    \tfrac{V(r,X^x_{t_n,r})}{\sqrt{T-r}}
    \bigg]\,\d r \Bigg] \\ 
    &\leq 
    \Bigg[\sup_{s\in [0,T)}\sup_{y\in O}
    \tfrac{L\Vnorm{
            w(s,y)}
        +\abs{f(s,y,0,0)}}
    {V(s,y)}\sqrt{T-s} \bigg]    
    \int_{t_n}^T
    \tfrac{E[V(r,X^x_{t_n,r})]}{\sqrt{T-r}}
    \,\d r \\
    &\leq 
    \Bigg[\sup_{s\in [0,T)}\sup_{y\in O}
    \tfrac{L\Vnorm{
            w(s,y)}
        +\abs{f(s,y,0,0)}}
    {V(s,y)}\sqrt{T-s} \bigg]    
    \int_{t_n}^T
    \tfrac{V(t_n,x)}{\sqrt{T-r}}
    \,\d r \\
    &=
    \Bigg[\sup_{s\in [0,T)}\sup_{y\in O}
    \tfrac{L\Vnorm{
            w(s,y)}
        +\abs{f(s,y,0,0)}}
    {V(s,y)}\sqrt{T-s} \bigg] 
    \bigg[\sup_{s\in [0,T]} V(s,x)\bigg]   
    2\sqrt{T-t_n}.
    \end{split} 
    \end{equation}
    Item~\ref{it:w_boundedness},
    the assumption that
    $\inf_{r \in (0, \infty)}$ 
    $[ \sup_{t \in [0,T)\setminus K_r} 
    \sup_{x \in O \setminus O_r} 
    (\frac{\lvert f(t,x,0,0) \rvert}{V(t,x)}
    \sqrt{T-t})] 
    = 0$,
    and the fact that 
    $V\in C^{1,2}([0,T]\times O, (0,\infty))$
    therefore show that
    for all $x\in O$
    it holds that
    \begin{equation}
    \label{eq:f_X_conv}
    \limsup_{n\to\infty}
    \E\bigg[\bigg| \int_{t_n}^T
    f(r,X^x_{t_n,r}, 
    w(r,X^x_{t_n,r}))
    \,\d r \bigg|\bigg]
    =0.
    \end{equation}
    In addition, note that
    \cite[Corollary 2.4]{beck2021existence}
    demonstrates that there exists
    compactly supported
    $\mathfrak{g}_n\in C(O, \R)$, $n \in \N$,
    which satisfy
    \begin{equation}\label{eq:g_approx}
    \limsup \limits_{n \to \infty} \left[ 
    \sup \limits_{x \in O} \left( \frac{\lvert \mathfrak{g}_n(x)-g(x) \rvert}{V(T,x)} \right) \right]
    = 0.
    \end{equation}
    This,
    the triangle inequality,
    and \eqref{eq:V_bd_startingpoint}
    show that
    for all $k,n\in\N$,
    $x\in O$
    it holds that
    \begin{equation}
    \begin{split}
    &\E[\abs{g(X^x_{t_n,T})-g(x)}]\\
    &\leq \E[\abs{g(X^x_{t_n,T})
        -\mathfrak{g}_k(X^x_{t_n,T})}]
    +\E[\abs{\mathfrak{g}_k(X^x_{t_n,T})
        -\mathfrak{g}_k(x)}]
    +\E[\abs{\mathfrak{g}_k(x)-g(x)}]\\
    &= \E\Bigg[\frac{\abs{g(X^x_{t_n,T})
            -\mathfrak{g}_k(X^x_{t_n,T})}}{V(T,X^x_{t_n,T}
        )}V(T,X^x_{t_n,T})\Bigg]
    +\E[\abs{\mathfrak{g}_k(X^x_{t_n,T})
        -\mathfrak{g}_k(x)}]\\
    &\qquad
    +\E\Bigg[
    \frac{\abs{\mathfrak{g}_k(x)
            -g(x)}}{V(T,x)}V(T,x)\Bigg]\\
    &\leq \Bigg[\sup_{y\in O}
    \frac{\abs{g(y)-\mathfrak{g}_k(y)}}{V(T,y)}\Bigg]
    \E[V(T,X^x_{t_n,T})]
    +\E[\abs{\mathfrak{g}_k(X^x_{t_n,T})
        -\mathfrak{g}_k(x)}]\\
    &\qquad 
    +\Bigg[\sup_{y\in O}\frac{\abs{\mathfrak{g}_k(y)
            -g(y)}}{V(T,y)}\Bigg] V(T,x)\\
    &\leq \Bigg[\sup_{y\in O}
    \frac{\abs{g(y)-\mathfrak{g}_k(y)}}{V(T,y)}\Bigg]
    V(t_n,x)
    +\E[\abs{\mathfrak{g}_k(X^x_{t_n,T})-\mathfrak{g}_k(x)}]\\
    &\qquad
    +\Bigg[\sup_{y\in O}\frac{\abs{
            \mathfrak{g}_k(y)
            -g(y)}}{V(T,y)}\Bigg]V(T,x)\\
    &\leq 2\Bigg[\sup_{y\in O}
    \frac{\abs{g(y)-\mathfrak{g}_k(y)}}{V(T,y)}\Bigg]
    \bigg[\sup_{s\in [0,T]} V(s,x)\bigg]
    +\E[\abs{\mathfrak{g}_k(X^x_{t_n,T})
        -\mathfrak{g}_k(x)}].
    \end{split}
    \end{equation}  
    This, 
    \cite[Lemma 3.7]{beck2021existence},
    the Portemonteau theorem, and
    \eqref{eq:g_approx}
    demonstrate for all 
    $x\in O$ that
    \begin{equation}
    \label{eq:g_X_conv}
    \limsup\nolimits_{n\to\infty}
    \E[\abs{g(X^x_{t_n,T})-g(x)}]
    =0.
    \end{equation}   
    Combining this with
    \eqref{eq:f_X_conv}
    proves that
    for all $x\in O$
    it holds that
    \begin{equation}
    \limsup\nolimits_{n\to\infty}
    \abs{v(t_n,x)-g(x)}
    =0.
    \end{equation}
    The assumption that 
    for all $x\in \R^d$
    it holds that
    $v(T,x)=g(x)$ 
    and the fact that
    $v \in  C([0,T)\times O,\R)$
    hence
    demonstrate that
    $v\in C([0,T]\times O, \R)$.
    Next note that
    items~\ref{it:u_diff_BEL}
    and \ref{it:u_der_BEL} of
    Theorem~\ref{thm:BEL_formula2}
    (applied for all 
    $t\in [0,T)$, $r\in (t,T]$
    with $O\curvearrowleft \R^d$,
    $f\curvearrowleft g$
    and
    $T\curvearrowleft r$,
    $f\curvearrowleft (\R^d \ni x
    \mapsto f(r,X^x_{t,r}, 
    w(r,X^x_{t,r}))
    \in \R)$
    in the notation of
    Theorem~\ref{thm:BEL_formula2}),
    Leibniz integral rule,
    Fubini's theorem,
    item~\ref{it:w_integrability},
    and the assumption that
    $f\in L^2([0,T]\times O\times\R\times\R^d,\R)$
    and $g\in L^2(O,\R)$
    show that
    $v\in C^{0,1}([0,T)\times
    O,\R)$
    and for all $t\in [0,T)$,
    $x\in\R^d$
    it holds that
    \begin{equation}
    \begin{split}
    &(\nabla_x v)(t,x)
    = \nabla_x \Big(
    \E[g(X^x_{t,T})]\Big)
    + \nabla_x \bigg(\int_t^T \E[f(r,X^x_{t,r},
    w(r,X^x_{t,r}))] \, \d r\bigg)\\
    &= \E[g(X^x_{t,T})Z^x_{t,T}]
    + \int_t^T \nabla_x\Big(
    \E\left[ f(r,X^x_{t,r},
    w(r,X^x_{t,r})) \right]\Big) \, \d r\\
    &=  \E[g(X^x_{t,T})Z^x_{t,T}]
    + \int_t^T \E\left[ f(r,X^x_{t,r},
    w(r,X^x_{t,r}))Z^x_{t,r} \right]\, \d r\\
    &=  \E[g(X^x_{t,T})Z^x_{t,T}]
    +  \E\bigg[\int_t^T f(r,X^x_{t,r},
    w(r,X^x_{t,r}))Z^x_{t,r}
    \, \d r\bigg]. 
    \end{split}
    \end{equation}   
    Item~\ref{it:w_presentation}
    therefore implies that
    for all $t\in [0,T)$,
    $x\in\R^d$
    it holds that
    $(w_2, w_3, \ldots, w_{d+1})(t,x)
    = (\nabla_x v)(t,x)$.
    This, 
    items~\ref{it:w_boundedness}-\ref{it:w_presentation},
    and the fact that
    $v\in C([0,T]\times O, \R)$
    establish item~\ref{it:vs1}.
    Next we prove items~\ref{it:vs2}
    and \ref{it:vs3}.
    For this let 
    $h\colon [0,T]\times O \to \R$
    satisfy for all $t\in[0,T)$,
    $x\in O$ that
    $h(t,x)=f(t,x,v(t,x), (\nabla_x v)(t,x))$.
    Note that 
    item~\ref{it:w_boundedness},
    the fact that for all
    $t \in [0,T]$, $x \in O$, 
    $a_1,a_2\in\R$, $w_1,w_2 \in \R^{d}$
    it holds that
    $\lvert f(t,x,a_1,w_1) 
    - f(t,x,a_2,w_2) \rvert 
    \leq L \Vnorm{
        (a_1, w_1)
        -(a_2,w_s)} $,
    and the fact that
    $\inf_{r \in (0, \infty)} [ \sup_{t \in [0,T)\setminus K_r} 
    \sup_{x \in O \setminus O_r} $
    $(\frac{\lvert f(t,x,0,0) \rvert}{V(t,x)}
    \sqrt{T-t})] 
    = 0$
    imply that $h\in C([0,T)\times O,\R)$ 
    and 
    \begin{equation}
    \label{eq:h_V_bd}
    \begin{split}
    &\limsup_{r\to\infty}\bigg[\sup_{t\in [0,T)\setminus K_r}\sup_{x\in O\setminus O_r} \bigg(\tfrac{\abs{h(t,x)}}{V(t,x)}\sqrt{T-t}\bigg)\bigg]\\
    &=\limsup_{r\to\infty}\bigg[\sup_{t\in [0,T)\setminus K_r}\sup_{x\in O\setminus O_r} \bigg(\tfrac{\abs{f(t,x,v(t,x),(\nabla_x v)(t,x))}}{V(t,x)}\sqrt{T-t}\bigg)\bigg]\\
    &\leq  \limsup_{r\to\infty}\bigg[\sup_{t\in [0,T)\setminus K_r}\sup_{x\in O\setminus O_r} \bigg(\tfrac{\abs{f(t,x,0,0)}}{V(t,x)}\sqrt{T-t}\\
    &\qquad
    +\tfrac{\abs{f(t,x,v(t,x),(\nabla_x v)(t,x))-f(t,x,0,0)}}{V(t,x)}\sqrt{T-t}
    \bigg)\bigg]\\
    &\leq  \limsup_{r\to\infty}\bigg[\sup_{t\in [0,T)\setminus K_r}\sup_{x\in O\setminus O_r} \bigg(\tfrac{\abs{f(t,x,0,0)}
        +L\Vnorm{ (v,\nabla_x v)(t,x)}}{V(t,x)}\sqrt{T-t}\bigg)\bigg]
    =0.
    \end{split}
    \end{equation}
    Proposition~\ref{prop:vs23},
    \eqref{eq:der_mu_sigma_loclip3},
    \eqref{eq:V_ass3_2},
    and \eqref{eq:v}
    therefore demonstrate that
    $v$ is a viscosity solution of
    \begin{equation}\label{eq:v_vs}
    \begin{split}
    &(\tfrac{\partial v}{\partial t})(t,x)
    +\langle \mu(t,x), (\nabla_x v)(t,x)\rangle
    +\tfrac{1}{2}\operatorname{Tr}(\sigma(t,x)[\sigma(t,x)]^*(\operatorname{Hess}_x v)(t,x))
    +h(t,x)
    =0
    \end{split}
    \end{equation}
    for $(t,x)\in (0,T)\times O$.
    This ensures that for all 
    $t\in(0,T)$, $x\in O$,
    $\phi\in C^{1,2}((0,T)\times O,\R)$
    with $\phi \geq v$ and
    $\phi(t,x)=v(t,x)$
    it holds that
    \begin{equation}
    \label{eq:v_vs1}
    \begin{split}
    &(\tfrac{\partial \phi}{\partial t})(t,x)
    +\langle \mu(t,x), (\nabla_x \phi)(t,x)\rangle
    +\tfrac{1}{2}\operatorname{Tr}(\sigma(t,x)[\sigma(t,x)]^*(\operatorname{Hess}_x \phi)(t,x))
    +h(t,x)
    \geq 0.
    \end{split}
    \end{equation}
    Moreover, observe that 
    \eqref{eq:v_vs}
    shows that for all
    $t\in(0,T)$, $x\in O$,
    $\phi\in C^{1,2}((0,T)\times O, \R)$
    with $\phi\leq v$ and 
    $\phi(t,x)=v(t,x)$
    it holds that
    \begin{equation}\label{eq:v_vs2}
    \begin{split}
    & (\tfrac{\partial \phi}{\partial t})(t,x)
    +\langle \mu(t,x), (\nabla_x \phi)(t,x)\rangle
    +\tfrac{1}{2}\operatorname{Tr}(\sigma(t,x)[\sigma(t,x)]^*(\operatorname{Hess}_x \phi)(t,x))
    +h(t,x)
    \leq 0.
    \end{split}
    \end{equation}
    Combining this with 
    \eqref{eq:v_vs1}
    proves that $v$ is a viscosity
    solution of 
    \begin{equation}
    \label{eq:v_vs3}
    \begin{split}
    &(\tfrac{\partial v}{\partial t})(t,x)
    +\langle \mu(t,x), (\nabla_x v)(t,x)\rangle
    \\
    &\qquad
    +\tfrac{1}{2}\operatorname{Tr}(\sigma(t,x)[\sigma(t,x)]^*(\operatorname{Hess}_x v)(t,x))
    +f(t,x,v(t,x),(\nabla_x v)(t,x))
    =0
    \end{split}
    \end{equation}
    for $(t,x)\in (0,T)\times O$.
    Proposition~\ref{prop:vs3_5},
    (applied with $u_1\curvearrowleft v$
    in the notation of 
    Proposition~\ref{prop:vs3_5}),
    \eqref{eq:V_ass3_2},
    and the fact that 
    $v\in \{\mathbf{u}\in C([0,T]\times O,\R)
    \cap C^{0,1}([0,T)\times O,\R)
    \colon$ $\limsup_{r\to\infty}[
    \sup_{t\in [0,T)\setminus K_r}$
    $\sup_{x\in O\setminus O_r}
    (\frac{(\mathbf{u},\nabla_x\mathbf{u})(t,x)}{V(t,x)}
    \sqrt{T-t})]=0\}$
    therefore establish
    items~\ref{it:vs2}
    and \ref{it:vs3}. 
    The proof of Theorem~\ref{thm:vs}
    is thus complete.
\end{proof}

The following Corollary applies the
results in Theorem~\ref{thm:vs}
to a function $V$ that is
independent of the time component.
The proof of the Corollary~\ref{cor:vs1}
is similar to the one of \cite[Corollary 3.8]{beck2021nonlinear}.

\begin{corollary}\label{cor:vs1}
    Let $d \in \N$, 
    $\alpha, c, L, T \in (0, \infty)$,
    $\rho\in\R$,
    let $\langle\cdot,\cdot\rangle\colon
    \R^d\times\R^d\to\R$ be the standard
    Euclidean scalar product on $\R^d$,
    let $\norm{\cdot}\colon\R^d\to[0,\infty)$
    be the standard Euclidean norm on $\R^d$,
    let $\Vnorm{\cdot}\colon
    \R^{d+1}\to[0,\infty)$
    be the standard Euclidean norm on $\R^{d+1}$,
    let $\norm{\cdot}_F\colon \R^{d\times d}
    \to [0,\infty)$ be the Frobenius norm
    on $\R^{d\times d}$,
    let $O\subseteq \R^d$ be an
    open set,
    for every $r \in (0, \infty)$
    let $K_r\subseteq [0,T)$,  
    $O_r \subseteq O$
    satisfy
    $K_r=[0,\max\{T-\frac{1}{r},0\}]$
    and
    $O_r = \{  x \in O\colon\norm{x} \leq r \text{ and } \{ y \in \R^d\colon \norm{y-x} 
    < \frac{1}{r} \} \subseteq O \}$,
    let $(\Omega, \mathcal{F}, \mathbb{P}, (\mathbb{F}_s)_{s \in [0,T]})$ 
    be a filtered probability space
    satisfying the usual conditions,
    let $W \colon [0,T] \times \Omega \to \R^d$ 
    be a standard $(\mathbb{F}_s)_{s \in [0,T]}$-Brownian motion,
    let $\mu \in  C^{0,1}([0,T] \times O, \R^d)$, 
    $\sigma \in C^{0,1} ([0,T] \times O, \R^{d \times d})$
    satisfy
    for all 
    $s\in[0,T]$, $x, y \in O$, $v\in\R^d$
    that
    \begin{equation}
    \max\Big\{\langle x-y,\mu(s,x)-\mu(s,y)\rangle, 
    \tfrac{1}{2}\norm{\sigma(s,x)-\sigma(s,y)}_F^2\Big\}
    \leq \tfrac{c}{2}\norm{x-y}^2
    \end{equation}
    and
    $v^* \sigma(s,x) (\sigma(s,x))^* v \geq \alpha \norm{v}^2$,
    assume for all $r\in (0,\infty)$,
    $j\in\{1,2,\ldots, d\}$ that
    \begin{equation}     
    \begin{split}
    \sup \bigg(&\bigg\{
    \tfrac{\norm{\frac{\partial \mu}{\partial x}(t,x)-\frac{\partial \mu}{\partial x}(t,y)}_F
        +\norm{\frac{\partial \sigma}{\partial x_j}(t,x)
            -\frac{\partial \sigma}{\partial x_j}(t,y)}_F}{\norm{x-y}}
    \colon\\
    &\qquad
    t\in[0,T], x,y\in O_r, x\neq y \bigg\} 
    \cup \{0\} \bigg)
    <\infty,
    \end{split}
    \end{equation}
    for every  $t\in [0,T]$,
    $x \in O$ let
    $X^x_t = (X^{x}_{t,s})_{s \in [t,T]} \colon [t,T] \times \Omega \to O$   
    be an 
    $(\mathbb{F}_s)_{s \in [t,T]}$-adapted 
    stochastic process with continuous
    sample paths satisfying that for all 
    $s \in [t,T]$ it holds a.s.\! that
    \begin{equation}
    X^x_{t,s} = x + \int_t^s \mu(r, X^x_{t,r}) \,\d r 
    + \int_t^s \sigma(r, X^x_{t,r}) \,\d W_r,
    \end{equation}
    assume for all
    $t\in [0,T]$,
    $\omega \in \Omega$
    that
    $\left([t,T] \times O \ni (s,x) 
    \mapsto X^x_{t,s}(\omega) \in O \right) \in C^{0,1}([t,T] \times O, O)$,
    for every $t\in[0,T]$, 
    $x \in O$ let
    $Z^x_t = (Z^x_{t,s})_{s \in (t,T]} 
    \colon (t,T] \times \Omega \to \R^{d+1}$ 
    be an $(\mathbb{F}_s)_{s \in (t,T]}$-adapted
    stochastic process 
    with continuous sample paths
    satisfying that 
    for all $s \in (t,T]$ 
    it holds a.s.\! that
    \begin{equation}
    Z^x_{t,s} = 
    \begin{pmatrix}
    1\\
    \frac{1}{s-t} \int_t^s
    (\sigma(r,   X^x_{t,r}))^{-1} \; \Big(\frac{\partial}{\partial x} X^x_{t,r}\Big) \,\d W_r
    \end{pmatrix},     
    \end{equation} 
    let $V \in C^{2}(O,(0, \infty))$
    satisfy for all 
    $t\in [0,T]$, $x\in O$
    that 
    \begin{equation}\label{eq:V_ass1_cor1}
    \begin{split}
    &\langle \mu(t, x),
    (\nabla V)(x)\rangle
    +\tfrac{1}{2}\operatorname{Tr}(\sigma(t, x)[\sigma(t, x)]^*(\operatorname{Hess} \!V)(x))\\
    &\qquad
    +\tfrac{1}{2}\tfrac{\norm{[(\nabla V)(x)]^*
            \sigma(t, x)}^2}{V(x)}
    \leq \rho V(x)
    \end{split}
    \end{equation}
    and 
    \begin{equation}\label{eq:V_ass2_cor1}
    \begin{split}
    &\langle \mu(t,x), (\nabla V)(x) \rangle
    +\tfrac{1}{2} \operatorname{Tr}(\sigma(t,x)[\sigma(t,x)]^* (\operatorname{Hess}\!V)(x))
    \\
    &\qquad
    +L\norm{(\nabla V)(x)}
    \leq \rho V(x),
    \end{split}
    \end{equation}
    let $f \in C([0,T] \times O \times \R\times \R^{d}, \R)
    \cap L^2([0,T] \times O \times \R\times \R^{d}, \R)$,
    $g \in C(O, \R) \cap L^2(O, \R)$ 
    satisfy for all
    $t \in [0,T]$, $x_1,x_2 \in O$, 
    $a_1,a_2\in\R$,
    $w_1,w_2 \in \R^{d}$  
    that
    $\lvert f(t,x_1,a_1,w_1) 
    - f(t,x_2,a_2,w_2) \rvert 
    \leq L\Vnorm{
        (a_1, w_1)
        -(a_2, w_2) }$,
    and assume that
    $\inf_{r \in (0, \infty)} [ \sup_{t \in [0,T)\setminus K_r} 
    \sup_{x \in O \setminus O_r} $
    $(\frac{\abs{ g(x)}
        +\abs{f(t,x,0,0)}\sqrt{T-t}}{V(x)} )] 
    = 0$, 
    $\liminf_{r\to\infty}
    [\inf_{x\in O\setminus O_r}$
    $V(x)]=\infty$, 
    and 
    $\inf_{x\in O}
    V(x)$
    $>0$.
    Then 
    \begin{enumerate}[label=(\roman*)]
        \item\label{it:vs1_cor1} 
        there exists a unique 
        $v\in C([0,T]\times O,\R)
        \cap C^{0,1}([0,T)\times O,\R)$
        which satisfies for all 
        $t\in[0,T)$, $x\in O$ that
        $\limsup_{r \to \infty}[
        \sup_{s \in [0,T)\setminus K_r} 
        \sup_{y \in O \setminus O_r}( \frac{\Vnorm{
                (v,\nabla_x v)(s,y)}}{V(y)}
        \sqrt{T-s} ) ] 
        = 0$,
        $\E[\abs{g(X^x_{t,T})} 
        \Vnorm{ Z^x_{t,T}}  
        + \int_t^T  \abs{f(r, X^x_{t,r}, 
            v(r, X^x_{t,r}),
            (\nabla_x v)(r,X^x_{t,r}))} 
        \Vnorm{ Z^x_{t,r}}  \,\d r ]
        <\infty$,
        and
        \begin{equation}\label{eq:v_cor1}
        \begin{split}
        &(v, \nabla_x v)(t,x) 
        =\E \left[  g(X^x_{t,T}) Z^x_{t,T}  
        + \int_t^T  f(r, X^x_{t,r}, v(r, X^x_{t,r}),
        (\nabla_x v)(r,X^x_{t,r})) 
        Z^x_{t,r}  \,\d r \right],
        \end{split} 
        \end{equation}
        \item\label{it:vs2_cor1}
        there exists a unique viscosity solution 
        $u\in\{\mathbf{u}\in C([0,T]\times O,\R)
        \cap C^{0,1}([0,T)\times O,\R)
        \colon \limsup_{r \to \infty}$
        $[
        \sup_{t \in [0,T)\setminus K_r}$ 
        $\sup _{x \in O \setminus O_r}
        ( \frac{\Vnorm{
                (\mathbf{u},\nabla_x\mathbf{u})(t,x)
        }}{V(x)}  
        \, \sqrt{T-t} ) ] = 0\}$ 
        of
        \begin{multline}
        (\tfrac{\partial u}{\partial t})(t,x)
        +\langle \mu(t,x), (\nabla_x u)(t,x)\rangle
        +\tfrac{1}{2}\operatorname{Tr}(\sigma(t,x)[\sigma(t,x)]^*(\operatorname{Hess}_x u)(t,x))
        \\
        +f(t,x,u(t,x), (\nabla_x u)(t,x))
        =0
        \end{multline}
        with $u(T,x)=g(x)$ for
        $(t,x)\in (0,T)\times O$,
        and
        \item\label{it:vs3_cor1} 
        for all $t\in[0,T]$, $x\in O$
        it holds that
        $u(t,x)=v(t,x)$.
    \end{enumerate}   
\end{corollary}

\begin{proof}[Proof of Corollary~\ref{cor:vs1}]
    Throughout this proof let
    $\V\colon [0,T]\times O\to \R$
    satisfy for all $t\in [0,T]$,
    $x\in O$ that
    $\V(t,x)= e^{-\rho t} V(x)$.
    Note that the
    product rule and 
    the fact that 
    $V\in C^2(O,(0,\infty))$
    ensure that 
    $\V\in C^{1,2}([0,T]\times O,(0,\infty))$.
    This and
    \eqref{eq:V_ass1_cor1}
    demonstrate that  
    for all $t\in [0,T]$, $x\in O$
    it holds that
    \begin{equation}
    \label{eq:V_prop1}
    \begin{split}
    &(\tfrac{\partial\V}{\partial t})(t,x)
    +\langle \mu(t, x), (\nabla_x \V)(t,x)
    \rangle
    +\tfrac{1}{2}\operatorname{Tr}(\sigma(t, x)[\sigma(t, x)]^*(\operatorname{Hess}_x\!\V)(t,x))\\
    &\qquad
    +\tfrac{1}{2}\tfrac{\norm{(\nabla_x \V)(t, x)\sigma(t, x)}^2_F}{\V(t,x)}\\
    &= e^{-\rho t} \Big(
    -\rho V(x)
    +\langle \mu(t, x),
    (\nabla V)(x)\rangle\\
    &\qquad
    +\tfrac{1}{2}\operatorname{Tr}(\sigma(t, x)[\sigma(t, x)]^*(\operatorname{Hess}\!V)(x))
    +\tfrac{1}{2}\tfrac{\norm{[(\nabla V)(x)]^*
            \sigma(t, x)}^2}{V(x)}\Big)
    \leq 0.
    \end{split}
    \end{equation}   
    Moreover, note that
    \eqref{eq:V_ass2_cor1}   
    shows
    that for all $t\in [0,T]$, $x\in O$
    it holds that
    \begin{equation}
    \label{eq:V_prop2}
    \begin{split}
    &(\tfrac{\partial\V}{\partial t})(t,x)
    +\langle \mu(t, x), 
    (\nabla_x \V)(t,x) \rangle
    +\tfrac{1}{2}\operatorname{Tr}(\sigma(t, x)[\sigma(t, x)]^*(\operatorname{Hess}_x\!\V)(t,x))\\
    &\qquad
    +L\norm{(\nabla_x \V)(x)}\\
    &= e^{-\rho t} \Big(
    -\rho V(x)
    +\langle \mu(t, x), (\nabla V)(x)
    \rangle\\
    &\qquad
    +\tfrac{1}{2}\operatorname{Tr}(\sigma(t, x)[\sigma(t, x)]^*(\operatorname{Hess}\!V)(x))
    +L\norm{(\nabla V)(x)}\Big)
    \leq 0.
    \end{split}
    \end{equation}  
    Furthermore, note that the 
    fact that
    $\inf_{r \in (0, \infty)} [ 
    \sup_{t\in [0,T)\setminus K_r}
    \sup_{x \in O \setminus O_r} 
    (\frac{\lvert  g(x) \rvert }{V(x)}+ \frac{\lvert f(t,x,0,0) \rvert}{V(x)}$
    $\cdot    \sqrt{T-t} 
    )] 
    = 0$
    ensures that
    \begin{equation}
    \label{eq:V_prop3}
    \begin{split}
    &\inf_{r \in (0, \infty)} \bigg[ 
    \sup_{t \in [0,T)\setminus K_r} 
    \sup_{x \in O \setminus O_r} 
    \bigg(\frac{\lvert  g(x) \rvert}{\V(T,x)} 
    +\frac{\lvert f(t,x,0,0) \rvert}{\V(t,x)}
    \sqrt{T-t}\bigg)\bigg] \\
    &= \inf_{r \in (0, \infty)} \bigg[ 
    \sup_{t \in [0,T)\setminus K_r} 
    \sup_{x \in O \setminus O_r} 
    \bigg(\frac{\lvert  g(x) \rvert}{e^{-\rho T} V(x)} 
    +\frac{\lvert f(t,x,0,0) \rvert}{e^{-\rho t} V(x)}
    \sqrt{T-t}\bigg)\bigg]\\
    &\leq e^{\rho T} \inf_{r \in (0, \infty)} \bigg[ 
    \sup_{t \in [0,T)\setminus K_r} 
    \sup_{x \in O \setminus O_r} 
    \bigg(\frac{\lvert  g(x) \rvert}{V(x)} 
    +\frac{\lvert f(t,x,0,0) \rvert}{V(x)}\sqrt{T-t}
    \bigg)\bigg] = 0.
    \end{split}
    \end{equation}
    In addition, observe that 
    the assumption that
    $\liminf_{r\to\infty}
    [\inf_{x\in O\setminus O_r}
    V(x)]$
    $=\infty$ 
    and 
    $\inf_{x\in O}
    V(x)>0$
    guarantee that
    \begin{equation}
    \liminf_{r\to\infty}
    [\inf_{t\in [0,T]}
    \inf_{x\in O\setminus O_r}
    \V(t,x)]=\infty 
    \qquad \text{and}\qquad 
    \inf_{t\in[0,T]}\inf_{x\in O}
    \V(t,x)>0.
    \end{equation}
    Combining this with
    \eqref{eq:V_prop1}-\eqref{eq:V_prop3}
    and Theorem~\ref{thm:vs}
    (applied with 
    $b\curvearrowleft 0$,
    $K\curvearrowleft 0$,
    $V\curvearrowleft \V$
    in the notation of
    Theorem~\ref{thm:vs})
    establishes item~\ref{it:vs1_cor1}-\ref{it:vs3_cor1}.
    The proof of Corollary~\ref{cor:vs1}
    is thus complete.
\end{proof}

In the following corollary we
show that the function $\R^d \ni x \mapsto
(1+\norm{x}^2)^{\frac{p}{2}}$
for $p\in (0,\infty)$
satisfies the conditions 
\eqref{eq:V_ass3_1}
and \eqref{eq:V_ass3_2}
in Theorem~\ref{thm:vs} and can 
therefore - under the right assumptions
on the coefficients $\mu$ and
$\sigma$- ensure that there exists a
solution to the SFPE in 
\eqref{eq:v}
which is also a viscosity solution to the 
corresponding PDE.
The proof of the following corollary, 
Corollary~\ref{cor:vs2}, 
is
similar to the proof of 
\cite[Corollary 3.9]{beck2021nonlinear}.

\begin{corollary}\label{cor:vs2}
    Let $d \in \N$, 
    $\alpha, c, L, T \in (0, \infty)$,
    let $\langle\cdot,\cdot\rangle\colon
    \R^d\times\R^d\to\R$ be the standard
    Euclidean scalar product on $\R^d$,
    let $\norm{\cdot}\colon\R^d\to[0,\infty)$
    be the standard Euclidean norm on $\R^d$,
    let $\Vnorm{\cdot}\colon\R^{d+1}\to[0,\infty)$
    be the standard Euclidean norm on $\R^{d+1}$,
    let $\norm{\cdot}_F\colon \R^{d\times d}
    \to [0,\infty)$ be the Frobenius norm
    on $\R^{d\times d}$,
    let $(\Omega, \mathcal{F}, \mathbb{P}, (\mathbb{F}_s)_{s \in [0,T]})$ 
    be a filtered probability space
    satisfying the usual conditions,
    let $W \colon [0,T] \times \Omega \to \R^d$ 
    be a standard $(\mathbb{F}_s)_{s \in [0,T]}$-Brownian motion,
    let $\mu \in  C^{0,1}([0,T] \times \R^d, \R^d)$, 
    $\sigma \in C^{0,1} ([0,T] \times \R^d, \R^{d \times d})$
    satisfy
    for all 
    $s\in[0,T]$, $x, y \in \R^d$, $v\in\R^d$
    that
    \begin{equation}
    \label{eq:mu_sigma_loclip3}
    \max\Big\{\langle x-y,\mu(s,x)-\mu(s,y)\rangle, 
    \tfrac{1}{2}\norm{\sigma(s,x)-\sigma(s,y)}_F^2\Big\}
    \leq \tfrac{c}{2}\norm{x-y}^2,
    \end{equation}
    $\max \{
    \langle x, \mu(t,x)\rangle,
    \norm{\sigma(t,x)}_F^2
    \}
    \leq c (1+\norm{x}^2)$,
    and
    $v^* \sigma(s,x) (\sigma(s,x))^* v \geq \alpha \norm{v}^2$,
    assume for all $r\in (0,\infty)$,
    $j\in\{1,2,\ldots, d\}$ that
    \begin{equation}     
    \begin{split}
    \sup \bigg(\bigg\{
    &\tfrac{\norm{\frac{\partial \mu}{\partial x}(t,x)-\frac{\partial \mu}{\partial x}(t,y)}_F
        +\norm{\frac{\partial \sigma}{\partial x_j}(t,x)
            -\frac{\partial \sigma}{\partial x_j}(t,y)}_F}{\norm{x-y}}
    \colon\\
    &\qquad
    t\in[0,T], 
    x,y \in \{ z\in \R^d,
    \norm{z}\leq r\}, 
    x\neq y \bigg\} 
    \cup \{0\} \bigg)
    <\infty,
    \end{split}
    \end{equation}
    for every  $t\in [0,T]$,
    $x \in \R^d$ let
    $X^x_t = (X^{x}_{t,s})_{s \in [t,T]} \colon [t,T]
    \times \Omega \to \R^d$   
    be an 
    $(\mathbb{F}_s)_{s \in [t,T]}$-adapted 
    stochastic process with continuous
    sample paths satisfying that for all 
    $s \in [t,T]$ it holds a.s.\! that
    \begin{equation}
    X^x_{t,s} = x + \int_t^s \mu(r, X^x_{t,r}) \,\d r 
    + \int_t^s \sigma(r, X^x_{t,r}) \,\d W_r,
    \end{equation}
    assume for all
    $t\in [0,T]$,
    $\omega \in \Omega$
    that
    $\left([t,T] \times \R^d \ni (s,x) 
    \mapsto X^x_{t,s}(\omega) \in \R^d \right) \in C^{0,1}([t,T]$ 
    $\times \R^d, \R^d)$,
    for every $t\in[0,T]$, 
    $x \in \R^d$ let
    $Z^x_t = (Z^x_{t,s})_{s \in (t,T]} 
    \colon (t,T] \times \Omega \to \R^{d+1}$ 
    be an $(\mathbb{F}_s)_{s \in (t,T]}$-adapted
    stochastic process 
    with continuous sample paths
    satisfying that 
    for all $s \in (t,T]$ 
    it holds a.s.\! that
    \begin{equation}
    Z^x_{t,s} = 
    \begin{pmatrix}
    1\\
    \frac{1}{s-t} \int_t^s
    (\sigma(r,   X^x_{t,r}))^{-1} \; \Big(\frac{\partial}{\partial x} X^x_{t,r}\Big) \,\d W_r
    \end{pmatrix},     
    \end{equation} 
    let $f \in C([0,T] \times \R^d \times \R\times \R^{d}, \R)
    \cap L^2([0,T] \times \R^d \times \R\times \R^{d}, \R)$,
    $g \in C(\R^d, \R) \cap L^2(\R^d, \R)$
    be at most polynomially growing,
    and assume for all
    $t \in [0,T]$, $x_1,x_2 \in \R^d$, 
    $a_1,a_2\in\R$,
    $w_1,w_2 \in \R^{d}$  
    that
    $\lvert f(t,x_1,a_1,w_1) 
    - f(t,x_2,a_2,w_2) \rvert 
    \leq L\Vnorm{
        (a_1, w_1)
        -(a_2, w_2)}$.
    Then 
    \begin{enumerate}[label=(\roman*)]
        \item\label{it:vs1_cor2} 
        there exists a unique 
        $v\in C([0,T]\times \R^d,\R)
        \cap C^{0,1}([0,T)\times \R^d,\R)$
        which satisfies that
        $((v, \nabla_x v)(t,x)$
        $\cdot\sqrt{T-t})_{t\in [0,T), x\in\R^d}$
        grows at most polynomially 
        and for all 
        $t\in[0,T)$, $x\in \R^d$
        it holds that
        $\E[\abs{g(X^x_{t,T})} 
        \Vnorm{ Z^x_{t,T}}  
        + \int_t^T  \abs{f(r, X^x_{t,r}, 
            v(r, X^x_{t,r}),
            (\nabla_x v)(r,X^x_{t,r}))}
         \Vnorm{ Z^x_{t,r}}  \,\d r ]
        <\infty$
        and
        \begin{equation}\label{eq:v_cor2}
        \begin{split}
        &(v, \nabla_x v)(t,x) 
        =\E \left[  g(X^x_{t,T}) Z^x_{t,T}  
        + \int_t^T  f(r, X^x_{t,r}, v(r, X^x_{t,r}),
        (\nabla_x v)(r,X^x_{t,r})) 
        Z^x_{t,r}  \,\d r \right], 
        \end{split}
        \end{equation}
        \item\label{it:vs2_cor2}
        there exists a unique viscosity solution 
        $u\in\{\mathbf{u}\in C([0,T]\times \R^d,\R)
        \cap C^{0,1}([0,T)\times \R^d,\R)
        \colon$ $((\mathbf{u},\nabla_x\mathbf{u})(t,x)\sqrt{T-t})_{t\in [0,T), x\in\R^d}$
        $\text{ grows at most polynomially}\} $
        of
        \begin{multline}
        (\tfrac{\partial u}{\partial t})(t,x)
        +\langle \mu(t,x), (\nabla_x u)(t,x)\rangle
        +\tfrac{1}{2}\operatorname{Tr}(\sigma(t,x)[\sigma(t,x)]^*(\operatorname{Hess}_x u)(t,x))
        \\
        +f(t,x,u(t,x), (\nabla_x u)(t,x))
        =0
        \end{multline}
        with $u(T,x)=g(x)$ for
        $(t,x)\in (0,T)\times \R^d$,
        and
        \item\label{it:vs3_cor2} 
        for all $t\in[0,T]$, $x\in \R^d$
        it holds that
        $u(t,x)=v(t,x)$.
    \end{enumerate}   
\end{corollary}

\begin{proof}[Proof of Corollary~\ref{cor:vs2}]
    Throughout this proof let
    $V_q\colon\R^d\to (0,\infty)$,
    $q\in (0,\infty)$,
    satisfy for all $q\in (0,\infty)$,
    $x\in \R^d$ that
    $V_q(x)=(1+\norm{x}^2)^{\frac{q}{2}}$.
    First 
    note that
    \cite[Lemma 3.3]{beck2021existence}
    (applied for every $q\in (0,\infty)$
    with $p\curvearrowleft q$,
    $\mathcal{O}\curvearrowleft \R^d$
    in the notation of
    \cite[Lemma 3.3]{beck2021existence})
    and the assumption that
    for all $t\in [0,T]$, $x\in\R^d$
    it holds that
    $\max \{
    \langle x, \mu(t,x)\rangle,
    \norm{\sigma(t,x)}_F^2
    \}
    \leq c (1+\norm{x}^2)$
    demonstrate that 
    \begin{enumerate}[label=(\Roman*)]
        \item \label{it:V_q_diff}
        for all $q\in (0,\infty)$
        it holds that $V_q\in 
        C^\infty(\R^d, (0,\infty))$ and
        \item \label{it:V_q_ineq}
        for all $q\in (0,\infty)$,
        $t\in [0,T]$, $x\in \R^d$
        it holds that
        \begin{equation}
        \begin{split}
        &\langle \mu(t,x), (\nabla V_q)(x) \rangle
        +\tfrac{1}{2}\operatorname{Tr}(\sigma(t,x)[\sigma(t,x)]^*(\operatorname{Hess} V_q)(x))\\
        &\leq \tfrac{cq}{2}
        \max\{q+1, 3\} V_q(x).
        \end{split}
        \end{equation}
    \end{enumerate}
    Observe that item~\ref{it:V_q_diff}
    and
    the product rule
    imply that 
    for all 
    $q\in (0,\infty)$, 
    $x\in\R^d$
    it holds that
    \begin{equation}
    \label{eq:V_derivatives1}
    \begin{split}
    (\nabla V_q)(x)
    = qx (1+\norm{x}^2)^{\frac{q}{2}-1}
    = q V_q(x) \tfrac{x}{1+\norm{x}^2}.
    \end{split}
    \end{equation}
    The fact that for all $a \in [0,\infty)$
    it holds that
    $a \leq 1+ a^2$
    therefore shows that
    for all $q\in (0,\infty)$,
    $x\in \R^d$
    it holds that
    \begin{equation}
    \begin{split}
    &\norm{(\nabla V_q)(x)}
    = q V_q(x) \tfrac{\norm{x}}{1+ \norm{x}^2}
    \leq q V_q(x). 
    \end{split}
    \end{equation}
    Combining this with
    item~\ref{it:V_q_ineq}
    proves that
    for all 
    $q\in (0,\infty)$,
    $t\in [0,T]$, $x\in O$
    it holds that 
    \begin{equation}\label{eq:V_q_ineq2}
    \begin{split}
    &\langle \mu(t,x), (\nabla V_q)(x)\rangle
    +\tfrac{1}{2} \operatorname{Tr}(\sigma(t,x)[\sigma(t,x)]^* (\operatorname{Hess} V_q)(x))
    \\
    &\qquad
    +L\norm{(\nabla V_q)(x)}
    \leq (\tfrac{cq}{2}
    \max\{q+1, 3\}+Lq) V_q(x).
    \end{split}
    \end{equation}
    Moreover, note that 
    \eqref{eq:V_derivatives1}
    and
    the assumption that
    for all $t\in [0,T]$, $x\in\R^d$
    it holds that
    $\max \{
    \langle x, \mu(t,x)\rangle,
    \norm{\sigma(t,x)}_F^2
    \}
    \leq c (1+\norm{x}^2)$
    demonstrate that
    for all $q\in (0,\infty)$,
    $t\in [0,T]$, $x\in \R^d$
    it holds that
    \begin{equation}
    \begin{split}
    &\frac{\norm{[(\nabla V_q)(x)]^* 
            \sigma(t,x)}^2}{V_q(x)}
    = q^2 \frac{\abs{V_q(x)}^2
        \norm{ x^* \sigma(t,x)}^2}{V_q(x)(1+\norm{x}^2)^2}\\
    &\leq q^2 V_q(x) \frac{
        \norm{x}^2 \norm{\sigma(t,x)}_F^2}{(1+\norm{x}^2)^2}
    \leq c q^2 V_q(x) \frac{
        \norm{x}^2 
        (1+\norm{x}^2)}{(1+\norm{x}^2)^2}
    \leq c q^2 V_q(x) .
    \end{split}
    \end{equation}
    Item~\ref{it:V_q_ineq}
    hence imples that for all
    $q\in (0,\infty)$,
    $t\in [0,T]$, $x\in \R^d$
    it holds that
    \begin{equation}\label{eq:V_q_ineq1}
    \begin{split}
    &\langle\mu(t, x),
    (\nabla V_q)(x) \rangle
    +\tfrac{1}{2}\operatorname{Tr}(\sigma(t, x)[\sigma(t, x)]^*(\operatorname{Hess} V_q)(x))\\
    &\qquad
    +\tfrac{1}{2}\tfrac{\norm{[(\nabla V_q)(x)]^*
            \sigma(t, x)}^2_F}{V_q(x)}\\
    &\leq \tfrac{cq}{2}
    \max\{q+1, 3\} V_q(x)
    + \tfrac{c q^2}{2} V_q(x)
    \leq cq \max\{q+1, 3\} V_q(x).
    \end{split}
    \end{equation}
    Combining this with
    \eqref{eq:V_q_ineq2}
    ensures that
    for all 
    $q\in (0,\infty)$
    there exists $\rho_q\in [0,\infty)$
    which satisfies for all
    $t\in [0,T]$, $x\in \R^d$
    that
    \begin{equation}
    \label{eq:V_q_ineq_rho1}
    \begin{split}
    &\langle \mu(t,x), (\nabla V_q)(x)\rangle
    +\tfrac{1}{2} \operatorname{Tr}(\sigma(t,x)[\sigma(t,x)]^* (\operatorname{Hess} V_q)(x))
    \\
    &\qquad
    +L\norm{(\nabla  V_q)(x)}
    \leq \rho_q V_q(t,x)
    \end{split}
    \end{equation}
    and
    \begin{equation}
    \label{eq:V_q_ineq_rho2}
    \begin{split}
    &\langle \mu(t, x), (\nabla V_q)(x)
    \rangle
    +\tfrac{1}{2}\operatorname{Tr}(\sigma(t, x)[\sigma(t, x)]^*(\operatorname{Hess} V_q)(x))\\
    &\qquad
    +\tfrac{1}{2}\tfrac{\norm{[(\nabla V_q)(x)]^*
            \sigma(t, x)}^2_F}{V_q(x)}
    \leq \rho_q V_q(x).
    \end{split}
    \end{equation}
    Next
    note that
    for all $q\in (0,\infty)$
    it holds that
    \begin{equation}
    \label{eq:V_q_prop}
    \liminf\nolimits_{r\to\infty}
    [\inf\nolimits_{x\in \R^d, \norm{x}>r}
    V_q(x)]=\infty  
    \qquad \text{and} \qquad 
    \inf\nolimits_{x\in \R^d}
    V(x)>0.
    \end{equation}
    Furthermore, observe that
    the assumption that $f$ and $g$
    are at most polynomially growing
    ensures that there exists
    $p\in (0,\infty)$ which satisfies
    \begin{equation}
    \sup_{t\in [0,T]}\sup_{x\in\R^d}
    \bigg(\frac{\abs{g(x)}
        +\abs{f(t,x,0,0)}\sqrt{T-t}}{V_p(x)}\bigg)
    <\infty.
    \end{equation}
    This shows that for all 
    $q\in [p,\infty)$
    it holds that
    \begin{equation}
    \limsup_{r\to\infty}
    \bigg[\sup_{t\in [\max\{T-\frac{1}{r},0\},T)}
    \sup_{x\in \R^d, \norm{x}>r} \bigg(
    \frac{\abs{g(x)+f(t,x,0,0)\sqrt{T-t}}}{V_q(x)}
    \bigg) \bigg]
    =0.
    \end{equation}
    Combining this with
    \eqref{eq:V_q_ineq_rho1},
    \eqref{eq:V_q_ineq_rho2},
    \eqref{eq:V_q_prop},
    and
    item~\ref{it:vs2}
    of Corollary~\ref{cor:vs1}
    (applied with
    $\rho\curvearrowleft \rho_{2p}$,
    $O\curvearrowleft \R^d$,
    $V\curvearrowleft V_{2p}$
    in the notation of 
    Corollary~\ref{cor:vs1})
    demonstrates that
    there exists a unique viscosity solution 
    $u\in\{\mathbf{u}\in C([0,T]\times \R^d,\R)
    \colon \limsup_{r \to \infty}$
    $[
    \sup_{t \in [0,T)\setminus K_r}
    \sup _{x \in \R^d, \norm{x}>r}
    ( \frac{\Vnorm{ (\mathbf{u},\nabla_x\mathbf{u})(t,x)}}{V_{2p}(x)}  
    \, \sqrt{T-t} ) ] = 0\}$ 
    of
    \begin{equation}
    \label{eq:u_is_vs}
    \begin{split}
    &(\tfrac{\partial u}{\partial t})(t,x)
    +\langle \mu(t,x), (\nabla_x u)(t,x) \rangle
    +\tfrac{1}{2}\operatorname{Tr}(\sigma(t,x)[\sigma(t,x)]^*(\operatorname{Hess}_x u)(t,x))
    \\
    &\qquad
    +f(t,x,u(t,x), (\nabla_x u)(t,x))
    =0
    \end{split}
    \end{equation}
    with $u(T,x)=g(x)$ for
    $(t,x)\in (0,T)\times \R^d$.
    Let $v\in C([0,T]\times\R^d,\R)$
    satisfy that 
    $((v,\nabla_x v)(t,x)$
    $\cdot\sqrt{T-t})_{t\in [0,T), x\in\R^d}$
    grows at most polynomially
    and be a viscosity solution of
    \begin{equation}
    \label{eq:v_is_vs}
    \begin{split}
    &(\tfrac{\partial v}{\partial t})(t,x)
    +\langle \mu(t,x), (\nabla_x v)(t,x)\rangle
    +\tfrac{1}{2}\operatorname{Tr}(\sigma(t,x)[\sigma(t,x)]^*(\operatorname{Hess}_x v)(t,x))
    \\
    &\qquad
    +f(t,x,v(t,x), (\nabla_x v)(t,x))
    =0
    \end{split}
    \end{equation}
    with $v(T,x)=g(x)$ for
    $(t,x)\in (0,T)\times \R^d$.
    Observe that
    the assumption that
    $((v, \nabla_x v)(t,x)$
    $\cdot\sqrt{T-t})_{t\in [0,T), x\in\R^d}$
    grows at most polynomially
    ensures that there exists
    $\beta\in [2p,\infty)$
    which satisfies that
    \begin{equation}
    \limsup_{r \to \infty}\bigg[
    \sup_{t \in [0,T)\setminus K_r}
    \sup _{x \in \R^d, \norm{x}>r}
    \bigg( \frac{\Vnorm{ (v, \nabla_x v)
            (t,x)}}{V_\beta(x)}  
    \, \sqrt{T-t} \bigg) \bigg] 
    = 0.
    \end{equation}
    Item~\ref{it:vs2_cor1}
    of Corollary~\ref{cor:vs1}
    (applied with
    $\rho\curvearrowleft \rho_\beta$,
    $O\curvearrowleft \R^d$,
    $V\curvearrowleft V_\beta$
    in the notation of 
    Corollary~\ref{cor:vs1}),
    \eqref{eq:u_is_vs},
    and \eqref{eq:v_is_vs}
    therefore demonstrate that
    $u=v$. This establishes
    item~\ref{it:vs2_cor2}.
    Next observe that item~\ref{it:vs1_cor1}
    of Corollary~\ref{cor:vs1}
    (applied with 
    $\rho\curvearrowleft \rho_{2p}$,
    $O\curvearrowleft \R^d$,
    $V \curvearrowleft V_{2p}$
    in the notation of 
    Corollary~\ref{cor:vs1})
    shows that
    for all $t\in [0,T]$, $x\in \R^d$
    it holds that
    $\E[\abs{g(X^x_{t,T})} 
    \Vnorm{ Z^x_{t,T}} 
    + \int_t^T  \abs{f(r, X^x_{t,r}, 
        u(r, X^x_{t,r}),
        (\nabla_x u)(r,X^x_{t,r}))} 
    \Vnorm{ Z^x_{t,r}}  \,\d r ]
    <\infty$
    and
    \begin{equation}\label{eq:u_is_sfpe_sol}
    \begin{split}
    &(u, \nabla_x u)(t,x) 
    =\E \left[  g(X^x_{t,T}) Z^x_{t,T}  
    + \int_t^T  f(r, X^x_{t,r}, u(r, X^x_{t,r}),
    (\nabla_x u)(r,X^x_{t,r})) 
    Z^x_{t,r}  \,\d r \right].
    \end{split}  
    \end{equation}
    Let $w\in C([0,T]\times\R^d,\R)$
    satisfy 
    that $((w, \nabla_x w)
    (t,x)\sqrt{T-t})_{t\in [0,T), x\in\R^d}$
    grows at most polynomially
    and that for all $t\in [0,T]$,
    $x\in \R^d$ it holds that
    \begin{equation}
       \E[\abs{g(X^x_{t,T})} \Vnorm{ Z^x_{t,T}} 
       + \int_t^T  \abs{f(r, X^x_{t,r}, 
           w(r, X^x_{t,r}),
           (\nabla_x w)(r,X^x_{t,r}))} 
       \Vnorm{ Z^x_{t,r}}  \,\d r ]
       <\infty
    \end{equation} 
    and
    \begin{equation}
    \label{eq:w_is_sfpe_sol}
    \begin{split}
    &(w, \nabla_x w)(t,x) 
    =\E \left[  g(X^x_{t,T}) Z^x_{t,T}  
    + \int_t^T  f(r, X^x_{t,r}, w(r, X^x_{t,r}),
    (\nabla_x w)(r,X^x_{t,r})) 
    Z^x_{t,r}  \,\d r \right].
    \end{split}
    \end{equation}
    Note that
    the fact that
    $((w, \nabla_x w)(t,x)
    \sqrt{T-t})_{t\in [0,T), x\in\R^d}$
    grows at most polynomially
    implies that there exists
    $\gamma\in [\beta,\infty)$
    which satisfies that
    \begin{equation}
    \limsup_{r \to \infty}\bigg[
    \sup_{t \in [0,T)\setminus K_r}
    \sup _{x \in \R^d, \norm{x}>r}
    \bigg( \frac{\Vnorm{ (w, \nabla_x w)
            (t,x)}}{V_\gamma(t,x)}  
    \, \sqrt{T-t} \bigg) \bigg] = 0.
    \end{equation}
    Items~\ref{it:vs1_cor1}
    and \ref{it:vs3_cor1}
    in Corollary~\ref{cor:vs1}
    (applied with $\rho\curvearrowleft \rho_\gamma$,
    $O\curvearrowleft \R^d$,
    $V\curvearrowleft V_\gamma$
    in the notation of
    Corollary~\ref{cor:vs1}),
    \eqref{eq:u_is_sfpe_sol},
    and
    \eqref{eq:w_is_sfpe_sol}
    therefore
    demonstrate that
    $u=v=w$.
    This establishes items~\ref{it:vs1_cor2}
    and \ref{it:vs3_cor2}.
    The proof of Corollary~\ref{cor:vs2}
    is thus complete.
\end{proof}
\subsection*{Acknowledgements}
This work has been funded by the
Deutsche Forschungsgemeinschaft
(DFG, German Research Foundation)
through the research grant
HU1889/6-2.

\bibliographystyle{acm}
\bibliography{bibfile}

\begin{thebibliography}{10}

\bibitem{Antonelli1993}
{\sc Antonelli, F.}
\newblock Backward-forward stochastic differential equations.
\newblock {\em The Annals of Applied Probability 3}, 3 (1993), 777--793.

\bibitem{BarlesBuckdahnPardoux1997}
{\sc Barles, G., Buckdahn, R., and Pardoux, E.}
\newblock Backward stochastic differential equations and integral-partial
  differential equations.
\newblock {\em Stochastics: An International Journal of Probability and
  Stochastic Processes 60}, 1-2 (1997), 57--83.

\bibitem{beck2021existence}
{\sc Beck, C., Gonon, L., Hutzenthaler, M., and Jentzen, A.}
\newblock On existence and uniqueness properties for solutions of stochastic
  fixed point equations.
\newblock {\em Discrete Contin. Dyn. Syst. Ser. B 26}, 9 (2021), 4927--4962.

\bibitem{beck2019overcoming}
{\sc Beck, C., Hornung, F., Hutzenthaler, M., Jentzen, A., and Kruse, T.}
\newblock Overcoming the curse of dimensionality in the numerical approximation
  of {A}llen-{C}ahn partial differential equations via truncated full-history
  recursive multilevel {P}icard approximations.
\newblock {\em arXiv:1907.06729\/} (2019).

\bibitem{beck2021nonlinear}
{\sc Beck, C., Hutzenthaler, M., and Jentzen, A.}
\newblock On nonlinear {F}eynman--{K}ac formulas for viscosity solutions of
  semilinear parabolic partial differential equations.
\newblock {\em Stoch. Dyn. 21}, 8 (2021), Paper No. 2150048, 68.

\bibitem{Bismut1976}
{\sc Bismut, J.-M.}
\newblock Th\'eorie probabiliste du contr\^ole des diffusions.
\newblock {\em Mem. Amer. Math. Soc. 4}, 167 (1976), xiii+130.

\bibitem{Carrell2017}
{\sc Carrell, J.~B.}
\newblock Groups, matrices, and vector spaces : A group theoretic approach to
  linear algebra, 2017.

\bibitem{CrandallEvansLions1984}
{\sc Crandall, M.~G., Evans, L.~C., and Lions, P.-L.}
\newblock Some properties of viscosity solutions of hamilton-jacobi equations.
\newblock {\em Transactions of the American Mathematical Society 282}, 2
  (1984), 487--502.

\bibitem{CrandallIshiiLions1992}
{\sc Crandall, M.~G., Ishii, H., and Lions, P.-L.}
\newblock User's guide to viscosity solutions of second order partial
  differential equations.
\newblock {\em Bull. Amer. Math. Soc. (N.S.) 27}, 1 (1992), 1--67.

\bibitem{CrandallLions1983}
{\sc Crandall, M.~G., and Lions, P.-L.}
\newblock Viscosity solutions of {H}amilton-{J}acobi equations.
\newblock {\em Trans. Amer. Math. Soc. 277}, 1 (1983), 1--42.

\bibitem{EHutzenthalerJentzenKruse2019}
{\sc E, W., Hutzenthaler, M., Jentzen, A., and Kruse, T.}
\newblock On multilevel {P}icard numerical approximations for high-dimensional
  nonlinear parabolic partial differential equations and high-dimensional
  nonlinear backward stochastic differential equations.
\newblock {\em Journal of Scientific Computing 79}, 3 (2019), 1534--1571.

\bibitem{Folland1984}
{\sc Folland, G.~B.}
\newblock {\em Real analysis : modern techniques and their applications}.
\newblock Pure and applied mathematics. Wiley, New York [u.a, 1984.

\bibitem{GihmanSkorohod1972}
{\sc G{\={\i}}hman, {\u{I}}.~{\=I}., and Skorohod, A.~V.}
\newblock {\em Stochastic differential equations}.
\newblock Springer-Verlag, New York, 1972.
\newblock Translated from the Russian by Kenneth Wickwire, Ergebnisse der
  Mathematik und ihrer Grenzgebiete, Band 72.

\bibitem{HairerHutzenthalerJentzen2015}
{\sc Hairer, M., Hutzenthaler, M., and Jentzen, A.}
\newblock Loss of regularity for {K}olmogorov equations.
\newblock {\em Ann. Probab. 43}, 2 (2015), 468--527.

\bibitem{HuPeng1995}
{\sc Hu, Y., and Peng, S.}
\newblock Solution of forward-backward stochastic differential equations.
\newblock {\em Probability Theory and Related Fields 103\/} (1995), 273--283.

\bibitem{HuYong2000}
{\sc Hu, Y., and Yong, J.}
\newblock Forward--backward stochastic differential equations with nonsmooth
  coefficients.
\newblock {\em Stochastic processes and their applications 87}, 1 (2000),
  93--106.

\bibitem{hutzenthaler2021multilevel}
{\sc Hutzenthaler, M., Jentzen, A., Kruse, T., et~al.}
\newblock Multilevel picard iterations for solving smooth semilinear parabolic
  heat equations.
\newblock {\em Partial Differential Equations and Applications 2}, 6 (2021),
  1--31.

\bibitem{hutzenthaler2020overcoming}
{\sc Hutzenthaler, M., Jentzen, A., Kruse, T., Nguyen, T.~A., and von
  Wurstemberger, P.}
\newblock Overcoming the curse of dimensionality in the numerical approximation
  of semilinear parabolic partial differential equations.
\newblock {\em Proc. A. 476}, 2244 (2020), 20190630, 25.

\bibitem{HP2023}
{\sc Hutzenthaler, M., and Pohl, K.}
\newblock On existence and uniqueness properties for solutions of stochastic
  fixed point equations with gradient-dependent nonlinearities, 2023.

\bibitem{ImbertSilvestre2013}
{\sc Imbert, C., and Silvestre, L.}
\newblock An introduction to fully nonlinear parabolic equations.
\newblock {\em An introduction to the K{\"a}hler-Ricci flow\/} (2013), 7--88.

\bibitem{IdRS2018}
{\sc Imkeller, P., dos Reis, G., and Salkeld, W.}
\newblock {Differentiability of SDEs with drifts of super-linear growth}.
\newblock {\em Electronic Journal of Probability 24}, none (2019), 1 -- 43.

\bibitem{KaratzasShreve1991}
{\sc Karatzas, I., and Shreve, S.~E.}
\newblock {\em Brownian motion and stochastic calculus}, second~ed., vol.~113
  of {\em Graduate Texts in Mathematics}.
\newblock Springer-Verlag, New York, 1991.

\bibitem{MaProtterYong1994}
{\sc Ma, J., Protter, P., and Yong, J.}
\newblock Solving forward-backward stochastic differential equations explicitly
  -- a four step scheme.
\newblock {\em Probability Theory and Related Fields 98}, 3 (1994), 339--359.

\bibitem{MaZhang2002}
{\sc Ma, J., and Zhang, J.}
\newblock Representation theorems for backward stochastic differential
  equations.
\newblock {\em The Annals of Applied Probability 12}, 4 (2002), 1390--1418.

\bibitem{Nualart1995}
{\sc Nualart, D.}
\newblock {\em The {M}alliavin calculus and related topics}.
\newblock Probability and its Applications (New York). Springer-Verlag, New
  York, 1995.

\bibitem{PardouxPeng1992}
{\sc Pardoux, {\'E}., and Peng, S.}
\newblock Backward stochastic differential equations and quasilinear parabolic
  partial differential equations.
\newblock In {\em Stochastic partial differential equations and their
  applications ({C}harlotte, {NC}, 1991)}, vol.~176 of {\em Lecture Notes in
  Control and Inform. Sci.} Springer, Berlin, 1992, pp.~200--217.

\bibitem{PardouxPeng1990}
{\sc Pardoux, {\'E}., and Peng, S.~G.}
\newblock Adapted solution of a backward stochastic differential equation.
\newblock {\em Systems Control Lett. 14}, 1 (1990), 55--61.

\bibitem{PardouxRascanu2014}
{\sc Pardoux, E., and Rascanu, A.}
\newblock {\em Stochastic Differential Equations}.
\newblock Springer International Publishing, Cham, 2014, pp.~135--227.

\bibitem{PardouxTang1999}
{\sc Pardoux, E., and Tang, S.}
\newblock Forward-backward stochastic differential equations and quasilinear
  parabolic {PDE}s.
\newblock {\em Probab. Theory Related Fields 114}, 2 (1999), 123--150.

\bibitem{Peng1991}
{\sc Peng, S.~G.}
\newblock Probabilistic interpretation for systems of quasilinear parabolic
  partial differential equations.
\newblock {\em Stochastics Stochastics Rep. 37}, 1-2 (1991), 61--74.

\bibitem{R2009}
{\sc Rudin, W.}
\newblock {\em Principles of Mathematical Analysis}, intern. ed., 3. ed.~ed.
\newblock McGraw-Hill international editions : Mathematics series. McGraw-Hill,
  1976.

\bibitem{Zhang2017}
{\sc Zhang, J.}
\newblock {\em Backward stochastic differential equations}.
\newblock Springer, 2017.

\end{thebibliography}
\end{document}